\documentclass[10pt,leqno]{amsart}
\usepackage{graphicx} 
\usepackage{pgfplots}
\usepackage[backend=biber,style=alphabetic,doi=true,isbn=true, maxalphanames=4, 
  minalphanames=4,maxbibnames=99, url = false]{biblatex}
\usepackage{eucal}

\usepackage[utf8]{inputenc}
\usepackage[english]{babel}
\usepackage{microtype}
\usepackage{mathrsfs}
\usepackage{amsthm}
\usepackage{amsmath, centernot}
\usepackage{amssymb}
\usepackage{indentfirst,csquotes}
\usepackage{amscd}
\usepackage{mathtools}
\usepackage{wasysym}
\usepackage{braket}
\usepackage{hyperref}
\usepackage{appendix}
\usepackage{dsfont}
\usepackage{enumitem}
\usepackage{framed}
\usepackage{caption}
\usepackage{subfig}
\usepackage[all]{xy}
\usepackage{xfrac}
\usepackage{hyperref}
\usepackage{color}
\usepackage{bbm}
\usepackage{stmaryrd}
\usepackage{fancyhdr}
\usepackage{tikz-cd}
\usepackage{graphicx}
\usepackage{nomencl}
\usepackage{ esint }
\usepackage{tikz}
\usepackage{multicol}

\DeclareMathAlphabet{\pazocal}{OMS}{zplm}{m}{n}

\numberwithin{equation}{section}

\allowdisplaybreaks

\newtheorem{teorema}{Theorem}[section]

\newtheorem{prop}[teorema]{Proposition}

\newtheorem{co}[teorema]{Corollary}
\newtheorem{lemma}[teorema]{Lemma}

\newtheorem{oss}[teorema]{Remark}

\newtheorem{ass}[teorema]{Assumption}

\newtheorem{df}[teorema]{Definition}
\newtheorem{quest}[teorema]{Question}

\newcommand{\R}{\mathbb{R}}

\newcommand{\Q}{\mathbb{Q}}
\newcommand{\N}{\mathbb{N}}

\newcommand{\PP}{\mathcal{P}}

\newcommand{\X}{\textit{X}}

\newcommand{\supp}{\operatorname{supp}}

\newcommand{\CE}{\operatorname{CE}}

\newsavebox{\TitleExpr}
\sbox{\TitleExpr}{$(\mathcal{P}_p(\mathbb{R}^d), W_p, Q)$}

\newcommand{\PPpa}{\PP^{\operatorname{pa}}}
\newcommand{\PPN}{\PP^N}
\newcommand{\e}{\mathrm{e}}
\newcommand{\ttm}{\mathtt{m}}
\newcommand{\ttg}{\mathtt{g}}
\newcommand{\tth}{\mathtt{h}}
\newcommand{\ttH}{\mathtt{H}}
\renewcommand{\Q}{Q}

\newcommand{\Ss}{\mathbb{S}}

\DeclareUnicodeCharacter{0308}{\^{}}

\title[The metric measure space of probability measures]{A study of the metric measure space of probability measures via a purely atomic superposition principle}
\author{Alessandro Pinzi}
\date{\today}

\begin{document}

\begin{abstract}
We study the continuity equation on the metric measure space $(\PP_p(X),W_p,\Q)$, when $X$ is either the Euclidean space or a compact, oriented, and boundaryless Riemannian manifold, for some suitable reference measure $\Q \in \PP_p(\PP_p(X))$, which by construction is concentrated over purely atomic measures. In fact, we consider the equation $\partial_tM_t +\operatorname{div}_{\PP}(b_tM_t) = 0$, where $(M_t)_{t\in[0,T]} \subset\PP(\PP(X))$ and $b:[0,T]\times \R^d \times \PP(\R^d) \to \R^d$, assuming that $M_t\ll \Q$ for all $t\in[0,T]$, to then show when the purely atomic property is inherited by the liftings of the curve $M_t$ given by the nested superposition principle in \cite{pinzisavare2025}.

On the Euclidean space, the main assumption is that the $r$-capacity of the diagonal $\Delta \subset \R^d \times \R^d$ is zero with respect to $\nu \otimes \nu$, where $\nu$ is the barycenter of the reference measure $\Q$. We will give sufficient conditions to ensure it, and in particular, thanks to the Nash embedding theorem, this analysis will allow us to extend the main results from the Euclidean space to Riemannian manifolds. 

Finally, we exploit this atomic superposition principle to show the lack of the Sobolev-to-Lipschitz property for functions in $W^{1,p'}(\PP_p(X),W_p,Q)$\footnotemark. Then, we complete the analysis showing that, however, the $L^2$-Wasserstein space endowed with suitable reference measure $Q$, satisfies a Bakry--\'Emery curvature condition.
\end{abstract}

\maketitle

\stepcounter{footnote}
\footnotetext{{Previous versions of this paper, including the published one, contain a minor error, claiming this statement in $W^{1,p}(\PP_p(\R^d), W_p,Q)$. This error does not propagate around the article and is confined to Section \ref{subsec: sob to lip poincare}. The author thanks Mathias Beiglb\"ock for pointing it out.}}

{\small
		.
		\par
	}

{\small\tableofcontents}

 \section{Introduction}
In recent years, the problem of endowing the space of probability measures with a ``nice'' reference measure attracted many researchers. 
It has been mostly studied to understand the closure of the associated Dirichlet form and then its corresponding diffusion process, as done in \cite{dello2022dirichlet, delloschiavo2024massive, schiavo2025hellinger, marx2018new, konarovskyi2019modified, von2009entropic, sturm2024wasserstein, ren2024diffusion, ren2024ornstein}. In this paper, we study the continuity equation, and in general absolutely continuous curves of measures, on the metric measure space $(\PP_p(X),W_p,\Q)$, for suitable $p>1$ and reference measure $\Q\in \PP(\PP(X))$ concentrated over $\PP_p(X)$, when $X$ is either the Euclidean space $\R^d$ or a compact, boundaryless, and oriented Riemannian manifold. To our knowledge, this analysis is the first of his kind and makes it possible to consider a wide range of reference measures which include some of the ones built in the aforementioned articles.

The reference measures we consider are built in the following way. Define 
\begin{equation}\label{inf simplex intro}
\bold{T} := \left\{ \bold{a}=(a_i)_{i\in \N} \in [0,1]^\N \ : \ a_{i+1}\leq a_i, \ \sum_{i=1}^{+\infty} a_i = 1 \right\};
\end{equation} 
consider any probability measure $\pi\in \PP(\bold{T})$. Now, given a Polish space $X$ and an atomless probability measure $\nu\in \PP(X)$; define the embedding function  
\begin{equation}\label{eq: embedding atomic}
\begin{aligned}
    \operatorname{em}:\bold{T} \times X^\infty  \to \PP(X),\quad
    \operatorname{em}(\bold{a}, \bold{x})  := \sum_{i=1}^{+\infty} a_i \delta_{x_i},
\end{aligned}    
\end{equation}
to set $\Q_{\pi,\nu}:= (\operatorname{em}_\# \pi \otimes \nu^{\infty})$. In particular, if $\nu\in \PP_p(X)$, then $\Q_{\pi,\nu}$ is concentrated over $\PP_p(X)$. Then, we denote our class of reference measures as
\begin{equation}\label{eq: ref meas class intro}
    \mathcal{Q}(X):= \{\Q_{\pi,\nu} \ : \ \pi \in \PP(\bold{T}), \ \nu\in \PP(X) \text{ atomless}\}.
\end{equation}

It is important to stress the fact that any $Q \in \mathcal{Q}(X)$ is concentrated over the set of \textit{purely atomic probability measures}, that can be defined as $\PPpa(X):= \operatorname{em}(\bold{T}\times X^\infty)$.

We will mainly work with $X = \R^d$. Then, the main results will be transferred to the Riemannian setting exploiting the Nash isometric embedding theorem. So, we proceed by describing the main results and the strategy of the proofs in the Euclidean setting. The details about the Riemannian case can be found in Section \ref{sec: manifold}.

\bigskip 

\noindent \textbf{Continuity equation and superposition principle.}
For a curve of (narrowly continuous) measures $(\mu_t)_{t\in[0,T]}\subset \PP(\R^d)$ and a given Borel vector field $v:[0,T]\times \R^d \to \R^d$, we say that the \textit{continuity equation} $\partial_t\mu_t + \operatorname{div}(v_t\mu_t) = 0$ is satisfied if 
\begin{equation}\label{eq: CE intro}
    \int_0^T \int_{\R^d} |v_t(x)|d\mu_t(x) dt <+\infty \quad \text{and} \quad \frac{d}{dt}\int_{\R^d}\phi(x) d\mu_t(x) = \int_{\R^d} v_t(x)\cdot \nabla\phi(x) d\mu_t(x),
\end{equation}
in the sense of distributions in $(0,T)$, for all $\phi \in C_c^{1}(\R^d)$. A fundamental result for us is the Ambrosio's superposition principle \cite[Theorem 8.2.1]{ambrosio2005gradient}: for any solution of the continuity equation, there exists $\lambda \in \PP(C([0,T],\R^d))$ which has marginals $\mu_t$, that is
\begin{equation}\label{eq: marginal intro}
    (\e_t)_\sharp \lambda = \mu_t \ \ \forall t\in[0,T], \quad \text{where } \e_t(\gamma):= \gamma(t),
\end{equation}
and that is concentrated over curves $\gamma \in C([0,T],\R^d)$ that satisfies
\begin{equation}\label{eq: ODE intro}
    \gamma\in AC([0,T],\R^d), \quad \text{and} \quad \dot{\gamma}(t) = v_t(\gamma(t)) \quad \text{for a.e. } t\in[0,T]. 
\end{equation}

Returning to the main setting of the present paper, when we endow the space of probability measure with a measure $Q \in \mathcal{Q}(\R^d)$, it will be necessary to work with curves of probability measures $(\mu_t)$ that are purely atomic. In this setting, a natural question about the superposition principle can be raised.

\begin{quest}\label{question}
If a (narrowly continuous) curve $(\mu_t)_{t\in[0,T]} \subset \PPpa(\R^d)$ solves a continuity equation in the sense of \eqref{eq: CE intro}, does it admit a purely atomic superposition of curves $\lambda \in \PPpa(C([0,T],\R^d))$ satisfying \eqref{eq: marginal intro} and is concentrated over $\gamma$'s as in \eqref{eq: ODE intro}?
\end{quest}

The answer to this question is in general false: in Section \ref{section: counterexample}, we build a $W_\infty$-Lipschitz curve of measures in $\PPpa(\R)$, and in particular for which there exists a bounded vector field $v:[0,T]\times \R \to \R$ such that \eqref{eq: CE intro} is satisfied, but it admits a unique lifting $\lambda \in \PP(C([0,T],\R))$ that is atomless. 

This raises non-trivial problems in the study of evolutions in the metric measure space $(\PP_p(\R^d),W_p,Q)$, for all $Q \in \mathcal{Q}(\R^d)$. A fundamental result to overcome this is to ask the continuity equation to be satisfied when tested against a more general class of functions. Notice that, a necessary condition for having a positive answer to the previous question is that the weights of the atoms are constant in time. In view of this, we strengthen \eqref{eq: CE intro} asking that for all $\hat{\phi} = \hat{\phi}(x,r) \in C_c^1(\R^d\times (0,1])$
\begin{equation}\label{eq: CE intro weights}
    \frac{d}{dt}\int_{\R^d} \phi(x,\mu_t(\{x\})) d\mu_t(x) = \int_{\R^d}v_t(x) \cdot \nabla_x \phi(x,\mu_t(\{x\})) d\mu_t(x),
\end{equation}
in the sense of distributions in $(0,T)$. Under this stronger assumption, we can give a positive answer to Question \ref{question}. For a more rigorous discussion, we refer to Lemma \ref{lemma: main} and to Section \ref{subsec: proof}.

\bigskip

\noindent \textbf{Continuity equation for random measures and nested superposition principle.} In \cite{pinzisavare2025}, it has been introduced the continuity equation for general (narrowly continuous) curves of measures $(M_t)_{t\in[0,T]} \subset \PP(\PP(\R^d))$. To recall it, we need to introduce what is a vector field in this setting and which is the natural class of test functions. The latter is given by the \textit{cylinder functions}, see also \cite{fornasier2023density, sodini2023general}.

\begin{df}[Cylinder functions]\label{def: cyl functions}
    A functional $F:\PP(\R^d)\to\R$ is called a $C^1_c$-cylinder function, and we write $F\in \operatorname{Cyl}_c^1(\PP(\R^d))$, if there exists $k\in \N$, $\Phi = (\phi_1,\dots,\phi_k)\in C^1_c(\R^d;\R^k)$ and $\Psi\in C^1_b(\R^k)$ such that 
    \begin{equation}\label{cyl functions}
        F(\mu) = \Psi\left( L_\Phi(\mu) \right), \quad L_\Phi(\mu) =  \big(L_{\phi_1}(\mu) ,\dots , L_{\phi_k}(\mu)\big), \quad  L_{\phi_i}(\mu) := \int_{\R^d} \phi_i(x) d\mu(x).
    \end{equation}
    Its Wasserstein gradient is then defined as 
    \begin{equation}\label{wass gradient cyl}
        \nabla_W F(x,\mu):= \sum_{i=1}^k \partial_{i}\Psi\left( L_{\Phi}(\mu)\right) \nabla \phi_i(x) \quad \forall x\in \R^d, \ \forall\mu \in \PP(\R^d).
    \end{equation} 
\end{df}

Then, it has been proved in \cite[Section 4]{pinzisavare2025} that the natural choice for a vector field in this setting is a \textit{non-local vector field}, that is a Borel map $b:[0,T]\times \R^d\times \PP(\R^d) \to \R^d$. They are also called probability vector fields in literature, see \cite{CSS22, cavagnari2023dissipative}. Then, we say that the continuity equation $\partial_t M_t + \operatorname{div}_\PP(b_tM_t) = 0$ is satisfied if 
\begin{equation}\label{eq: CERM intro}
    \begin{gathered}
        \int_0^T \int_{\PP}\int_{\R^d} |b_t(x,\mu)|d\mu(x) dM_t(\mu) dt <+\infty \quad  \text{and} \\  \forall F \in \operatorname{Cyl}_c^1(\PP(\R^d))\quad 
        \frac{d}{dt} \int_{\PP} F(\mu) dM_t(\mu)  = \int_{\PP}\int_{\R^d} \nabla_W F(x,\mu) \cdot b_t(x,\mu) d\mu(x) dM_t(\mu),
    \end{gathered}   
    \end{equation}
in the sense of distributions in $(0,T)$. One of the main results of \cite{pinzisavare2025} is a version of the Ambrosio's superposition principle for this setting, which can be resumed as follows (for a more rigorous statement, see also Theorem \ref{thm: nested superposition}). Given $(M_t)_{t\in[0,T]} \subset \PP(\PP(\R^d))$ a narrowly continuous curve that solves $\partial_tM_t + \operatorname{div}_\PP(b_tM_t) = 0$, then it can be lifted to:
\begin{itemize}
    \item$\Lambda \in \PP(C([0,T],\PP(\R^d)))$ such that $\Lambda$a.e.~$\boldsymbol{\mu}$ solves 
        \begin{equation}\label{eq: nlce intro}
        \partial_t\mu_t + \operatorname{div}(b_t(\cdot,\mu_t)\mu_t)=0.
        \end{equation}
        \item $\mathfrak{L}\in \PP(\PP(C([0,T],\R^d)))$ such that $\mathfrak{L}$-a.e.~$\lambda\in \PP(C([0,T],\R^d))$
        is concentrated over 
        absolutely continuous curves 
        $\gamma$ that are solutions of
        \begin{equation}\label{eq: sps intro}
        \dot{\gamma}(t) = b(t,\gamma_t,(e_t)_\#\lambda)\quad\text{in }(0,T).
        \end{equation}
\end{itemize}

In this paper, we are interested in the evolution of curves $(M_t)_{t\in[0,T]}\subset \PP(\PP(\R^d))$ such that $M_t \ll Q$ for all $t\in [0,T]$, for some fixed $Q \in \mathcal{Q}(\R^d)$. Consequently, any lifting $\Lambda \in \PP(C([0,T],\PP(\R^d)))$ will be automatically concentrated on curves of purely atomic measures. Then, as for Question \ref{question}, it is natural to ask when there exists a lifting $\mathfrak{L}\in \PP(\PP(C([0,T],\R^d)))$ that satisfies the above properties and it is further concentrated over $\PPpa(C([0,T],\R^d))$. 

Again, this is a non-trivial question, and to positively answer to it we need further assumptions. In particular, assume that $Q = Q_{\pi,\nu}$, then we need to control that the diagonal in $\R^d\times \R^d$ is not too loaded by the measure $\nu\otimes \nu$, in a capacitary sense: to this aim, we introduce, for all $r\in[1,+\infty)$, the quantity 
\begin{equation*}
    \begin{aligned}
         \operatorname{cap}_{r,\nu}(\Delta) :=\inf \bigg\{ \int_{\R^{2d}} |h(x,y)|^r +  |\nabla h(x,y)|^r d\nu\otimes\nu(x,y) : 
         h\in C^{1}_b(\R^{2d}), \ h= 1 \text{ on }\Delta, \  h \leq 1 \bigg\}.
    \end{aligned}
\end{equation*}

Then, one of the main theorems of this paper is the following, where we use the notations
\begin{equation}\label{eq: maps for marginals}
    \begin{aligned}
        & E:\PP(C([0,T],\R^d)) \to C([0,T],\PP(\R^d)), \quad E(\lambda):= ((\e_t)_\sharp \lambda)_{t\in [0,T]},
        \\
        & E_t:\PP(C([0,T],\R^d)) \to \PP(\R^d), \quad E_t(\lambda) := (\e_t)_\sharp \lambda,
        \\
        & \mathfrak{e}_t: C([0,T],\PP(\R^d)) \to \PP(\R^d), \quad \mathfrak{e}_t(\boldsymbol{\mu}):= \mu_t.
    \end{aligned}
\end{equation}

\begin{teorema}\label{main theorem intro}
    Let $\Lambda \in \PP(C([0,T],\PP(\R^d)))$ be concentrated over solutions of \eqref{eq: nlce intro} and define $\boldsymbol{M} = (M_t)_{t\in[0,T]} \in C([0,T],\PP(\PP(\R^d)))$ as $M_t = (\mathfrak{e}_t)_\sharp \Lambda$. Assume:
    \begin{enumerate}
        \item $M_t = u_t \Q$, with $u \in L^1_t(L^q(\Q))$, with $q\in[1,+\infty]$;
        \item for some $p>1$, it holds
        \begin{equation}
            \int\int_0^T\int_{\PP(\R^d)} \int_{\R^d} |b_t(x,\mu)|^p d\mu(x) dM_t(\mu) dt<+\infty;
        \end{equation}
        \item $\operatorname{cap}_{r,\nu}(\Delta) = 0$, with $r\in [1,+\infty)$;
        \item $\frac{p'}{r}+\frac{1}{q}\leq1$, where $p'$ is the conjugate exponent of $p$.
    \end{enumerate}
    Then, there exists $\mathfrak{L}\in \PP(\PP(AC^p([0,T],\R^d)))$ such that $E_\sharp \mathfrak{L} = \Lambda$, it is concentrated over $\lambda \in \PPpa(C([0,T],\R^d))$ that are, in turn, concentrated over solutions of \eqref{eq: sps intro}.
    In particular, $(E_t)_\sharp \mathfrak{L} = M_t$ for all $t\in[0,T]$ and $\Lambda$ is concentrated over $C([0,T],\PP^{\operatorname{pa}}(\R^d))$.
\end{teorema}

The proof of the previous theorem relies on extending the class of functions that we can use to test the continuity equation $\partial_tM_t + \operatorname{div}_\PP(b_tM_t) = 0$ and $\partial_t\mu_t + \operatorname{div}(b_t(\cdot,\mu_t)\mu_t) = 0$, for $\Lambda$-a.e. $(\mu_t)_{t\in[0,T]}$. Indeed, under the assumptions of Theorem \ref{main theorem intro}, we can show that we can test them against the set of \textit{generalized cylinder functions}, already introduced in \cite{dello2022dirichlet,delloschiavo2024massive}.

\begin{df}[Generalized cylinder functions]\label{def: GC functions}
    Given $\hat{\phi}\in C_c^1(\R^d\times (0,1])$, define
    \[\hat{L}_{\hat{\phi}}(\mu) := \int_{\R^d} \phi(x,\mu[x]) d\mu(x), \quad \mu[x]:= \mu(\{x\}).\]
    Then the set of cylinder functions $\operatorname{GC}_c^1(\PP(\R^d))$ is made of functions $\hat{F}:\PP(\R^d)\to \R$ such that 
    \begin{equation}\label{gen cyl functions}
        \hat{F}(\mu) = \Psi\big(\hat{L}_{\hat{\Phi}}(\mu)\big), \qquad \hat{L}_{\hat{\Phi}}(\mu) = \left( \hat{L}_{\hat{\phi}_1}(\mu),\dots, \hat{L}_{\hat{\phi}_k}(\mu) \right)
    \end{equation}
    for some $k\in \N$, $\Psi\in C^1(\R^k)$, $\hat{\Phi} = (\hat{\phi}_1,\dots,\hat{\phi}_k) \in C^1_c(\R^d\times(0,1];\R^k)$. 
    The Wasserstein differential of a generalized cylinder function is defined by 
    \begin{equation}
        \nabla_W \hat{F}(x,\mu) := \sum_{i=1}^k \partial_i \Psi\big(\hat{L}_{\hat{\Phi}}(\mu)\big) \nabla_x \phi_i(x,\mu[x]). 
    \end{equation}
    In the case $\Psi \in C^\infty(\R^k)$ and $\hat{\phi}_1,\dots,\hat{\phi}_k \in C_c^\infty(\R^d \times (0,1])$, we say that $\hat{F}\in \operatorname{GC}_c^\infty(\PP(\R^d))$.
\end{df}

The capacity assumption is needed to suitably approximate any generalized cylinder function with a sequence of cylinder functions, see Theorem \ref{superposition in duality with gen cyl functions}, whose proof has been inspired by similar approximation results in \cite{dello2022dirichlet, delloschiavo2024massive}. This approximation will be the ingredient that will allow us to use them as test functions for the continuity equations, so that now we can rely on a positive answer to Question \eqref{question} for $\Lambda$-a.e. $(\mu_t)_{t\in [0,T]}$ to conclude the proof. The details are left for Section \ref{subsec: nested atomic}.

\bigskip

\textbf{Metric counterpart and the Riemannian case.} A metric version of the atomic nested superposition principle, Theorem \ref{main theorem intro}, can be proved as well. We mean that, instead of working with a curve of measures $(M_t)_{t\in[0,T]} \subset \PP(\PP(\R^d))$ that solves some continuity equation, we assume that it is $p$-absolutely continuous with respect to the $p$-Wasserstein on Wasserstein metric. The rigorous statement is given in Theorem \ref{main theorem metric}. This result relies on the strong link between curves of measures that solve a continuity equation and absolutely continuous curves, as shown in \cite{ambrosio2005gradient, stepanov2017three, pinzisavare2025}.

In Section \ref{section: capacity}, the capacity assumption in Theorem \ref{main theorem intro} will be extensively studied. In particular, we will see how the exponent $r$ is related to the geometry of the support of the measure $\nu$. This analysis, will allow us to work with $\nu$ that are supported on an embedded compact Riemannian manifold without boundary, and in particular, the Nash isometric embedding theorem will give also an intrinsic version of Theorem \ref{main theorem intro}, when the ambient space $\R^d$ is replaced with a compact, oriented, and boundaryless Riemannian manifold.

\bigskip

\textbf{Application to the metric measure Wasserstein space.} Assume now that $X$ is either $\R^d$ or a compact, oriented, and boundaryless Riemannian manifold, and fix $Q = Q_{\pi,\nu} \in \mathcal{Q}(X)$, such that $\nu\in \PP_p(X)$, for $p\in(1,+\infty)$. We use a metric-measure viewpoint to introduce the spaces of Sobolev functions over $\PP_p(X)$, and in particular we will deal with $W^{1,p'}(\PP_p(X),W_p,Q)$. In particular, the definition through test plans fits particularly well in our setting, and it can be used to show that, if $p$ is large enough, the Sobolev-to-Lipschitz property is not satisfied, that is: there exists $F \in W^{1,p'}(\PP_p(X),W_p,Q)$ with bounded gradient $|DF|\leq L$ which does not admit a $L$-Lipschitz representative. Our counterexample shows actually something stronger: in fact, we build a function $F \in W^{1,p'}(\PP_p(X),W_p,Q)$ that has null gradient $Q$-a.e., but that cannot admit a constant representative. In particular, this counterexample shows the lack of the $p'$-Poincar\'e inequality as well: for all $N\in \N$ it holds
\begin{equation}\label{eq: poincare intro}
    \int_{\PP_p(X)} \bigg| F(\mu) - \int_{\PP_p(X)} F dQ_{\pi,\nu} \bigg|^{p'} dQ_{\pi,\nu}(\mu) > N \int_{\PP_p(X)} |DF|^{p'}(\mu) dQ_{\pi,\nu}.
\end{equation}

The previous construction works under very general assumptions:
\begin{itemize}
    \item if $X = \R^d$, with $d\geq 2$, and $\nu \ll \mathcal{L}^d$ with bounded density, then the previous counterexample applies for $p\geq \frac{d}{d-1}$, see Theorem \ref{thm: euclidean case} and Proposition \ref{prop: sharp capacity in R^d};
    \item in the Riemannian case, say that $k\geq 2$ is the dimension of the manifold and that $\nu$ is absolutely continuous with respect to its volume measure, with bounded density. Then, the previous counterexample works for $p\geq \frac{k}{k-1}$, see Theorem \ref{thm: manifold case}.
\end{itemize}

In the previous assumptions, we can see that the case $p=2$ is covered, implying that the metric measure space $(\PP_2(X),W_2,Q)$ cannot be a $\operatorname{RCD}$ space, whose definition is recalled in Definition \ref{def: rcd}. However, thanks to \cite{fornasier2023density}, it is always an infinitesimally Hilbertian space. Moreover, only for the Riemannian setting, in Section \ref{subsec: BE} we show that it also shares curvature property in the Bakry--\'Emery sense, if the base space $(X,d,\nu)$ does, see Theorem \ref{thm: BE wass}. A similar study was done in \cite{schiavo2023wasserstein}, and in \cite{schiavo2021configuration, schiavo2022configuration} for configuration spaces. The proof of the Bakry--\'Emery property relies on the characterization of the diffusion process associated with $(\PP_2(X),W_2,Q)$ (and its heat semigroup) proved in \cite{delloschiavo2024massive}, that has been proved only in the compact manifold case. If such a characterization holds in the Euclidean setting as well, then the Bakry--\'Emery property can be recovered in this setting as well. For more details on this discussion, see Section \ref{sec: sobolev and curvature}.

It is important to stress that the validity of the solely Bakry--\'Emery property is, in general, not strong enough. For example, under the ergodicity assumption of the reference measure $Q$, it can be shown that the Bakry--\'Emery condition gives the validity of a Poincar\'e inequality, see \cite{bakry2013analysis, chafai2024logarithmic}. In our case this cannot hold by absence of ergodicity, as already observed in \cite[Theorem 1.4]{delloschiavo2024massive}.

\bigskip

\textbf{Plan of the paper.} In \textbf{Section \ref{sec: prel}}, we give some preliminaries, in particular we fix the notation about the space of probability measures, together with its natural narrow topology and the Wasserstein distance. We also recall the atomic topology presented in \cite{ethier1994convergence}, that will be useful throughout the paper. Then, we recall the standard and the nested superposition principle, and we define the notion of $(p,\nu)$-capacity of the diagonal $\Delta\subset \R^d\times \R^d$. 

In \textbf{Section \ref{section: counterexample}}, we build the example of a $W_\infty$-Lipschitz curve of purely atomic measures that does not admit a purely atomic lifting. 

In \textbf{Section \ref{section atomic}}, we move to the proof of our main results, the atomic nested superposition principles, in the Euclidean setting. In particular, we introduce the set of generalized cylinder functions and we show how the capacity assumption allows to use them to test the continuity. 

Then, in \textbf{Section \ref{section: capacity}} we make an analysis on the capacity assumption (see also Appendix \ref{app: capacity}).

In \textbf{Section \ref{sec: manifold}}, we study the same problem when the underlying space is a boundaryless, oriented, and compact Riemannian manifold. The key tool is the Nash isometric embedding theorem, that allows to move back and forth all the objects between the manifold and the Euclidean space in which it is isometrically embedded. 

Finally, in \textbf{Section \ref{sec: sobolev and curvature}}, we introduce general metric measure spaces, and in particular the definition of Sobolev space through test plans, and the synthetic notion of Ricci curvature bounded from below known as Bakry--\'Emery condition. We then prove the failure of the Sobolev-to-Lipschitz property and the Poincar\'e inequality as a consequence of the atomic nested superposition principles (both in the Euclidean and Riemannian setting), and we conclude the analysis of the Wasserstein metric measure space showing the inheritance of the Bakry--\'Emery condition from the underlying space (in the Riemannian setting).

\bigskip

\noindent \textbf{Acknowledgments.} The author warmly thanks Lorenzo Dello Schiavo, Giuseppe Savar\'e and Giacomo Enrico Sodini for the fruitful discussions on this topic. The author sincerely thanks the anonymous reviewer for their careful and thorough reading of our manuscript, and for their comments which contributed to improve the clarity and overall presentation of the paper.
\bigskip

\noindent \textbf{Declaration on funding.} This research did not receive any specific grant from funding agencies in the public, commercial, or not-for-profit sectors.
\section{Preliminaries}\label{sec: prel}
Here is a list of the main notation used in the following:\smallskip
\halign{$#$\hspace{0.7cm}\ & #\hfil
\cr
C_b(X), \, (C_b(X;\R^n) )&continuous and bounded functions from $X$ to $\R$ (resp. $\R^n$)
\cr 
\Delta, \, \Delta_\varepsilon &resp., the diagonal and the $\varepsilon$-strip around the diagonal of $\R^d\times \R^d$
\cr
\mathcal{B}(X) & Borel sets of $X$
\cr
\PP(X) & the space of Borel probability measures over a Polish space $X$
\cr
\PPpa(X) & the space of purely atomic probability measures
\cr
\PP(B), \,\text{for }B\in \mathcal{B}(X) & probability measures in $\PP(X)$ concentrated over $B$
\cr
\PP(\PP(X)) & the space of random measures over a Polish space $X$
\cr
g_\sharp, \, g_{\sharp\sharp} & resp., the push forward and the double push-forward (see \eqref{eq: double push forward})
\cr
\widetilde{M} & see \eqref{eq: widetilde M}
\cr
W_{p,d} & the $p$-Wasserstein distance built on the metric space $(X,d)$
\cr
C([0,T],X) & continuous curves in $X$
\cr
AC^p([0,T],X) & absolutely continuous curves in $(X,d)$
\cr
\e_t,\, \mathfrak{e}_t, \, E_t  & see §\ref{subsub: curves}
\cr
}

\vspace{0.5cm}

\noindent Let $(X,\tau)$ be a Polish space. We will denote by $\mathcal{B}(X)$ its Borel $\sigma$-algebra.
\subsection{Probability measures, narrow topology and Wasserstein distance}
\subsubsection{Space of probability measures}
Let $\mathcal{P}(X)$ be the space of the Borel probability measures over $X$. We endow it with the narrow topology, i.e. the coarsest one for which 
\begin{equation}
    \mathcal{P}(X) \ni \mu \mapsto \int_{X} \phi d\mu \ \text{ is continuous for all } \phi\in C_b(X),
\end{equation}
where $C_b(X)$ denotes the space of real-valued, bounded, and continuous functions. 

If not otherwise specified, the space $\PP(X)$ is always considered endowed with the narrow topology, denoted $\tau_{\mathtt{n}}$. In particular, we denote by $C([0,T],\PP(X))$ the space of curves of probability measures that are narrowly continuous. For any $\mu_n,\mu \in \PP(X)$, we write $\mu_n \to \mu$ if the convergence holds with respect to the narrow topology.

Given a measurable function $f:Z_1 \to Z_2$, where $(Z_i,\mathcal{S}_i)$ are general measurable spaces, and a measure $\mu\in \PP(Z_1)$, we denote with $f_\sharp\mu \in \PP_+(Z_2)$ the push-forward measure, defined as 
\[f_\sharp\mu(S) := \mu(f^{-1}(S)) \quad \forall S\in \mathcal{S}_2.\]

The same construction works for the set of Borel non-negative measures $\mathcal{M}_+(X)$, that is a Polish space when endowed with the narrow topology.

We will often denote as $\PP(B)$, for a given set $B\in \mathcal{B}(X)$, the probability measures $\mu \in \PP(X)$ for which $\mu(B) = 1$.

\subsubsection{Space of (laws of) random probability measures} The resulting space $\PP(X)$ endowed with narrow topology is still a Polish space, so we can reiterate the construction and endow $\PP(\PP(X))$ with the so-called narrow over narrow topology. As above, with $C([0,T],\PP(\PP(X)))$ we denote the space of continuous curves of random probability measures.

Given $M\in \PP(\PP(X))$ we define the \textit{barycenter measure} $\operatorname{bar}[M]\in \PP(X)$ via 
\begin{equation}\label{eq: bar}
    \int_{X} f(x) d\operatorname{bar}[M](x) = \int_{\PP(X)}\int_X f(x) d\mu(x) dM(\mu) \ \ \  \text{for all Borel functions }f:X \to [0,1].
\end{equation}

Moreover, it will be useful the iterated push-forward of a random measure induced by a function $g:X \to Y$, where $X$ and $Y$ are two Polish spaces: given a random measure $M\in \PP(\PP(X)),$ we denote 
\begin{equation}\label{eq: double push forward}
g_{\sharp\sharp}M:= G_{\sharp}M, \quad G:\PP(X) \to \PP(Y), \ G(\mu) := g_\sharp\mu.
\end{equation}
It is well defined due to \cite[Appendix D]{pinzisavare2025}, where the Borel measurability of $G$ is shown. Moreover, for all $M\in \PP(\PP(X))$, we will denote by $\widetilde{M} \in \PP(X\times \PP(X))$ the measure defined through
\begin{equation}\label{eq: widetilde M}
    \int H(x,\mu) d\widetilde{M}(x,\mu) = \int \int H(x,\mu) d\mu(x) dM(\mu) \text{ for all measurable }H:X\times \PP(X) \to [0,1]. 
\end{equation}

\subsubsection{Wasserstein metric} Fix $p\in [1,+\infty]$, and a complete and separable metric space $(X,d)$. As usual, with $\PP_p(X)$ we denote the space of probability measures $\mu$ for which 
\begin{equation}
    \|d(\cdot,x^*)\|_{L^p(\mu)}<+\infty \ \text{ for some (and then all) } x^* \in X.
\end{equation}
Then, the $p$-Wasserstein distance between $\mu,\nu\in\PP_p(X)$ is defined as 
\[W_{p,d}(\mu,\nu) := \min \left\{ \|d(\cdot,\cdot)\|_{L^p(\X\times X,\pi)} \ : \ \pi \in \Gamma(\mu,\nu) \right\},\]
\[\Gamma(\mu,\nu):= \Big\{\pi \in \PP(X\times X) \ : \ \pi(A\times X) = \mu(A), \quad \pi(X\times B)  = \nu(B) \quad \forall A,B \in \mathcal{B}(X)\Big\}.\]
It is well-known \cite[Proposition 7.1.5]{ambrosio2005gradient} that $(\PP_p(X),W_{p,d})$ is complete and separable, and for $p\in[1,+\infty)$ the convergence induced is 
\[W_{p,d}(\mu_n,\mu) \to 0 \iff \mu_n \to \mu \quad \text{and} \quad \int_X d^p(x,x^*) d\mu_n(x) \to \int_X d^p(x,x^*) d\mu(x).\]
In particular, if the distance $d$ is bounded, $\PP_p(X) = \PP(X)$ and $W_{p,d}$ metrizes the narrow topology.
As for the narrow over narrow topology, we can reiterate the construction to build the \textit{Wasserstein on Wasserstein} distance $\mathcal{W}_p:= W_{p,W_p}$ over $\PP_p(\PP_p(X))$.

\subsubsection{Purely atomic measures and atomic topology}\label{subsec: atomic topology}
Let $(X,\tau)$ be a Polish space. We introduce some notations used in \cite{dello2022dirichlet, delloschiavo2024massive}. To introduce the set of purely atomic probability measures, denote the infinite simplex $\bold{T}\subset [0,1]^\N$ as in \eqref{inf simplex intro}, and its subset $\bold{T}_0$ as
\begin{equation}\label{eq: infinite simplex}
\begin{aligned}
    \bold{T}_0 := \left\{ \bold{a} \in \bold{T} \ : \ a_{i+1}<a_i \ \forall i \in \N  \right\}.
\end{aligned}
\end{equation}
Moreover, denote by $X^\infty_{\neq} \subset X^\N$ the set of sequences $(x_n)_{n\in \N}$ such that $x_i\neq x_j$ for all $i\neq j$. Both $\bold{T}$ and $X^\N$ are endowed with the infinite product topology, and in particular $\bold{T}_0$ is endowed with the subspace topology. 
We then consider the following spaces of purely atomic measures:
\begin{equation}
    \begin{aligned}
        \PP^{\operatorname{pa}}(X) & := \left\{\mu\in \PP(X) \ : \exists(z_i)_{i\in \N} \in X^\infty_{\neq}, \ \exists(a_i)_{i\in \N} \in \mathbf{T}, \text{ s.t. } \mu = \sum_i a_i \delta_{z_i}\right\},
    \\
        \PP_0^{\operatorname{pa}}(X) & := \left\{\mu\in \PP(X) \ : \ \exists(z_i)_{i\in \N} \in X^\infty_{\neq}, \ \exists(a_i)_{i\in \N} \in \mathbf{T}_0, \text{ s.t. } \mu = \sum_i a_i \delta_{z_i}\right\},
        \end{aligned}
    \end{equation}

We will make use of a useful topology on the space of probability measures, called `atomic topology'. We recall here the definition and its main properties, for which we refer to \cite[§2]{ethier1994convergence}. Given $\mu\in \PP(X)$, we denote 
\begin{equation}
    \mu^*:= \sum_{x\in X} \mu(\{x\})^2 \delta_x \in \mathcal{M}_+(X).
\end{equation}

\begin{df}[Atomic convergence]
    Let $\mu_n,\mu\in \mathcal{P}(X)$. We say that $\mu_n \overset{\tau_a}{\longrightarrow} \mu$ if $\mu_n\to\mu$ narrowly and $\mu_n^*(X) \to \mu^*(X)$.
\end{df}

\begin{lemma}
    Let $d$ be any distance that induces $\tau$ and such that $(X,d)$ is complete. Let $\Psi:[0,+\infty)\to [0,1]$ be a continuous and non-increasing function. Then, define the distance $D_{p,d,\Psi}$ over $\PP_p(X,d)$ as
    \begin{equation}\label{eq: atomic wass distance}
        D_{p,d,\Psi}(\mu,\nu) := W_{p,d}(\mu,\nu) + \sup_{\varepsilon\in(0,1)}\left| \int_{X\times X} \Psi\left( \frac{d(x,y)}{\varepsilon} \right) d\big( \mu\otimes \mu - \nu\otimes \nu\big)(x,y)\right|.   
    \end{equation}
    Then $D_{p,d,\Psi}(\mu_n,\mu) \to 0 $ if and only if $W_{p,d}(\mu_n,\mu) \to 0$ and $\mu_n^* \to \mu^*$. In particular, if $d$ is bounded, $D_{p,d,\Psi}$ induces the atomic convergence and the induced topology $\tau_{\mathtt{a}}$, called atomic topology, is Polish. In particular, the Borel sets of $\PP(X)$ built with respect to this topology, coincide with the ones built using the narrow topology.
\end{lemma}

\begin{proof}
    The first part is a consequence of Lemma 2.1, Lemma 2.2 and Lemma 2.3 from \cite{ethier1994convergence}. The second part follows from \cite{schwariz1973radon} and the fact that $\tau_a$ is stronger than $\tau_n$.
\end{proof}

The following lemma shows how the atomic convergence actually forces convergence of the atoms and their weights. For the proof, we refer to \cite[Lemma 2.5]{ethier1994convergence}.

\begin{lemma}\label{lemma: atomic topology}
    Assume that $\mu_n \to \mu$ narrowly. Let $\{a_i^{(n)}, x_i^{(n)}\}_{i\in \N}$ and $\{a_i ,x_i\}_{i\in \N}$ be, respectively, the set of atoms of $\mu_n$ and $\mu$, with locations $(x_i^{(n)})_{i\in \N},(x_i)_{i\in\N}\in X^\infty_{\neq}$ and ordered weights $(a_i^{(n)})_{i\in \N}, (a_i)_{i\in \N} \in \bold{T}$. Then 
    \begin{enumerate}
        \item $\mu_n \overset{\tau_a}{\longrightarrow} \mu$ if and only if $a_i^{(n)} \to a_i$ for all $i\in \N$.
    Moreover, if $a_k > a_{k+1}$ for some $k\in \N$, then the set of locations $\{x_1^{(n)},\dots,x_k^{(n)}\} \to \{x_1,\dots,x_k\}$;
    \item if $\mu\in \PPpa(X)$, then $\mu_n \overset{\tau_a}{\longrightarrow} \mu$ if and only if $\sum_i |a_i^{(n)} -a_i| \to 0$.
    \end{enumerate}
\end{lemma}

The next result can be found in \cite{delloschiavo2024massive}. Anyway, we give an original proof for completeness. 

\begin{lemma}\label{lemma: pa measures are measurable}
    The sets $\PP^{pa}(X)$ and $\PP^{pa}_0(X)$
    are Borel subsets of $\PP(X)$.
\end{lemma}

\begin{proof}
    The set $\PP^{\operatorname{pa}}(X)$ is closed in the atomic topology, and in particular it is a Borel subset of $\PP(X)$. 
    Regarding $\PP^{\operatorname{pa}}_0(X)$, for any $j\in \N$, define 
    \(\PP^{\operatorname{pa}}_{0,j}(X) := \left\{ \mu\in \PP^{\operatorname{pa}}(X)\ : \ a_j > a_{j+1} \right\} 
    \)
    which is open because its complement $\{\mu \in \PP^{\operatorname{pa}}(X) \ : \ a_j = a_{j+1}\}$ is closed thanks to Lemma \ref{lemma: atomic topology}. Then, simply notice that 
    \(\PP^{\operatorname{pa}}_0(X) = \bigcap_{j\in \N} \PP^{\operatorname{pa}}_{0,j}(X).\)
\end{proof}

A last simple but useful lemma is stated here.

\begin{lemma}\label{lemma: push for of PA meas}
    Let $X$ and $Y$ be Polish spaces and $f:X \to Y$ a Borel map. If $\mu\in \PPpa(X)$, then $f_\sharp \mu \in \PPpa(Y)$.
\end{lemma}
\begin{proof}
    A simple calculation shows that, writing $\mu = \sum_{i} a_i \delta_{x_i}$, for all $G:Y \to [0,1]$ Borel function
    \[\int_Y G(y)d(f_\sharp \mu)(y) = \sum_{i} a_i G(f(x_i)) = \int_Y G(y) d\bigg( \sum_i a_i \delta_{f(x_i)} \bigg),\]
    giving that $f_\sharp\mu = \sum_i a_i \delta_{f(x_i)} \in \PPpa(Y)$.
\end{proof}

\subsection{Curves of probability measures and probability over curves}

\subsubsection{Probability measures over curves}\label{subsub: curves}
Let $(X,\tau)$ be a Polish space. The space $C([0,T],X)$ of continuous curves in $X$, endowed with the compact-open topology, is a Polish space. The function $\e_t: C([0,T],X) \to X$ corresponds to the evaluation at a given time $t\in[0,T]$. 
\\
The previous setting applies to the spaces $C([0,T],\PP(X))$ and $\PP(C([0,T],X))$, so that we can consider probability measures over them, that we will usually denote as $\Lambda \in \PP(C([0,T],\PP(X)))$ and $\mathfrak{L}\in \PP(\PP(C([0,T],X)))$.

Since these last two spaces will play a fundamental role for the rest of the paper, we fix some additional notation regarding them:
\begin{itemize}
    \item the evaluation map from $C([0,T],\PP(X))$ to $\PP(X)$ at a fixed time $t\in [0,T]$, will be denoted as $\mathfrak{e}_t$;
    \item the push-forward of the evaluation map $\e_t$ will be denoted as $E_t:= (\e_t)_{\sharp}$. In particular, $(E_t)_\sharp = (\e_t)_{\sharp\sharp}$
\end{itemize}

\subsubsection{Absolutely continuous curves} Let $(X,d)$ be a complete and separable metric space. We say that a continuous curve $\boldsymbol{x}:[0,T] \to X$ is absolutely continuous, and we write $\boldsymbol{x}\in 
   AC([0,T],X)$, if it exists a function $g\in L^1(0,T)$ such that 
   \begin{equation}\label{abs cont def}
        d(\boldsymbol{x}(t),\boldsymbol{x}(s)) \leq \int_s^t g(r) dr \quad \text{whenever}\quad 0\leq s \leq t \leq T.
   \end{equation}
   If $g\in L^p(0,T)$, for some $p\in (1,+\infty]$, we say that $\boldsymbol{x}\in 
   AC^p([0,T],X)$. The space $AC^p([0,T],X)$, for $p\in[1,+\infty]$ is a Borel subset of $C([0,T],X)$.

It is well known \cite{ambrosio2005gradient} that absolutely continuous curves admit a \textit{metric derivative}, that is,
\begin{equation}
        \lim_{s\to t} \frac{d(
        \boldsymbol{x}(s),
        \boldsymbol{x}(t))}{|t-s|} =:|\dot{\boldsymbol{x}}|_{d}
        \kern-2pt(t)
    \end{equation}
    exists for $\mathcal{L}^1$-a.e. $t\in [0,T]$ 
    and it 
    provides the smallest $g$ (in an almost everywhere pointwise sense) that fits in \eqref{abs cont def}. 

    In particular, we can apply this definition to the Wasserstein space $(\PP_p(X),W_p)$ and the Wasserstein on Wasserstein space $(\PP_p(\PP_p(X)),\mathcal{W}_p)$. In this paper we will deal with the spaces:
    \begin{itemize}
        \item $AC^p([0,T],\PP_p(X))$ that is a Borel subset of $C([0,
        T],\PP(X))$ for all $p\geq 1$;
        \item $AC^p([0,T],\PP_p(\PP_p(X)))$ that is a Borel subset of $C([0,
        T],\PP(\PP(X)))$ for all $p\geq 1$.
    \end{itemize}

\subsubsection{Continuity equation for curve of measures}
Given a Borel measurable vector field $v:[0,T]\times \R^d \to \R^d$ and a curve of probability measures $\boldsymbol{\mu} = (\mu_t)_{t\in[0,T]} \in C([0,T],\PP(\R^d))$, we say that $\partial_t \mu_t + \operatorname{div}(v_t\mu_t) = 0$ if 
    \begin{equation}\label{eq: CE}
    \begin{aligned}
        \int_0^T \int_{\R^d} |v_t(x)|d\mu_t(x) dt <+\infty \quad \text{and} \quad \text{for all }\xi\in C_c^1(0,T), \ \phi \in C_c^1(\R^d)
        \\
        \int_0^T \xi'(t) \int_{\R^d} \phi(x) d\mu_t(x) dt = - \int_0^T \xi(t) \int_{\R^d} \nabla \phi(x) \cdot v(t,x) d\mu_t(x) dt.
    \end{aligned}   
    \end{equation}
We recall some important properties of curves that are solution of a continuity equation: they characterize absolutely continuous curve and they can be represented as superposition of solution to the associated ordinary differential equation. 

\begin{teorema}\cite[Theorem 8.3.1]{ambrosio2005gradient}
    Let $\boldsymbol{\mu} = (\mu_t)_{t\in[0,T]}\in C([0,T],\PP(\R^d)) $ and assume that $\mu_0 \in \PP_p(\R^d)$. If $p\in(1,+\infty)$, the following are equivalent:
    \begin{enumerate}
        \item $\boldsymbol{\mu}\in AC^p([0,T],\PP_p(\R^d))$;
        \item there exists $v:[0,T]\times \R^d \to \R^d$ such that 
        \begin{equation}
            \int_0^T \int_{\R^d}|v(t,x)|^pd\mu_t(x) dt <+\infty \quad \text{and} \quad \partial_t\mu_t + \operatorname{div}(v_t\mu_t) = 0.
        \end{equation}
    \end{enumerate}
\end{teorema}

\begin{teorema}\label{thm: superposition}\cite[Theorem 8.2.1]{ambrosio2005gradient}
    Let $v:[0,T]\times \R^d \to \R^d$ and $\boldsymbol{\mu} = (\mu_t)_{t\in[0,T]} \in C([0,T],\PP(\R^d))$ be satisfying \eqref{eq: CE}. Then, there exists a probability measure $\PP(C([0,T],\R^d))$ such that $(\e_t)_\sharp\lambda = \mu_t$ for all $t\in[0,T]$ and $\lambda$ is concentrated over curves $\gamma\in AC([0,T],\R^d)$ that satisfies 
    \begin{equation}\label{eq: ODE}
        \dot{\gamma}(t) = v_t(\gamma(t)) \quad \text{for a.e. }t\in[0,T].
    \end{equation}
    Vice versa, if $\lambda\in \PP(C([0,T],\R^d))$ is concentrated over absolutely continuous curves that satisfy \eqref{eq: ODE} and $\int \int_0^T |\dot\gamma(t)| dt d\lambda(\gamma)<+\infty$, then $\mu_t := (\e_t)_\sharp\lambda$ satisfies \eqref{eq: CE}.
\end{teorema}

As a consequence of the previous two theorems, we recover the lifting result by S. Lisini \cite{lisini2007characterization} in the Euclidean space.

\begin{teorema}
     Let $\boldsymbol{\mu}=(\mu_t)_{t\in [0,T]} \in AC^p([0,T],\PP_p(\R^d))$ with $p\in(1,+\infty)$. Then, there exists a lifting $\lambda\in \PP(C([0,T],\R^d))$ such that  $(\e_t)_{\#}\lambda = \mu_t$ for all $t\in [0,T]$ and $\lambda$ is concentrated over $AC^p([0,T],\R^d)$. Moreover,
        \begin{equation}\label{minimal lifting}
            \int \int_0^T |\dot\gamma(t)|^p dt d\lambda (\gamma) = \int_0^T |\dot{\boldsymbol{\mu}}|^p_{W_p} < +\infty.
        \end{equation}
    On the other hand, for any $p\in [1,+\infty)$, given $\lambda\in \PP(C_T(\R^d))$ concentrated over absolutely continuous curves, with $(e_0)_\#\lambda\in \PP_p(\R^d)$, the curve $\boldsymbol{\mu}:= ((e_t)_\#\lambda)_{t\in [0,T]}$ belongs to $AC^p_T(\PP_p(\R^d))$.
\end{teorema}

For simplicity, we presented this structure in the space $\R^d$, but it can be recovered in more general metric spaces as well, as shown in \cite{stepanov2017three}.

\subsubsection{Continuity equation for random measures} In \cite{pinzisavare2025}, the authors showed that a similar structure holds at the level of random measures.

As already said in the introduction, here the role of the leading vector field is played by a non-local Borel measurable vector field $b:[0,T]\times \R^d \times \PP(\R^d) \to \R^d$, while the class of `smooth' functions to test the continuity equation is the one of cylinder functions, defined in \ref{def: cyl functions}. Now, let $\boldsymbol{M} = (M_t)_{t\in[0,T]}\in C([0,T],\PP(\PP(\R^d)))$ be a continuous curve of random probability measures: we say that $\partial_t M_t + \operatorname{div}_\PP(b_tM_t) = 0$ if 
\begin{equation}\label{eq: CERM}
    \begin{aligned}
        \int_0^T \int_{\PP}\int_{\R^d} |b_t(x,\mu)|d\mu(x) dM_t(\mu) dt <+\infty \  \text{and} \  \text{for all }\xi\in C_c^1(0,T), \ F \in \operatorname{Cyl}_c^1(\PP(\R^d))
        \\
        \int_0^T \xi'(t) \int_{\PP} F(\mu) dM_t(\mu) dt = - \int_0^T \xi(t) \int_{\PP}\int_{\R^d} \nabla_W F(x,\mu) \cdot b_t(x,\mu) d\mu(x) dM_t(\mu) dt.
    \end{aligned}   
\end{equation}
    
Then, we recall some of the main results presented in \cite{pinzisavare2025} (see also \cite{pinzifirst2025} for more general integrability conditions and for stochastic evolutions).

\begin{teorema}
    Let $\boldsymbol{M} = (M_t)_{t\in[0,T]}\in C([0,T],\PP(\PP(\R^d))) $ and assume that $M_0 \in \PP_p(\PP_p(\R^d))$. If $p>1$, the following are equivalent:
    \begin{enumerate}
        \item $\boldsymbol{M}\in AC^p([0,T],\PP_p(\PP_p(\R^d)))$;
        \item there exists $b:[0,T]\times \R^d \times \PP(\R^d) \to \R^d$ Borel measurable such that 
        \begin{equation}
            \int_0^T \int_\PP \int_{\R^d}|b_t(x,\mu)|^pd\mu(x) dM_t(\mu) dt <+\infty \quad \text{and} \quad \partial_tM_t + \operatorname{div}_\PP(b_tM_t) = 0.
        \end{equation}
    \end{enumerate}
\end{teorema}

\begin{teorema}\label{thm: nested superposition}
    Let $b:[0,T]\times \R^d \times \PP(\R^d)\to \R^d$ be Borel measurable and $\boldsymbol{M} = (M_t)_{t\in[0,T]} \in C([0,T],\PP(\PP(\R^d)))$ be satisfying \eqref{eq: CERM}. Then, there exists $\Lambda\in \PP(C([0,T],\PP(\R^d)))$ and $\mathfrak{L}\in \PP(\PP(C([0,T],\R^d)))$ such that $(\mathfrak{e}_t)_{\#}\Lambda = M_t$ and $(E_t)_\#\mathfrak{L}=M_t$ for any $t\in[0,T]$, and:
    \begin{enumerate}
        \item\label{prop of Lambda R^d}  $\Lambda$-a.e.~$\boldsymbol{\mu}$ belongs to $
        AC([0,T],\PP(\R^d))$ 
        and solves 
        $$\partial_t\mu_t + \operatorname{div}(b_t(\cdot,\mu_t)\mu_t)=0
        \quad\text{as in \eqref{eq: CE}}.$$
        \item\label{properties of L R^d} $\mathfrak{L}$-a.e.~$\lambda\in \PP(C([0,T],\R^d))$
        is concentrated over 
        absolutely continuous curves 
        $\gamma$ that are solutions of
        $$\dot{\gamma}(t) = b(t,\gamma_t,(e_t)_\#\lambda)\quad\text{in }(0,T).$$
    \end{enumerate}
    Vice versa, if $\Lambda$ satisfies (1) and $\int \int_0^T \int|b_t(x,\mu_t)|d\mu_t(x) dt d\Lambda(\boldsymbol{\mu}) <+\infty$, then $M_t := (\mathfrak{e}_t)_\sharp \Lambda$ satisfies \eqref{eq: CERM}. Similarly, if $\mathfrak{L}$ satisfies (2) and $\int\int \int_0^T |b_t(\gamma(t),(\e_t)_\sharp\lambda)|dt d\lambda(\gamma) d\mathfrak{L}(\gamma) <+\infty$, then $M_t := (E_t)_\sharp\mathfrak{L}$ satisfies \eqref{eq: CERM}.
\end{teorema}

Similarly to the classic continuity equation, a metric result can be recovered from the last two theorems, but we will not need it in the following.

Given a Borel non-local vector field $b:[0,T]\times \R^d \times \PP(\R^d) \to \R^d$, we recall also the definition of the sets $\operatorname{CE}(b) \subset C([0,T],\PP(\R^d))$ and $\operatorname{SPS}(b) \subset \PP(C([0,T],\R^d))$: 
\begin{equation}
\begin{aligned}
    \CE(b):= \bigg\{(\mu_t)_{t\in[0,T]} \in C([0,T],\PP(\R^d)) \ : \ 
    \int \int |b_t(x,\mu_t)|d\mu_t(x) dt <+\infty, & \\
    \partial_t\mu_t + \operatorname{div}(b_t(\cdot,\mu_t)\mu_t)=0& \bigg\},
\end{aligned}
\end{equation}

\begin{equation}\label{SPSb}
\begin{aligned}
    \operatorname{SPS}(b):= \{\lambda \in & \PP(C([0,T],\R^d)) \ : \ \int \int |b_t(\gamma(t),(e_t)_\#\lambda)| d\lambda(\gamma) dt<+\infty ,
    \\
    &\lambda \big( \operatorname{AC}([0,T],\R^d) \big) =1, \ \dot{\gamma}(t) = b_t(\gamma_t,(e_t)_\#\lambda) \ \mathcal{L}^1_T\otimes\lambda\text{-a.e.} \}.
\end{aligned}
\end{equation}

As shown in \cite[Section 5]{pinzisavare2025}, they are Borel measurable. Moreover, in Theorem \ref{thm: nested superposition}, (1) and (2) can be summarized, resp., as $\Lambda$ concentrated on $\operatorname{CE}(b)$ and $\mathfrak{L}$ concentrated on $\operatorname{SPS}(b)$.

\subsection{The family \texorpdfstring{$\mathcal{Q}$}{} of reference random atomic measures}\label{subsec: reference}
Let $(X,\tau)$ be a Polish space. We denote $\boldsymbol{X}:= X^\infty_{\neq}$.
Recall from \eqref{eq: embedding atomic} and \eqref{eq: ref meas class intro}, respectively, the definition of the map $\operatorname{em}$ and of the class of reference measures $\mathcal{Q}$. Then, we have the following results.

\begin{lemma}
    Endowing $\bold{T}\times \boldsymbol{X}$ with the product topology, the map $\operatorname{em}$ is continuous. Moreover, the image of $\operatorname{em}$ is $\PPpa(X)$.
\end{lemma}

\begin{proof}
    It is clearly continuous considering the atomic topology over $\PP(X)$, and thus Borel.
\end{proof}

\begin{lemma}\label{lemma: barycenter of mu times mu}
    Let $\Q_{\pi,\nu}\in \mathcal{Q}$. The following hold:
    \begin{enumerate}
        \item $\operatorname{bar}[\Q_{\pi,\nu}] =  \nu$;
        \item let 
        \[c_1 := \int_{\bold{T}} \bigg(\sum_{i\neq j} a_ia_j\bigg) d\pi(\bold{a}), \quad c_2:= \int_{\bold{T}} \bigg(\sum_{i=1}^{+\infty} a_i^2 \bigg) d\pi(\bold{a}).\]
        Then for all $g:X\times X \to [0,1]$ Borel measurable it holds
        \begin{equation}
            \int_{\PP(X)}\int_{X \times X} g(x,y) d\mu\otimes \mu(x,y) d\\Q_{\pi,\nu}(\mu) = c_1\int_{ X \times  X }g(x,y)d\nu\otimes\nu(x,y) + c_2\int_{ X } g(x,x) d\nu(x).
        \end{equation}
        In other words, given $C:\PP( X )\to\PP( X \times  X )$ to be $C(\mu):= \mu\otimes \mu$, then $\operatorname{bar}[C_\#\Q_{\pi,\nu}] = c_1\nu\otimes \nu + c_2 (\operatorname{id,\operatorname{id}})_\#\nu$.
    \end{enumerate}
\end{lemma}

\begin{proof}
    For all $f: X  \to [0,1]$ Borel measurable, we have
    \begin{align*}
        & \int_{\PP( X )}\int_{ X } f(x) d\mu(x) d\Q_{\pi,\nu}(\mu) = \int_{\bold{T}} \left(\int_{\boldsymbol{X}}\sum_{i=1}^{+\infty} a_i f(x_i) d\nu^{\infty}(\bold{x})\right)d\pi(\bold{a})
        \\
        & \ =  
        \int_{\bold{T}} \left(\sum_{i=1}^{+\infty} a_i \int_{\boldsymbol{X}}f(x_i) d\nu^{\infty}(\bold{x})\right)d\pi(\bold{a})
        = 
        \int_{\bold{T}} \left(\sum_{i=1}^{+\infty} a_i \int_{ X }f(x) d\nu(x)\right)d\pi(\bold{a})
        =  \int_{ X } f(x)d\nu(x).
    \end{align*}
    Regarding the second part 
    \begin{align*}
        & \int_{\PP( X )}\int_{X^2} g(x,y) d\mu\otimes \mu(x,y) d\Q_{\pi,\nu}(\mu) 
        = 
        \int_{\bold{T}} \int_{\boldsymbol{X}} \bigg( \sum_{i,j\geq 1} a_ia_j g(x_i,x_j) \bigg) d\nu^{\infty}(\bold{x}) d\pi(\bold{a})
        \\
        & \ =
        \int_{\bold{T}} \bigg(\sum_{i\neq j} a_ia_j \int_{\boldsymbol{X}} g(x_i,x_j) d\nu^\infty(\bold{x}) + \sum_{i\geq 1} a_i^2\int_{\boldsymbol{X}} g(x_i,x_i) d\nu^\infty(\bold{x}) \bigg)d\pi(\bold{a})
        \\
        & \ =
        \int_{\bold{T}} \bigg(\sum_{i\neq j} a_ia_j \int_{X^2} g(x,y) d\nu\otimes\nu(x,y) \bigg) d\pi(\bold{a}) + \int_{\bold{T}} \bigg(\sum_{i\geq 1} a_i^2\int_{ X } g(x,x) d\nu(x) \bigg) d\pi(\bold{a})
        \\
        & \ =
        c_1 \int_{X^2} g(x,y) d\nu\otimes\nu(x,y) + c_2 \int_{ X } g(x,x) d\nu(x). 
        \qedhere
    \end{align*}
\end{proof}

    We have many examples of measures $\Q_{\pi,\nu}\in \mathcal{Q}$ coming from Bayesian non-parametric statistical theory. One is the Dirichlet-Ferguson measure $D_{\beta\nu}$ (see \cite{ferguson1973bayesian, dello2022dirichlet}), where $\beta>0$: the measure $\pi\in \PP(\bold{T})$ is determined by the stick-breaking procedure \cite{sethuraman1994constructive}, and it is concentrated over $\bold{T}_0$, so that $D_{\beta \nu}$ is concentrated over $\PP^{\operatorname{pa}}_0( X )$.
    Another example is the Poisson measure $P_{\lambda,\nu}$, with $\lambda>0$: a sample $\mu \sim P_\nu$ is built sampling $N\in \N$ according to a Poisson distribution with parameter $\lambda$ and $\bold{x} \sim \nu^\infty$ independently, then $\mu= \frac{1}{N+1} \sum_{i=1}^{N+1} \delta_{x_i}$.

\subsection{Capacity of the diagonal with respect to \texorpdfstring{$\nu\otimes\nu$}{}}\label{subsec: capacity}
The definitions in this subsection are specific to $\R^d$. However, in Section \ref{sec: manifold} we will exploit it to bring on a similar study on compact Riemannian manifolds.

Given an atomless probability measure $\nu\in \PP(\R^d)$, we will make use of the concept of \textit{capacity} with respect to the product measure $\nu\otimes \nu$. In particular, given $r\geq 1$, we need the definition of $(r,\nu)$-capacity of the diagonal set $\Delta:= \{(x,x) \ : \ x\in \R^d\} \subset \R^d \times \R^d$, that we define as
    \begin{equation}\label{capacity definition}
    \begin{aligned}
         \operatorname{cap}_{r,\nu}(\Delta) :=\inf \bigg\{ & \int_{\R^{2d}} |h(x,y)|^r +  |\nabla_{x,y} h(x,y)|^r d\nu\otimes\nu(x,y) \ : \\
         & h\in C^{1}_b(\R^{2d}), \ h= 1 \text{ on }\Delta, \  h \leq 1 \bigg\}.
    \end{aligned}
    \end{equation}

A fundamental assumption for our main theorem will be that $\operatorname{cap}_{r,\nu} (\Delta)= 0$ for some $r\geq 1$. We will extensively study this assumption in Section \ref{section: capacity}, in which we will find sufficient conditions on $\nu$ and $r$ for which that holds.

Before proceeding, we want to say that a more common definition for the capacity would be 
\begin{equation}
\widetilde{\operatorname{cap}}_{r,\nu} (\Delta) :=\inf \bigg\{  \int_{\R^{2d}} |h|^r +  |\nabla_{x,y} h|^r d\nu\otimes\nu \ : h\in C^{1}_b(\R^{2d}), \ h\geq 1 \text{ on }\Delta \bigg\}.
\end{equation}
In Appendix \ref{app: capacity}, we show that the two definitions coincide, at least when $\nu\ll \mathcal{L}^d$.

\section{Continuity equation for atomic measures}\label{section: counterexample}
\subsection{An empirical superposition principle and the tangent space of purely atomic measures}
In this section, we will deal with narrowly continuous curves of atomic probability measures, i.e. $\boldsymbol{\mu} = (\mu_t)_{t\in[0,T]} \in C([0,T],\PP(\R^d))$ such that $\mu_t \in \PPpa(\R^d)$ for all $t\in[0,T]$.  

First, we recall a superposition principle for measures with a fixed finite number of atoms, shown in \cite[Theorem C.1]{cavagnari2022lagrangian}, that is an empirical version of Theorem \ref{thm: superposition}. We denote 
\[\PP^N(\R^d):= \left\{\mu\in \PP(\R^d) \ : \ \exists x_1,\dots,x_n \in \R^d \text{ s.t. }\mu = \frac{1}{N}\sum_{i=1}^N \delta_{x_i}\right\} \quad \forall N\geq1.\]

\begin{teorema}\label{thm: atomic superposition}
    Let $(\mu_t)_{t\in[0,T]} \in C([0,T],\PP(\R^d))$ be such that there exists $N \geq 1$ for which $\mu_t \in \PPN(\R^d)$ for all $t\in[0,T]$. Assume that there exists a Borel vector field $v:[0,T]\times \R^d \to \R^d$ and 
    \begin{equation}\label{eq:2}
    \int_0^T \int_{\R^d}|v(t,x)| d\mu_t(x)dt<+\infty \quad \text{and} \quad \partial_t \mu_t + \operatorname{div}(v_t\mu_t) = 0.
    \end{equation}
    Then there exists $\lambda \in \PPN(C([0,T],\R^d))$ such that $(\e_t)_\sharp\lambda = \mu_t$ for all $t\in[0,T]$ and such that 
    \begin{equation}\label{eq:3}
    \lambda = \frac1N \sum_{i=1}^N \delta_{\gamma_i} \quad \text{and} \quad \dot\gamma_i(t) = v(t,\gamma_i(t)) \ \text{for a.e. }t\in [0,T].
    \end{equation}
    Moreover, the vector field $v$ for which \eqref{eq:2} holds, is uniquely determined $\mu_t\otimes dt$-a.e.
\end{teorema}

Thanks to \cite[Proposition 8.4.5]{ambrosio2005gradient}, the previous theorem can be used to show that for all $\mu\in \PPN(\R^d)$, the tangent space $\operatorname{Tan}_\mu \PP_p(\R^d)$ is the whole $L^p(\mu;\R^d)$, for all $p>1$.
In the next proposition, we prove it for general $\mu\in \PPpa(\R^d)$ directly from the definition of the tangent space.

\begin{lemma}\label{lemma: tangent}
    For any $\mu\in \PPpa(\R^d)$ and $p>1$, we have 
    \[\operatorname{CoTan}_\mu \PP_p(\R^d) := \operatorname{Clos}_{L^{p'}(\mu;\R^d)} \left\{ \nabla f \ : \ f \in C_c^{\infty}(\R^d \right\} = L^{p'}(\mu;\R^d).\]
    In particular, $\operatorname{Tan}_\mu \PP_p(\R^d):= j_{p'}(\operatorname{CoTan}_\mu \PP_p(\R^d)) = L^p(\mu;\R^d)$, where $j_{p'}(H) = |H|^{p'-2}H$ for all $H\in L^{p'}(\mu;\R^d)$
\end{lemma}

\begin{proof}
    It suffices to show that if $H\in L^p(\mu;\R^d)$ satisfies
    \begin{equation}\label{eq:1}
        \int_{\R^d} H(x)\cdot \nabla f(x) d\mu(x) = 0 \quad \forall f \in C_c^\infty(\R^d),
    \end{equation}
    then $H = 0$ $\mu$-a.e. To fix the notation, let $\mu = \sum_{i\geq1} a_i \delta_{x_i}$ with $(a_i)_{i\in \N} \in \bold{T}$ and $(x_i)_{i\in \N} \in (\R^d)^{\infty}_{\neq}$. Let $\rho\in C_c^\infty(\R^d)$ be a bump function, that is, such that $\operatorname{supp}\rho \subset B(0,2)$, $0\leq\rho \leq1$ and $\rho = 1$ on $B(0,1)$. Then, for all $i,m\in \N$ and $v\in \R^d$ let 
    \(f_{m,v,i}(x):= \rho(m(x-x_i)) v \cdot (x-x_i) \in C_c^\infty(\R^d).\)
    Its gradient is given by 
    \(\nabla f_{m,v,i}(x) = \rho(m(x-x_i))v + mv \cdot(x-x_i) \nabla \rho(m(x-x_i)),\)
    so that 
    $\nabla f_{m,v,i}(x_i) = v$ and
    \[|\nabla f_{m,v,i}(x_j)| \leq (|v| + m|v||x_j-x_i|\|\nabla\rho\|_\infty)\mathds{1}_{B(0,\frac2m)}(x_j-x_i) \leq |v|(1+ 2\|\nabla\rho\|_\infty)\mathds{1}_{B(0,\frac2m)}(x_j-x_i). \]
    This shows that $\nabla f_{m,v,i}(x) \to v \mathds{1}_{\{x_i\}}(x)$ as $m\to +\infty$ for $\mu$-a.e. $x\in \R^d$, dominated by the right hand side in the previous computation, so that 
    \[0 = \int_{\R^d} H(x)\nabla f_{m,v,i}(x) d\mu(x)\longrightarrow H(x_i)\cdot v \quad \text{as }m\to +\infty\]
    This shows that $H(x_i) = 0$ for all $i\in \N$ and concludes the proof.
\end{proof}

\subsection{Liftings of Lipschitz curves of purely atomic measures are not necessarily purely atomic measures of curves}
At this point, one is tempted to say that a version of Theorem \ref{thm: atomic superposition} holds substituting $\PPN(\R^d)$ with $\PPpa(\R^d)$. Unfortunately, this is not true, and in this section we build a counterexample for it. In fact, we build a curve of purely atomic probability measures that is Lipschitz with respect to $W_\infty$, but for which a purely atomic lifting does not exist.

We proceed in the following way: consider the infinite product space $\Omega := \{0,1\}^\mathbb{N}$, endowed with the Borel $\sigma$-algebra generated by the product topology, and with the probability measure $\mathbb{P}:= (\frac12\delta_0 + \frac12\delta_1)^{\infty}$. Notice that $\mathbb{P}$ is an atomless probability measure. Consider the sequence of times $\{t_k\}_{k\geq 0}$ given by
\[t_0 = 0, \quad t_{k+1} = t_k + \frac{1}{2^{k+1}} \quad \forall k\geq 0\]For any $\boldsymbol{\omega} = (\omega_0,\omega_1,\dots)\in \Omega$ let $\Tilde{\gamma}\in C([0,1],\R)$ be defined as 
\begin{equation}\label{eq:4}
    \Tilde{\gamma}_{\boldsymbol{\omega}}(t):= \sum_{k\geq 0} \omega_k\big( (t-t_k)_+ \wedge (t_{k+1}-t_k) \big).
\end{equation}

\begin{center}
\begin{tikzpicture}[scale=4.5]

\draw[->] (0,0) -- (1.05,0) node[right] {};
\draw[-] (0,0) -- (0,1.05) node[above] {};

\color{red}
\draw[->] (0,1) -- (0.25,0.75);
\draw[->] (0,1) -- (0.125,0.875);
\draw[->] (0,1) -- (0.375,0.625);
\draw[->] (0.5,0.5) -- (0.5,0.423);
\draw[->] (0.5,0.5) -- (0.5,0.337);
\draw[->] (0.5,0.25) -- (0.5625,0.1875);
\draw[->] (0.625,0.125) -- (0.625,0.07);
\normalcolor
\foreach \x/\y in { 
    0/1, 
    0/0.5,
    0/0.25, 0.25/0.25, 0.5/0.25, 0.75/0.25,
    0.125/0.125, 0.375/0.125, 0.625/0.125, 0.875/0.125, 
    0.5/0.5, 0/0.125, 0.25/0.125, 0.5/0.125, 0.75/0.125,0/0.0625,
    0.125/0.125-0.0625,
    0.25/0.125-0.0625,
    0.375/0.125-0.0625,
    0.5/0.125-0.0625,
    0.625/0.125-0.0625,
    0.75/0.125-0.0625,
    0.875/0.125-0.0625,
    0+0.0625/0.0625,
    0.125+0.0625/0.125-0.0625,
    0.25+0.0625/0.125-0.0625,
    0.375+0.0625/0.125-0.0625,
    0.5+0.0625/0.125-0.0625,
    0.625+0.0625/0.125-0.0625,
    0.75+0.0625/0.125-0.0625,
    0.875+0.0625/0.125-0.0625
}{
    \fill (\x,\y) circle[radius=0.2pt];
}

\draw[thick] (0,1) -- (0,0.5);
\draw[thick] (0,1) -- (0.5,0.5);

\draw[thick] (0,0.5) -- (0,0.25);
\draw[thick] (0,0.5) -- (0.25,0.25);
\draw[thick] (0.5,0.5) -- (0.5,0.25);
\draw[thick] (0.5,0.5) -- (0.75,0.25);

\draw[thick] (0,0.25) -- (0.125,0.125);
\draw[thick] (0.25,0.25) -- (0.375,0.125);
\draw[thick] (0.5,0.25) -- (0.625,0.125);
\draw[thick] (0.75,0.25) -- (0.875,0.125);
\draw[thick] (0,0.25) -- (0,0.125);
\draw[thick] (0.25,0.25) -- (0.25,0.125);
\draw[thick] (0.5,0.25) -- (0.5,0.125);
\draw[thick] (0.75,0.25) -- (0.75,0.125);

\draw[thick] (0,0.125) -- (0,0.0625);
\draw[thick] (0.125,0.125) -- (0.125,0.125-0.0625);
\draw[thick] (0.25,0.125) -- (0.25,0.125-0.0625);
\draw[thick] (0.375,0.125) -- (0.375,0.125-0.0625);
\draw[thick] (0.5,0.125) -- (0.5,0.125-0.0625);
\draw[thick] (0.625,0.125) -- (0.625,0.125-0.0625);
\draw[thick] (0.75,0.125) -- (0.75,0.125-0.0625);
\draw[thick] (0.875,0.125) -- (0.875,0.125-0.0625);
\draw[thick] (0,0.125) -- (0+0.0625,0.0625);
\draw[thick] (0.125,0.125) -- (0.125+0.0625,0.125-0.0625);
\draw[thick] (0.25,0.125) -- (0.25+0.0625,0.125-0.0625);
\draw[thick] (0.375,0.125) -- (0.375+0.0625,0.125-0.0625);
\draw[thick] (0.5,0.125) -- (0.5+0.0625,0.125-0.0625);
\draw[thick] (0.625,0.125) -- (0.625+0.0625,0.125-0.0625);
\draw[thick] (0.75,0.125) -- (0.75+0.0625,0.125-0.0625);
\draw[thick] (0.875,0.125) -- (0.875+0.0625,0.125-0.0625);

\node[below] at (1,0) {$x=1$};
\node[left] at (0,0) {$t=1$};
\node[left] at (0,1) {$t_0=0$};
\node[left] at (0,0.5) {$t_1=\frac{1}{2}$};
\node[left] at (0,0.25) {$t_2=\frac{3}{4}$};
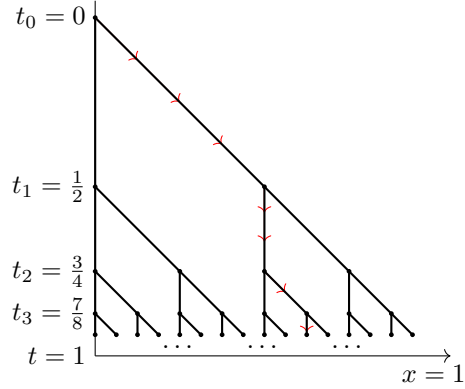
\node[left] at (0,0.125) {$t_3=\frac{7}{8}$};

\node at (0.25, 0.0275) {$\dots$};
\node at (0.5, 0.0275) {$\dots$};
\node at (0.75, 0.0275) {$\dots$};

\end{tikzpicture}

\end{center}

The construction of these curves is summarized in the figure above, where the branches happen at the time $t_k$ and the choice between left or right depends on the value of $\omega_k$. In red, you can see the start of a curve corresponding to a sequence $\boldsymbol{\omega}$ with $\omega_0 = 1$, $\omega_1 = 0$, $\omega_2 = 1$ and $\omega_3 = 0$. It is clear that the map $\tilde{\iota}: \Omega \to C([0,1],\R)$ defined as $\tilde{\iota}(\boldsymbol{\omega}) = \tilde{\gamma}_{\boldsymbol{\omega}}$, is injective.

Then, define $\tilde{\lambda}:= \tilde{\iota}_\#\mathbb{P} \in \PP(C([0,1],\R))$ and $\tilde{\mu}_t:= (\e_t)_\sharp \tilde{\lambda}$. Notice that $\tilde{\mu}_t\in \PPpa(\R)$ for any $t\in[0,1)$: in fact, for any $t\in[0,1)$, there exists $n\in \N$ such that $t \in [t_n,t_{n+1}]$, and thanks to \eqref{eq:4} we have 
\[\tilde{\mu}_t = \frac{1}{2^{n+1}} \sum_{(\omega_0,\dots,\omega_n) \in \{0,1\}^{n+1}} \delta_{x_{\omega_0,\dots,\omega_n}(t)}, \quad \text{where }x_{\omega_0,\dots,\omega_n}(t) = \sum_{k=0}^{n-1} \frac{\omega_k}{2^{k+1}} + \omega_n(t-t_n).\]
Moreover, $(\mu_t)_{t\in[0,1]} \in \operatorname{Lip}\big([0,1],(\PP([0,1]),W_\infty)\big)$: indeed, for any $t_n\leq s<t \leq t_{n+1}$, the transport plan
\[\pi_{s,t}:=\frac{1}{2^{n+1}} \sum_{(\omega_0,\dots,\omega_n) \in \{0,1\}^{n+1}} \delta_{(x_{\omega_0,\dots,\omega_n}(t), x_{\omega_0,\dots,\omega_n}(s))} \]
is a coupling between $\mu_t$ and $\mu_s$, and with a simple computation we can see that 
\[W_{\infty}(\mu_s,\mu_t) \leq \|x-y\|_{L^\infty(\pi)} = (t-s).\]
Then, exploiting the triangular inequality and the continuity of the curve at the time $t=1$, we achieve Lipschitzianity in the whole interval $[0,1]$.

At this point, the only thing we need to fix is the fact that $\tilde{\mu}_1$ is atomless, since $\tilde{\iota}$ is injective (it gives actually the uniform distribution of $[0,1]$). To do this, we consider:
\begin{equation}
    \begin{aligned}
        \iota:\Omega \to \PP(C([0,1],\R)), \quad \iota(\boldsymbol{\omega}) = \gamma_{\boldsymbol{\omega}} \ \text{where } \gamma_{\boldsymbol{\omega}}(t) = (1-t) \tilde{\gamma}_{\boldsymbol{\omega}}(t),\quad 
        \lambda = \iota_\# \mathbb{P}, \quad \mu_t = (\e_t)_\#\lambda.
    \end{aligned}
\end{equation}
The distortion induced  by the term $(1-t)$ can be seen in the picture below (not on scale for the sake of visualization).

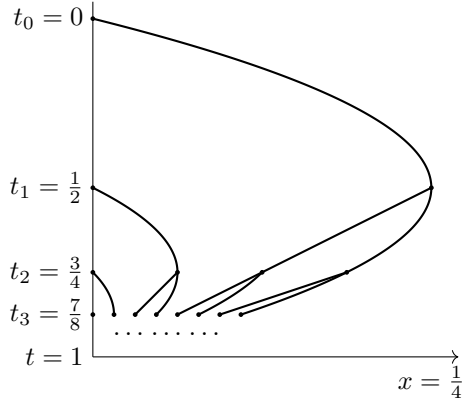
\begin{center}
\begin{tikzpicture}[x = 4cm, y=1cm, scale = 4.5]

\draw[->] (0,0) -- (0.27,0) node[right] {};
\draw[-] (0,0) -- (0,1.05) node[above] {};

\foreach \x/\y in { 
    0/1, 
    0/0.5,
    0/0.25, 0.25/0.25, 0.5/0.25, 0.75/0.25,
    0.125/0.125, 0.375/0.125, 0.625/0.125, 0.875/0.125, 
    0.5/0.5, 0/0.125, 0.25/0.125, 0.5/0.125, 0.75/0.125
}{
    \fill ({\x*\y},{\y}) circle[radius=0.2pt];
}

\node[below] at (0.25,0) {$x=\frac{1}{4}$};
\node[left] at (0,0) {$t=1$};
\node[left] at (0,1) {$t_0=0$};
\node[left] at (0,0.5) {$t_1=\frac{1}{2}$};
\node[left] at (0,0.25) {$t_2=\frac{3}{4}$};
\node[left] at (0,0.125) {$t_3=\frac{7}{8}$};

\node at (0.028, 0.0675) {$\dots$};
\node at (0.056, 0.0675) {$\dots$};
\node at (0.084, 0.0675) {$\dots$};

\draw[black,thick,domain=0.125:1,samples=200] plot ({\x*(1-\x)},{\x});
\draw[black,thick,domain=0.125:0.5,samples=200] plot ({\x*(1/2-\x)},{\x});
\draw[black,thick,domain=0.125:0.5,samples=200] plot ({\x/2)},{\x});
\draw[black,thick,domain=0.125:0.25,samples=200] plot ({\x/4)},{\x});
\draw[black,thick,domain=0.125:0.25,samples=200] plot ({\x*(1/4-\x)},{\x});
\draw[black,thick,domain=0.125:0.25,samples=200] plot ({\x*3/4},{\x});
\draw[black,thick,domain=0.125:0.25,samples=200] plot ({\x*(3/4-\x)},{\x});

\end{tikzpicture}

\end{center}

\noindent Now, it is not hard to see that:
\begin{enumerate}
    \item $\mu_t \in \PPpa(\R)$ for all $t\in[0,1]$ and $\mu_1 = \delta_0$;
    \item $(\mu_t)_{t\in [0,T]}\in \operatorname{Lip}\big([0,1],(\PP(\R),W_\infty)\big)$;
    \item $\lambda$ is atomless, since $\iota$ is injective as well.
\end{enumerate}
We conclude thanks to the following lemma.

\begin{lemma}
    If $\lambda^1,\lambda^2\in \PP(C([0,1],\R))$ are such that $(\e_t)_\sharp\lambda^1 = (\e_t)_\sharp\lambda^2 = \mu_t$ for all $t\in [0,1]$, then $\lambda^1 = \lambda^2$. In particular, $\lambda$ is the only admissible lifting of $(\mu_t)_{t\in[0,1]}$.
\end{lemma}

\begin{proof}
    Fix $t_n$ for some $n\in \N$ and define, for $i=1,2$,
    \[\Gamma_n:= \big\{ \gamma \in C([0,t_n],\R) \ : \ \gamma(t) \in \operatorname{supp} \mu_t \ \forall t\in[0,t_n]\big\}, \quad \lambda^i_n:= (\operatorname{restr}_{[0,t_n]})_\sharp \lambda^i \in \PP(C([0,t_n],\R)),\]
    where $\operatorname{restr}_{[0,t]} : C([0,T] , \R) \to C([0,t],\R)$ is the restriction map for all $t \in (0,T)$. 
    Notice that $\lambda_n$ is supported on $\Gamma_n$, that is a closed subset of $C([0,t_n],\R)$. Moreover, 
    \[\operatorname{supp}\mu_{t_n} = \{(1-t_n)x_{a_0,\dots,a_{n-1}}(t_n) \ : \ a_0,\dots,a_{n-1} \in \{0,1\}\}\]
    and the function $e_{t_n} : \Gamma_n \to \operatorname{\supp}\mu_{t_n}$ is injective. Then, disintegrating both $\lambda_n^1$ and $\lambda_n^2$ with respect to $\e_{t_n}$ (see also \cite{CavSavSod25}), we see that
    \[\mu_{t_n} = \frac{1}{2^{n}}\sum_{x\in \operatorname{supp}\mu_{t_n}} \delta_x \implies \lambda^1_n = \lambda^2_n = \frac{1}{2^n}\sum_{\gamma \in \Gamma_n} \delta_\gamma.\]
    At this point, consider any $0=s_0 <\dots < s_k < s_{k+1}=1$, and consider $t_n>s_k$, so that for any Borel set $A_0, \dots A_{k+1}$ of $\R$, it holds
    \[\lambda^i(\{\gamma \ : \ \gamma(s_j) \in A_j \ \forall j\leq k+1\}) = 
    \begin{cases}
        \lambda^i_n(\{\gamma \ : \ \gamma(s_j) \in A_j \ \forall j\leq k\}) & \text{if }0 \in A_{k+1}
        \\
        0 & \text{if }0 \notin A_{k+1},
    \end{cases}\]
    and we conclude because the right hand side is independent on $i=1,2$, showing that $\lambda^1 = \lambda^2$. 
\end{proof}
\section{A superposition principle for laws of atomic random measures}\label{section atomic}

From the result by S. Lisini in \cite{lisini2007characterization}, we know that the map $E = ((\e_t)_\sharp)_{t\in[0,T]}$ defined as
\[E:\PP(C([0,T],\R^d)) \to C([0,T],\PP(\R^d)), \quad E(\lambda) := \big( (\e_t)_\sharp \lambda \big)_{t\in[0,T]},\]
is surjective seen as a map 
\begin{equation}\label{eq:5}
    E:\PP(AC^p([0,T],\R^d)) \to AC^p([0,T],\PP_p(\R^d))
\end{equation}
for all $p\in(1,+\infty)$. Similarly, thanks to the superposition result \cite[Theorem 8.2.1]{ambrosio2005gradient}, $E$ is surjective as a map 
\begin{equation}\label{eq:6}
    E:\operatorname{SPS}(b) \to \operatorname{CE}(b)
\end{equation}
for all Borel $b:[0,T]\times \R^d \times \PP(\R^d) \to \R^d$.
In Section \ref{section: counterexample}, we have shown that, if we restrict again the previous maps to purely atomic measures, we lose surjectivity: in fact, for $p>1$
\begin{equation}\label{eq:7}
    E:\PPpa(C([0,T],\R^d))\cap \PP(AC^p([0,T],\R^d)) \to AC^p([0,T],\PP_p(\R^d)) \cap C([0,T],\PPpa(\R^d))
\end{equation}
is not onto, since the counterexample we built belongs to $AC^p([0,1],\PP_p(\R^)) \cap C([0,1],\PPpa(\R^))$ for all $p>1$. Similarly, since for all $\boldsymbol{\mu}=(\mu_t)_{t\in[0,T]} \in AC^p([0,T],\PP_p(\R^d))$ there exists a vector field $v:[0,T]\times \R^d \to \R^d$ such that $\boldsymbol{\mu} \in \operatorname{CE}(v)$, again the same counterexample shows that there exists a vector field $v$ such that the following map is not onto:
\begin{equation}\label{eq:8}
    E:\PPpa(C([0,T],\R^d))\cap \operatorname{SPS}(v) \to C([0,T],\PPpa(\R^d)) \cap \operatorname{CE}(v).
\end{equation}

Despite this lack of surjectivity, we will show that, for all $\Q\in \mathcal{Q}$, given a non-local vector field $b:[0,T]\times \R^d \times \PP(\R^d) \to \R^d$ and a curve of random measures $(M_t)_{t\in[0,T]} \in C([0,T],\PP(\PPpa(\R^d)))$ that satisfies
\begin{equation}\label{eq:9}
    \partial_tM_t + \operatorname{div}_\PP(b_tM_t) =0, \quad M_t \ll \Q \ \text{ for all }t\in [0,T],
\end{equation}
under the additional assumptions of Theorem \ref{main theorem intro}, there exist:
\begin{enumerate}
    \item $\Lambda \in \PP\big(C([0,T],\PPpa(\R^d))\cap \operatorname{CE}(b)\big)$ such that $(\mathfrak{e}_t)_\sharp \Lambda = M_t$ for all $t\in[0,T]$;
    \item $\mathfrak{L}\in \PP\big( \PPpa(C([0,T],\R^d)) \cap \operatorname{SPS}(b) \big)$ such that $(E_t)_\sharp\mathfrak{L} = M_t$ for all $t\in [0,T]$.
\end{enumerate}

We will see that the key ingredient to find a lifting $\lambda \in \PPpa(C([0,T],\R^d))$ for a curve $\boldsymbol{\mu} \in C([0,T],\PPpa(\R^d)$ is to solve a continuity equation in duality with \textit{generalized cylinder functions}, see Lemma \ref{lemma: main}.

\subsection{Approximation of generalized cylinder functions by cylinder functions}
For the whole section, we fix $\pi\in \PP(\bold{T})$ and an atomless measure $\nu\in \PP(\R^d)$, so that, for the sake of notation, we write $\Q:= \Q_{\pi,\nu} \in \PP(\PP(\R^d))$. Moreover, we denote $\bold{X} := (\R^d)^\infty_{\neq}$. 

Recall from Definition \ref{def: GC functions}, the notion of generalized cylinder functions. The aim of this section is to suitably approximate any generalized cylinder function with cylinder functions, with the aim of using them to test the continuity equation.

Notice that generalized cylinder functions are less regular than usual cylinder functions: they are still Borel measurable, since they are continuous with respect to the atomic topology, but not narrowly continuous. Moreover, they are $0$ when evaluated in an atomless measure $\mu$, see \cite[Proposition 3.8]{delloschiavo2024massive} for more details. However, their Wasserstein differential is well-defined, since for all smooth vector fields $\omega:\R^d \to \R^d$, given $\Psi^{\omega,t}$ its flow map at time $t$, it holds
\[\frac{d}{dt} \hat{F}(\Psi^{\omega,t}_\#\mu) |_{t=0}= \int_{\R^d} \nabla_W \hat{F}(x,\mu) \cdot \omega(x) d\mu(x).\]

\noindent Before proceeding, we need the following result, see also \cite[Lemma 7.19]{dello2022dirichlet}.

\begin{lemma}\label{interaction cyl function}
    Let $r\geq 1$. Given $\rho\in C^1(\R)$, $f\in C^1_b(\R^d)$, $h\in C^1_b(\R^d\times \R^d)$, define
    \begin{equation}\label{eq: interacting cyl functions} 
    F_{h,\rho,f}(\mu) := \int_{\R^d} f(x) \cdot \rho \left( \int_{\R^d} h(x,y) d\mu(y) \right) d\mu(x). 
    \end{equation}
    Then, there exists a sequence of cylinder functions $F_n\in \operatorname{Cyl}_c^1(\PP(\R^d))$ such that 
    \begin{equation}\label{appr by cyl of F_h,rho,f}
        \int_{\PP(\R^d)} |F_n(\mu) - F_{h,\rho,f}(\mu)|^r + \left(\int_{\R^d} |\nabla_W F_n - \nabla_W F_{h,\rho,f}|^r d\mu\right) d\Q(\mu) \to 0,
    \end{equation}
    where 
    \begin{equation}
    \begin{aligned}
        \nabla_W F_{h,\rho,f}(x,\mu) := & \rho\big(L_{h(x,\cdot)}(\mu)\big)\nabla f(x) + \rho'\big(L_{h(x,\cdot)}(\mu)\big)f(x) \int_{\R^d} \nabla_xh(x,y)d\mu(y) 
        \\
        & + \int_{\R^d} \rho'\big(L_{h(z,\cdot)}(\mu)\big) f(z) \nabla_y h(z,x) d\mu(z).
    \end{aligned}
    \end{equation}
\end{lemma}

\begin{proof}
	Let us first show that \eqref{appr by cyl of F_h,rho,f} holds when replacing $F_{h,\rho,f}$ with $F = \Psi \circ L_{\Phi}$ as in \eqref{cyl functions} with $\Phi = (\phi_1,\dots,\phi_k) \in C_b^1(\R^d)$. Indeed, consider $\zeta\in C^\infty([0,+\infty))$ such that $\zeta(r) = 1$ if $r\in[0,1]$, $\zeta(r) =0$ if $r\geq 2$ and $|\zeta'|\leq 2$ in $[1,2]$. Then, for all $n\geq 1$ define $\eta_n\in C^\infty([0,+\infty))$ as $\eta_n(r) = \zeta(r/n)$, which satisfies
	\[\eta_n(r)= 1\ \ \forall r\in[0,n], \quad \eta_n(r) = 0\ \ \forall r\geq 2n, \quad |\eta_n'(r)|\leq \frac{2}{n} \ \ \forall r\in[n,2n].\]
	Now, define $\Phi_n:= (\phi_{1,n}, \dots, \phi_{k,n}) \in C_c^1(\R^d;\R^k)$ as $\phi_{i,n}(x) := \phi_i(x) \eta_n(|x|)$. We can use them to define the cylinder functions $F_n := \Psi \circ L_{\Phi_n}\in\operatorname{Cyl}_c^1(\PP(\R^d))$. Now, we briefly check the convergence:
	\begin{align*}
		\int_{\PP(\R^d)} |F_n(\mu) - F(\mu)|^r dQ(\mu) \to 0 \quad \text{ as }n\to+\infty,
	\end{align*}
	by dominated convergence theorem, since $F_n$ and $F$ are uniformly bounded by the supremum of $\Psi$ and $L_{\phi_{i,n}}(\mu) \to L_{\phi_i}(\mu)$ for all $\mu\in \PP(\R^d)$ again by dominated convergence. Regarding the gradients
	\begin{align*}
		\int_{\PP(\R^d)} & \int_{\R^d} |\nabla_W F_n - \nabla_W F|^r d\mu \, dQ(\mu) 
		\\
		= & \int_{\PP(\R^d)} \int_{\R^d} \left| \sum_{i=1}^k \partial_i\Psi(L_{\Phi_{n}}(\mu)) \nabla\phi_{i,n}(x) - \partial_i\Psi(L_{\Phi}(\mu)) \nabla\phi_{i}(x) \right|^r d\mu(x) dQ(\mu)
		\\
		\leq & 2^{r-1} \int_{\PP(\R^d)} \int_{\R^d} \left| \sum_{i=1}^k \partial_i\Psi(L_{\Phi_{n}}(\mu))  \nabla\phi_{i}(x)\eta_n(|x|) - \partial_i\Psi(L_{\Phi}(\mu)) \nabla\phi_{i}(x) \right|^r d\mu(x) dQ(\mu)\\
		& + 2^{r-1} \int_{\PP(\R^d)} \int_{\R^d} \left| \sum_{i=1}^k \partial_i\Psi(L_{\Phi_{n}}(\mu))  \nabla \phi_i(x) \eta_n'(|x|) \right|^r d\mu(x) dQ(\mu),
	\end{align*}
	and both terms go to $0$ as $n\to +\infty$: the first term converges pointwise in $(x,\mu)$ since $L_{\Phi_n}(\mu) \to L_{\Phi}(\mu)$ as observed before and $\eta_n(|x|) \to 1$, thus we conclude by dominated convergence theorem; the second term is bounded by $\frac{2^{2r-1}}{n^r}  (\sum_{i} \|\partial_i \Psi\|_\infty \|\nabla \phi_i\|_\infty)^r $.

	We now show \eqref{appr by cyl of F_h,rho,f} for $F_{h,\rho,f}$.
    The definition of Wasserstein gradient is justified from the fact that for all smooth vector fields $\omega: \R^d \to \R^d$, it holds
    \begin{align*}
        & \frac{d}{dt}F_{h,\rho,f}(\Psi^{\omega,t}_\#\mu) |_{t=0} 
        = \int_{\R^d}\rho\left( \int_{\R^d} h(z,y) d\mu(y) \right) \nabla f(z) \cdot \omega(z)d\mu(z) + 
        \\
        & \ \ \ + \int_{\R^d} f(z) \rho'\left(\int_{\R^d} h(z,y) d\mu(y)\right) \cdot \left(\int_{\R^d}\nabla_{z,y} h(z,y)\cdot (\omega(z),\omega(y)) d\mu(y) \right) d\mu(z)
        \\
        & \ = \int_{\R^d} \nabla_W F_{h,\rho,f}(x,\mu)\cdot \omega(x) d\mu(x).
    \end{align*}
    Without loss of generality we can assume that $\rho(0) = 0$, otherwise consider $\tilde{\rho} = \rho - \rho(0)$ and notice that $F_{h,\rho,f}(\mu) = F_{h,\tilde{\rho},f}(\mu) + \rho(0)L_{f}(\mu)$ for all $\mu\in \PP(\R^d)$.
    Now, consider $h_n \in \operatorname{Span}\big( g_1\otimes g_2 \ : \ g_i\in C_b^1(\R^d) \big)\subset C_b^1(\R^d\times \R^d)$ converging to $h$ on bounded sets w.r.t. $\|\cdot\|_{C^1}$ and such that $\|h_n\|_{C^1}\leq 2\|h\|_{C^1}$. Consider also $\rho_n(a) := \sum_{j=1}^n c_j a^j$ be a sequence of polynomials converging to $\rho$ in $C^1$ norm on $[-2\|h\|_{C^1}, 2\|h \|_{C^1} ]$ and define 
    \[F_n(\mu):= F_{h_n,\rho_n,f}(\mu).\]
    Given the structure of the chosen functions $h_n$ and $\rho_n$, the $F_n$'s are cylinder functions (with inner functions that are continuous and bounded as in the first step of the proof), and following the previous argument it is not hard to see that they satisfy \eqref{appr by cyl of F_h,rho,f}.
\end{proof}

We are going to use this result to approximate generalized cylinder functions with cylinder functions, under the additional assumption that the $r$-capacity of the diagonal $\Delta  \subset \R^d\times \R^d$ is $0$, see §\ref{subsec: capacity}.

\begin{prop}\label{prop: approx GC with cyl}
    Let $r\in[1,+\infty)$ be such that $\operatorname{cap}_{r,\nu}(\Delta) = 0$, where $\Delta := \{(x,x)\in \R^d\times\R^d\ : \ x\in \R^d \}$. Then, for any $\hat{F}\in \operatorname{GC}_c^1(\PP(\R^d))$, there exists a sequence $F_n \in \operatorname{Cyl}_c^1(\PP(\R^d))$ such that 
    \begin{equation}\label{eq: appr of GC from cyl}
        \int_{\PP(\R^d)} |F_n(\mu) - \hat{F}(\mu)|^r +\left( \int_{\R^d} |\nabla_WF_n(x,\mu) - \nabla_W \hat{F}(x,\mu)|^rd\mu(x) \right) d\Q(\mu) \to 0.
    \end{equation}
\end{prop}

\begin{proof}
    It suffices to check the result for $\hat{F} = \hat{L}_{\hat{\phi}}$, for any $\hat{\phi} \in C^1_c(\R^d \times (0,1])$ of the form $\hat{\phi}(x,a) = f(x) \rho(a)$, with $f\in C^1_c(\R^d)$ and $\rho \in C^1_c((0,1])$. Moreover, thanks to Lemma \ref{interaction cyl function}, we can choose the approximating functions of the form \eqref{eq: interacting cyl functions}.

    \textbf{Step 1}: choice of the approximating sequence. Consider a sequence of functions $h_n \in C_b^1(\R^d \times \R^d)$ realizing the infimum in the definition of $\operatorname{cap}_{r,\nu}(\Delta)$, that is such that $ h_n \leq 1$, $h_n(x,x) = 1$ for all $x\in \R^d$ and 
    \[\int |h|^r + |\nabla h|^r d\nu\otimes \nu \to 0.\]
    Then, define $F_n(\mu) := F_{h_n,\rho,f}(\mu)$. Thanks to Lemma \ref{interaction cyl function} it is enough to show \eqref{eq: appr of GC from cyl} with this choice.
    \\
    \textbf{Step 2}: convergence of $F_n$ to $\hat{F}$ in $L^r(\Q)$. 
    \begin{align*}
        \int_{\PP(\R^d)} & |F_n(\mu) - \hat{F}(\mu)|^rd\Q(\mu) 
        =
        \int \left| \int_{\R^d} f(x) \left[ 
        \rho\left(\int_{\R^d} h_n(x,y) d\mu(y)\right) - \rho(\mu[x]) \right] \right|^r d\Q(\mu)
        \\
        \leq & 
        \|f\|_\infty^r\int \int  \left|
        \rho\left(\int_{\R^d} h_n(x,y) d\mu(y)\right) - \rho(\mu[x])  \right|^rd\mu(x) d\Q(\mu)
        \\
        \leq &
        \|f\|_\infty^r \|\rho'\|_\infty^r \int \int  \left|\int_{\R^d} h_n(x,y) d\mu(y) - \mu[x]  \right|^rd\mu(x) d\Q(\mu)
        \\
        = &
        \|f\|_\infty^r \|\rho'\|_\infty^r\int_{\bold{T}\times \boldsymbol{X}} \sum_{i=1}^\infty a_i \left| 
        \sum_{j\neq i} a_j h_n(x_i,x_j) \right|^r d\pi\otimes \nu^\infty(\bold{a},\bold{x})
        \\
        \leq &
        \|f\|_\infty^r\|\rho'\|_\infty^r \int_{\bold{T}\times \boldsymbol{X}} \sum_{i\neq j} a_i a_j \left|h_n(x_i,x_j) \right|^r d\pi\otimes \nu^\infty(\bold{a},\bold{x})
        \\
        \leq &
        \|f\|_\infty^r \|\rho'\|_\infty^r \int_{\R^d\times\R^d} \left|h_n(x,y) \right|^r d \nu^2(x,y) \to 0.
    \end{align*}
    \\
    \textbf{Step 3}: convergence of the Wasserstein gradients.
    \begin{align*}
         \int &  \int |\nabla_WF_n(x,\mu) - \nabla_W \hat{F}(x,\mu)|^rd\mu(x) d\Q(\mu)
         =
         \int \int \bigg| \rho\big(L_{h_n(x,\cdot)}(\mu)\big)\nabla f(x) - \rho(\mu[x])\nabla f(x) 
         \\
         & + \rho'\big(L_{h_n(x,\cdot)}(\mu)\big)f(x) \int_{\R^d} \nabla_x h_n(x,z)d\mu(z) 
         + \int_{\R^d} \rho'\big(L_{h_n(z,\cdot)}(\mu)\big) f(z) \nabla_y h_n(z,x) d\mu(z) \bigg|^r d\mu(x)d\Q(\mu)
         \\
         \leq &
         2^{r-1}\|\nabla f \|_\infty^r \int \int \bigg| \rho\big(L_{h_n(x,\cdot)}(\mu)\big) - \rho(\mu[x])\bigg|^r d\mu(x) d\Q(\mu) 
         \\
         & + 
         2^{r-1}\|\rho'\|_\infty^r \|f\|_\infty^r \int \int \int |\nabla_x h_n(x,z)|^r + |\nabla_y h_n(z,x)|^r d\mu(z) d\mu(x) d\Q(\mu)
         \\
         = & I_{1,n} + I_{2,n}.
    \end{align*}
    The first term $I_{1,n}$ goes to $0$, exactly as in the previous step. For the second term, thanks to Lemma \ref{lemma: barycenter of mu times mu}, we have 
    \begin{align*}
        I_{2,n}\leq & C(r,\rho',f) \int \int \int |\nabla_{x,y}h_n(x,y)|^rd\mu\otimes \mu(x,y) d\Q(\mu)
        \\
        = & 
        C(r,\rho',f)\left(c_1 \int_{\R^d\times\R^d} |\nabla_{x,y}h_n(x,y)|^r d\nu\otimes \nu(x,y) + c_2 \int_{\R^d} |\nabla_{x,y}h_n(x,x)|^r d\nu(x) \right) 
        \\
        = & 
        C(r,\rho',f)c_1 \int_{\R^d\times\R^d} |\nabla_{x,y}h_n(x,y)|^r d\nu\otimes \nu(x,y) \to 0,
    \end{align*}
    where the last equality follows from the fact that the gradient on diagonal terms vanishes.
\end{proof}

Now that we know that cylinder functions can approximate generalized cylinder functions in the sense of \eqref{eq: appr of GC from cyl}, we show that, under additional assumptions, any function in $\operatorname{GC}_c^1(\PP(\R^d))$ can be used to test the continuity equation.

\begin{teorema}\label{superposition in duality with gen cyl functions}
    Let $b:[0,T]\times \R^d \times \PP(\R^d)\to \R^d$ be a Borel vector field. Let $\Lambda\in \PP(C([0,T].\PP(\R^d)))$ be concentrated on $\operatorname{CE}(b)$ and define $\boldsymbol{M}=(M_t)_{t\in[0,T]}\in C_T(\PP(\PP(\R^d)))$ as $M_t := (\mathfrak{e}_t)_\sharp \Lambda$. Assume that 
    \begin{enumerate}
        \item $M_t = u_t \Q$, with $u \in L^1_t(L^q(\Q))$, with $q\in[1,+\infty]$. In other words 
        \begin{equation}
           \|u\|_{L^1_t(L^q(\Q))}:= \int_0^T \|u_t\|_{L^{q}(\Q)} dt <+\infty;
        \end{equation}
        \item $\|b\|_{L^p(\widetilde{M}_t\otimes dt)}^p=\int_0^T \int_{\PP(\R^d)} \int_{\R^d} |b_t(x,\mu)|^p d\mu(x) dM_t(\mu) dt<+\infty$, with $p\in[1,+\infty]$;
        \item $\operatorname{cap}_{r,\nu}(\Delta) = 0$, with $r\in [1,+\infty)$;
        \item $\frac{p'}{r}+\frac{1}{q}\leq1$.
    \end{enumerate}
    Then, for $\Lambda$-a.e. $\boldsymbol{\mu}= (\mu_t)_{t\in [0,T]} \in C_T(\PP(\R^d))$ and for all $\hat{F}\in \operatorname{GC}_c^1(\PP(\R^d))$, it holds
    \begin{equation}\label{eq: non-local CE in dual with GC}
        \frac{d}{dt} \hat{F}(\mu_t) = \int_{\R^d} \nabla_W\hat{F}(x,\mu_t) \cdot  b(t,x,\mu_t) d\mu_t(x),
    \end{equation}
    in the sense of distributions in $(0,T)$.
    Moreover, it holds $\partial_t M_t + \operatorname{div}_\PP(b_tM_t) = 0$ in the sense of \eqref{eq: CERM}, and in duality with generalized cylinder functions, that is for any $\hat{F}\in \operatorname{GC}_c^1(\PP(\R^d))$ it holds 
    \begin{equation}\label{eq: CERM in dual with GC}
        \frac{d}{dt}\int_{\PP(\R^d)} \hat{F}(\mu) dM_t(\mu) = \int_{\PP(\R^d)} \int_{\R^d} \nabla_W\hat{F}(x,\mu) \cdot  b(t,x,\mu) d\mu(x) dM_t(\mu).
    \end{equation}
\end{teorema}

\begin{proof}
Let $\hat{F} \in\operatorname{GC}_c^1(\PP(\R^d))$ and consider $(F_n)_{n\in \N}\subset \operatorname{Cyl}_c^1(\PP(\R^d))$ such that $F_n {\longrightarrow} \hat{F}$ in the sense of \eqref{eq: appr of GC from cyl}. For all $\xi \in C_c^1(0,T)$, let us first prove
\begin{equation}\label{eq: proof}
    0 = \int \bigg| \int_0^T \xi'(t)\hat{F}(\mu_t) + \xi(t)\left(\int_{\R^d} \nabla_W \hat{F}(x,\mu_t) \cdot b_t(x,\mu_t) d\mu_t(x) \right) dt \bigg| d\Lambda(\boldsymbol{\mu}).
\end{equation}
Since \eqref{eq: non-local CE in dual with GC} holds with $F_n$ in place of $\hat{F}$ for $\Lambda$-a.e. $\boldsymbol{\mu}$, we have 
\begin{align*}
     \int & \bigg| \int_0^T \xi'(t)\hat{F}(\mu_t) + \xi(t)\left(\int_{\R^d} \nabla_W \hat{F}(x,\mu_t) \cdot b_t(x,\mu_t) d\mu_t(x) \right) dt \bigg| d\Lambda(\boldsymbol{\mu})
     \\
     & = \int \bigg| \int_0^T \xi'(t)\hat{F}(\mu_t) - \xi'(t)F_n(\mu_t) dt
     \\
     & \quad \ \  + \int_0^T \xi(t)\left(\int_{\R^d} \nabla_W \hat{F}(x,\mu_t) \cdot b_t(x,\mu_t) - \nabla_W F_n\cdot b_t(x,\mu_t) d\mu_t(x) \right) dt \bigg| d\Lambda(\boldsymbol{\mu})
     \\
     & \leq \int  \bigg|
        \int_0^T  \xi'(t)\hat{F}(\mu_t) dt - \int_0^T \xi'(t)F_n(\mu_t) dt
        \bigg| d\Lambda\big(\boldsymbol{\mu}\big)
        \\
        & \quad \ \ + \int \bigg|  \int_0^T \int_{\R^d}\xi(t)\nabla_W \hat{F}(x,\mu_t) - \nabla_W F_n(x,\mu_t)\cdot b_t(x,\mu_t)d\mu_t(x) dt  \bigg|d\Lambda\big(\boldsymbol{\mu}\big)
        \\
        & \leq \|\xi'\|_\infty\int_0^T \int_{\PP(\R^d)} |\hat{F}(\mu)-F_n(\mu)| dM_t(\mu) dt
        \\
        & \quad \ \ + \|\xi\|_\infty
        \int_0^T \int_{\PP(\R^d)} \int_{\R^d} \left|\Big( \nabla_W \hat{F}(x,\mu)-\nabla_W F_n(x,\mu) \Big)\cdot b_t(x,\mu) \right|d\mu(x) dM_t(\mu)dt
        \\
        & = I_{1,n} + I_{2,n}.
\end{align*}
Both terms go to $0$, in fact:
\begin{align*}
    I_{1,n} = & \|\xi'\|_\infty
        \int_0^T \int |\hat{F}(\mu)- F_n(\mu)|u_t(\mu) d\Q(\mu) dt 
        \\
        \leq &
        \lVert \xi' \rVert_{\infty} \int_0^T\left(\int |\hat{F}-F_n |^r d\Q\right)^{\frac{1}{r}} \left( \int_\PP u_t^q  d\Q\right)^{\frac{1}{q}}dt \to 0;
\end{align*}

\begin{align*}
    I_{2,n}
        \leq &
        \|\xi\|_\infty \|b\|_{L^p(\tilde{M}_t\otimes dt)} \left(\int_0^T \int\int | \nabla_W \hat{F}-\nabla_W F_n|^{p'} d\mu(x) dM_t(\mu) dt\right)^{\frac{1}{p'}}
        \\
        \leq &
        \|\xi\|_\infty \|b\|_{L^p} \left( \int_0^T \left\{\int \left[\int |\nabla_W \hat{F}-\nabla_W F_n|^{p'} d\mu(x)\right]^{q'} d\Q(\mu)\right\}^{\frac{1}{q'}} \left\{\int |u_t|^q d\Q \right\}^{\frac1q} dt \right)^{\frac{1}{p'}}
        \\
        = & \|\xi\|_\infty \|b\|_{L^p}\|u\|_{L^1_t(L^q(\Q))}^{1/p'}\left\{\int \left[\int |\nabla_W \hat{F}-\nabla_W F_n |^{p'} d\mu(x)\right]^{q'} d\Q(\mu)\right\}^{\frac{1}{p'q'}}
        \\
        \leq &
        \|\xi\|_\infty \|b\|_{L^p}\|u\|_{L^1_t(L^q(\Q))}^{1/p'}\left(\int \int |\nabla_W\hat{F}-\nabla_WF_n |^{p'q'}d\mu(x)d\Q(\mu)\right)^{\frac{1}{p'q'}}
        \\
        \leq &
        \|\xi\|_\infty \|b\|_{L^p}\|u\|_{L^1_t(L^q(\Q))}^{1/p'}\left(\int \int |\nabla_W\hat{F}-\nabla_WF_n |^{r}d\mu(x)d\Q(\mu)\right)^{1/r} \to 0
\end{align*}
Exploiting \eqref{eq: proof} for a countable dense subset of $\{\xi_n\}\subset C_c^1(0,T)$, we have that \eqref{eq: non-local CE in dual with GC} holds for the chosen $\hat{F}$, for all $\boldsymbol{\mu} \notin \mathcal{N}_{\hat{F}} \subset C([0,T],\PP(\R^d))$, with $\Lambda(\mathcal{N}_{\hat{F}}) = 0$.  We need to find a set $\mathcal{N}$ that does not depend on the chosen generalized cylinder functions. To this aim, consider a countable dense subset $\{\hat{\phi}_n\}$ in $C_c^1(\R^d \times (0,1])$ and consider $\mathcal{N}:= \bigcap_{n\in \N} \mathcal{N}_{\hat{L}_{\hat{\phi}_n}}$. By density, for all $\boldsymbol{\mu}\notin \mathcal{N}$, \eqref{eq: non-local CE in dual with GC} holds for $\hat{L}_{\hat{\phi}}$ for all $\hat{\phi} \in C_c^1(\R^d\times (0,1])$. Then, the same holds for any $\hat{F}\in \operatorname{GC}_c^1(\PP(\R^d))$ simply by chain rule property.
Finally, \eqref{eq: CERM in dual with GC} follows from the previous computations.
\end{proof}

\subsection{Nested superposition principles for laws of purely atomic random measures}\label{subsec: nested atomic}

The main goal of this subsection is to prove the following \textit{atomic nested superposition principles}, that are, respectively, of a differential and a metric nature. The first theorem coincides with Theorem \ref{main theorem intro}, and we restate it here for completeness.

\begin{teorema}\label{main theorem}
    Let $b:[0,T]\times \R^d \times \PP(\R^d) \to \R^d$ be Borel measurable, $\Lambda \in \PP(C([0,T],\PP(\R^d)))$ and $\boldsymbol{M} = (M_t)_{t\in[0,T]}\in C([0,T],\PP(\PP(\R^d)))$ be satisfying the assumptions of Theorem \ref{superposition in duality with gen cyl functions}. Then, there exists $\mathfrak{L}\in \PP(\PP(C([0,T],\R^d)))$ such that $E_\sharp \mathfrak{L} = \Lambda$ and $\mathfrak{L}\in \PP\big( \PPpa(C([0,T],\R^d)) \cap \operatorname{SPS}(b) \big)$. In particular, $(E_t)_\sharp \mathfrak{L} = M_t$ for all $t\in[0,T]$ and $\Lambda$ is concentrated over $\operatorname{CE}(b)\cap C([0,T],\PP^{\operatorname{pa}}(\R^d))$.
\end{teorema}

Notice that, given $(M_t)_{t\in[0,T]}$ that satisfies the hypothesis of Theorem \ref{superposition in duality with gen cyl functions}, then any $\Lambda \in \PP(C([0,T],\PP(\R^d)))$ given by Theorem \eqref{thm: nested superposition} satisfies the hypothesis of the previous theorem.

\begin{teorema}\label{main theorem metric}
    Let $\Lambda \in \PP(C([0,T],\PP(\R^d)))$ be concentrated over $AC^p([0,T],\PP_p(\R^d))$ and define $\boldsymbol{M} = (M_t)_{t\in[0,T]} \in C([0,T],\PP(\PP(\R^d)))$ as $M_t = (\mathfrak{e}_t)_\sharp \Lambda$. Assume:
    \begin{enumerate}
        \item $M_t = u_t \Q$, with $u \in L^1_t(L^q(\Q))$, with $q\in[1,+\infty]$;
        \item for some $p>1$, it holds
        \begin{equation}
            \int\int_0^T |\dot{\boldsymbol{\mu}}|_{W_p}^p(t)dt d\Lambda(\boldsymbol{\mu})<+\infty;
        \end{equation}
        \item $\operatorname{cap}_{r,\nu}(\Delta) = 0$, with $r\in [1,+\infty)$;
        \item $\frac{p'}{r}+\frac{1}{q}\leq1$.
    \end{enumerate}
    Then, there exists $\mathfrak{L}\in \PP(\PPpa(AC^p([0,T],\R^d)))$ such that $E_\sharp \mathfrak{L} = \Lambda$ and 
    \begin{equation}\label{eq: minimal lifting}
        \int\int\int_0^T |\dot\gamma|^p(t) dt d\lambda (\gamma)d\mathfrak{L}(\lambda) <+\infty
    \end{equation}
    In particular, $(E_t)_\sharp \mathfrak{L} = M_t$ for all $t\in[0,T]$ and $\Lambda$ is concentrated over $\PP(C([0,T],\PP^{\operatorname{pa}}(\R^d)))$.
\end{teorema}

Both the theorems are a consequence of the following lemma, whose proof is postponed to the next subsection.

\begin{lemma}\label{lemma: main}
    Let $\boldsymbol{\mu} = (\mu_t)_{t\in[0,T]}\in C([0,T],\PP(\R^d))$ be satisfying:
    \begin{itemize}
        \item[(H1)] there exists $v:[0,T]\times \R^d \to \R^d$ Borel such that $\int_0^T \int_{\R^d} |v
        _t(x)| d\mu_t(x) dt<+\infty$ and $\partial_t\mu_t + \operatorname{div}(v_t\mu_t) = 0$ in the classic sense and in duality with $\operatorname{GC}_c^1(\PP(\R^d))$, that is
        \begin{equation}
        \frac{d}{dt} \hat{F}(\mu_t) = \int_{\R^d} \nabla_W\hat{F}(x,\mu_t) \cdot  v_t(x) d\mu_t(x) \quad \forall \hat{F}\in \operatorname{GC}_c^1(\PP(\R^d));
    \end{equation}
        \item[(H2)] there exists $I_{\boldsymbol{\mu}} \subset [0,T]$ of full Lebesgue measure such that $\mu_t \in \PPpa(\R^d)$ for all $t\in I_{\boldsymbol{\mu}}$.
    \end{itemize}
    Then, there exists $\lambda \in \PPpa(C([0,T],\R^d))$ such that $(\e_t)_\sharp\lambda = \mu_t$ for all $t\in [0,T]$ and that is concentrated over $\gamma \in AC([0,T],\R^d)$ solving $\dot{\gamma}(t) = v_t(\gamma(t))$ for a.e. $t\in [0,T]$. In particular $I_{\boldsymbol{\mu}} = [0,T]$, that is $\mu_t\in \PPpa(\R^d)$ for all $t\in[0,T]$, 
\end{lemma}

\begin{proof}[Proof of Theorem \ref{main theorem}]
    Thanks to Theorem \ref{superposition in duality with gen cyl functions}, $\Lambda$-a.e. $\boldsymbol{\mu}$ satisfies the hypothesis \textit{(H1)} in Lemma \ref{lemma: main}, with $v(t,x) = b(t,x,\mu_t)$. On the other hand, the following computation shows that $\Lambda$-a.e. $\boldsymbol{\mu}$ satisfies hypothesis \textit{(H2)} as well:
    \[
    \begin{aligned}
    \int &  \int_0^T \mathds{1}_{\{(t,\boldsymbol{\mu}) \ : \ \mu_t \notin \PPpa(\R^d)\}} dt d\Lambda(\boldsymbol{\mu}) = \int_0^T \int \mathds{1}_{\{(t,\boldsymbol{\mu}) \ : \ \mu_t \notin \PPpa(\R^d)\}} d\Lambda(\boldsymbol{\mu}) dt
    \\
    & = \int_0^T \int \mathds{1}_{\{(t,\mu) \ : \ \mu \notin \PPpa(\R^d)\}} dM_t(\mu) dt = \int_0^T M_t\big( \PP(\R^d) \setminus \PPpa(\R^d) \big) dt = 0.
    \end{aligned}
    \]
    Then, there exists $\lambda \in \PPpa(C([0,T],\R^d)) \cap \operatorname{SPS}(b)$ such that $E(\lambda) = \boldsymbol{\mu}$. In particular $\Lambda$ is concentrated over the set $E\big(\PPpa(C([0,T],\R^d)) \cap \operatorname{SPS}(b)\big)$, that is a Souslin set (and thus $\Lambda$-measurable) because $\PPpa(C([0,T],\R^d)) \cap \operatorname{SPS}(b)$ is Borel measurable, see \cite[Proposition 5.8]{pinzisavare2025}, and $E$ is continuous. Then, using the measurable selection theorem \cite[Theorem 6.9.1]{bogachev2007measure}, we can find a Souslin-Borel measurable right inverse $G: E(\PPpa(C([0,T],\R^d)) \cap \operatorname{SPS}(b)) \to \PPpa(C([0,T],\R^d)) \cap \operatorname{SPS}(b)$, i.e. such that $E(G(\boldsymbol{\mu})) = \boldsymbol{\mu}$ for all $\boldsymbol{\mu}\in E\big(\PPpa(C([0,T],\R^d)) \cap \operatorname{SPS}(b)\big)$. Then, since Souslin sets are universally measurable, the measure $\mathfrak{L}:= G_\sharp \Lambda$ is well-defined for any $\Lambda$ as in the statement, and it clearly satisfies the desired properties.
\end{proof}

\begin{proof}[Proof of Theorem \ref{main theorem metric}]
    Similarly as before, $\Lambda$-a.e. $\boldsymbol{\mu}$ satisfies the hypothesis \textit{(H2)} in Lemma \ref{lemma: main}. On the other hand, from \cite[Theorem 8.3.1]{ambrosio2005gradient}, for $\Lambda$-a.e. $\boldsymbol{\mu}$ there exists $v_{\boldsymbol{\mu}} : [0,T]\times \R^d \to \R^d$ Borel measurable and such that 
    \begin{equation}\label{eq: vf and velocity}
    \partial_t\mu_t + \operatorname{div}(v_{\boldsymbol{\mu},t}\mu_t) = 0 \quad \text{and} \quad \int_{\R^d} |v_{\boldsymbol{\mu}}(t,x)|^pd\mu_t(x) = |\dot{\boldsymbol{\mu}}|_{W_p}^p(t) \quad \text{for a.e. }t\in[0,T].
    \end{equation}
    We want to show that the continuity equation $\partial_t\mu_t + \operatorname{div}(v_{\boldsymbol{\mu},t}\mu_t) = 0$ is satisfied in duality with cylinder functions as well. Let $\hat{F}\in \operatorname{GC}_c^1(\PP(\R^d))$ and consider $F_n\in \operatorname{Cyl}_c^1(\PP(\R^d))$ such that $F_n \to F$ as in \eqref{eq: appr of GC from cyl}. By chain rule, we know that for all $\xi\in C_c^1(0,T)$ we have 
    \[\int_0^T \xi'(t) F_n(\mu_t) dt = - \int_0^T \xi(t) \int_{\R^d} \nabla_WF_n(x,\mu_t)\cdot v_{\boldsymbol{\mu},t}(x) d\mu_t(x) dt = 0.\]
    Then, we are done if we prove that 
    \begin{equation}\label{first}
        \int_0^T \xi'(t) F_n(\mu_t) dt \to \int_0^T \xi'(t) \hat{F}(\mu_t) dt
    \end{equation}
    and
    \begin{equation}\label{second}
    \int_0^T \xi(t) \int_{\R^d} \nabla_WF_n(x,\mu_t)\cdot v_{\boldsymbol{\mu},t}(x) d\mu_t(x) dt \to \int_0^T \xi(t) \int_{\R^d} \nabla_W\hat{F}(x,\mu_t)\cdot v_{\boldsymbol{\mu},t}(x) d\mu_t(x) dt
    \end{equation}
    for $\Lambda$-a.e. $\boldsymbol{\mu} = (\mu_t)_{t\in[0,T]}$, since, as for the proof of Theorem \ref{superposition in duality with gen cyl functions}, it suffices to check it for a countable amount of $\xi \in C_c^1(0,T)$ and $\hat{F}\in \operatorname{GC}_c^1(\PP(\R^d))$. Regarding \eqref{first}, notice that 
    \begin{align*}
        \int \bigg|\int_0^T \xi'(t) F_n(\mu_t) dt - \int_0^T \xi'(t) \hat{F}(\mu_t) dt \bigg| d\Lambda(\boldsymbol{\mu}) \leq \|\xi'\|_\infty \int_0^T \int |F_n(\mu) - \hat{F}(\mu)| dM_t(\mu) dt \to 0,
    \end{align*}
    as $I_{1,n}$ in the proof of Theorem \eqref{superposition in duality with gen cyl functions}, then \eqref{first} holds for $\Lambda$-a.e. $\boldsymbol{\mu}$. Regarding \eqref{second}, we have 
    \begin{align*} 
    \bigg|& \int_0^T \xi(t) \int_{\R^d} \nabla_WF_n(x,\mu_t)\cdot v_{\boldsymbol{\mu},t}(x) d\mu_t(x) dt - \int_0^T \xi(t) \int_{\R^d} \nabla_W\hat{F}(x,\mu_t)\cdot v_{\boldsymbol{\mu},t}(x) d\mu_t(x) dt \bigg| 
    \\
    & \leq \|\xi\|_\infty \int_0^T \int_{\R^d} |v_{\boldsymbol{\mu},t}(x)| |\nabla_W F_n(x,\mu_t) - \nabla_W\hat{F}(x,\mu_t)| d\mu_t(x) dt
    \\
    & \leq \|\xi\|_\infty \left( \int_{0}^T \int_{\R^d} |v_{\boldsymbol{\mu},t}|^pd\mu_t(x) dt \right)^{\frac{1}{p}}\left( \int_0^T \int_{\R^d} |\nabla_W F_n(x,\mu_t) - \nabla_W\hat{F}(x,\mu_t)|^{p'} d\mu_t(x) dt \right)^{\frac{1}{p'}}
    \\
    & = \|\xi\|_\infty \left( \int_{0}^T |\dot{\boldsymbol{\mu}}|_{W_p}^p(t) dt \right)^{\frac{1}{p}}\left( \int_0^T \int_{\R^d} |\nabla_W F_n(x,\mu_t) - \nabla_W\hat{F}(x,\mu_t)|^{p'} d\mu_t(x) dt \right)^{\frac{1}{p'}},
    \end{align*}
    and this last formula can be integrated with respect to $\Lambda$ since it is measurable in the variable $\boldsymbol{\mu}$ (see \cite[Appendix D]{pinzisavare2025}), so that 
    \begin{align*}
        & \int  \left( \int_{0}^T |\dot{\boldsymbol{\mu}}|_{W_p}^p(t) dt \right)^{\frac{1}{p}}\left( \int_0^T \int_{\R^d} |\nabla_W F_n(x,\mu_t) - \nabla_W\hat{F}(x,\mu_t)|^{p'} d\mu_t(x) dt \right)^{\frac{1}{p'}} d\Lambda(\boldsymbol{\mu}) 
        \\
        & \leq \left( \int \int_{0}^T |\dot{\boldsymbol{\mu}}|_{W_p}^p(t) dt d\Lambda(\boldsymbol{\mu}) \right)^{\frac{1}{p}}\left( \int \int_0^T \int_{\R^d} |\nabla_W F_n(x,\mu_t) - \nabla_W\hat{F}(x,\mu_t)|^{p'} d\mu_t(x) dt d\Lambda(\boldsymbol{\mu})\right)^{\frac{1}{p'}}
    \end{align*}
    and it goes to $0$ following the same computations done for $I_{2,n}$ in the proof of Theorem \ref{superposition in duality with gen cyl functions}. This shows that $\Lambda$-a.e. $\boldsymbol{\mu}$ satisfies \textit{(H1)} as well. Then, applying Lemma \ref{lemma: main}, for $\Lambda$-a.e. $\boldsymbol{\mu}$ there exists $\lambda \in \PPpa(C([0,T],\R^d))$ concentrated over absolutely continuous curves and such that $\dot\gamma(t) = v_{\boldsymbol{\mu},t}(\gamma(t))$ for $\lambda$-a.e. $\gamma$ and a.e. $t\in[0,T]$. Moreover, thanks to \eqref{eq: vf and velocity}, $\lambda$ is supported over $AC^p([0,T],\R^d)$ and 
    \[\int \int_0^T |\dot\gamma|^p(t)dt d\lambda(\gamma) = \int_0^T |\dot{\boldsymbol{\mu}}|_{W_p}^p(t)dt.\] In particular, $\Lambda$ is concentrated over the set $E\big( \PPpa(AC^p([0,T],\R^d)) \cap \mathfrak{A}_{\min} \big)$, where 
    \[\mathfrak{A}_{\min} = \left\{ \lambda \in \PP(AC^p([0,T],\R^d)) \ : \ \int \int_0^T |\dot\gamma|^p(t)dt d\lambda(\gamma) = \int_0^T |\dot{\boldsymbol{\mu}}|_{W_p}^p(t)dt, \text{ with }{\boldsymbol{\mu}} = E(\lambda) \right\}.\]
    Such a set is Souslin, being the image of the intersection of two Borel subsets (see for example \cite[Appendix D]{pinzisavare2025}) through the continuous function $E$. Then, using again the measurable selection theorem \cite[Theorem 6.9.1]{bogachev2007measure}, there exists a Souslin-Borel measurable map $G: E\big( \PPpa(AC^p([0,T],\R^d)) \big) \to\PPpa(AC^p([0,T],\R^d))\cap \mathfrak{A}_{\min}$ such that $E(G(\boldsymbol{\mu})) = \boldsymbol{\mu}$ for $\Lambda$-a.e. $\boldsymbol\mu \in E\big( \PPpa(AC^p([0,T],\R^d)) \cap \mathfrak{A}_{\min} \big)$. Then, the measure $\mathfrak{L}:= G_\sharp \Lambda$ is well-defined (because Souslin set are universally measurable) and clearly satisfies the requirements.
\end{proof}

\subsection{Proof of Lemma \ref{lemma: main}}\label{subsec: proof}

From the assumption \textit{(H2)}, for all $t\in I_{\boldsymbol{\mu}}$ there exists $\bold{a}(t) = (a_i(t))_{i\in \N} \in \bold{T}$ and $\bold{x}(t) = (x_i(t))_{i\in \N} \in \bold{X}$ such that 
    \begin{equation}
        \mu_t = \sum_{i\in \N} a_i(t) \delta_{x_i(t)} \quad \text{ for all }t \in I_{\boldsymbol{\mu}}.
    \end{equation}

We need to find a measure $\lambda \in \PP(C([0,T],\R^d))$ such that $(\e_t)_\sharp\lambda = \mu_t$ for all $t\in[0,T]$ and that has the form 
\begin{equation}
    \lambda = \sum_{i\in \N} a_i \delta_{\gamma_i},
\end{equation}
for some $\bold{a} \in \bold{T}$ and $(\gamma_i)_{i\in\N} \subset AC([0,T],\R^d)$ that satisfy $\dot{\gamma}_i(t) = v_t(\gamma_i(t))$. It is clear that, a posteriori, the weights $a_i(t)$ will not depend on $t$. The next lemma gets us closer to this property.

\begin{lemma}\label{lemma: dual with indicators}
    Let $\boldsymbol{\mu}\in C([0,T],\PP(\R^d))$ be satisfying (H1).
    Then, we can use as test functions 
    \begin{equation}
        \hat{L}_{\phi,a}(\mu):= \int_{\R^d} \phi(x) \mathds{1}_{\{a\}}(\mu[x]) d\mu(x), \quad \text{ with} \quad \nabla_W \hat{L}_{\phi,a}(x,\mu) := \nabla \phi(x) \mathds{1}_{\{a\}}(\mu[x])
    \end{equation}
    for all $\phi\in C_b^1(\R^d)$ and $a\in (0,1]$, that is, for all $\xi \in C_c^1(0,T)$
    \begin{equation}
        \int_0^T \xi'(t)\hat{L}_{\phi,a}(\mu_t) dt = - \int_0^T \xi(t) \int_{\R^d} \nabla_W \hat{L}_{\phi,a}(x,\mu_t)\cdot v_t(x) d\mu_t(x) dt.
    \end{equation}
    Moreover, there exists $J_{\boldsymbol{\mu}}\subset I_{\boldsymbol{\mu}}$ of full Lebesgue measure such that for all $a\in(0,1]$ there exists $C_a \in [0,1]$ such that
    \begin{equation}\label{eq: const on J_mu}
    \int_{\R^d} \mathds{1}_{\{a\}}(\mu_t[x]) d\mu_t(x) = C_a \ \text{ for all }t\in J_{\boldsymbol{\mu}}. 
    \end{equation}
\end{lemma}

\begin{proof}
    Let $\phi\in C_c^1(\R^d)$ and $a\in (0,1]$, the function $\mathds{1}_{\{a\}}$ can be approximated pointwise by $\rho_n \in C_c^1((0,1])$. The assumptions allow us to use the generalized cylinder function $\hat{L}_{\phi,\rho_n}(\mu) := \int_{\R^d} \phi(x) \rho_n(\mu\{x\})d\mu(x)$ as a test function, to obtain that $(\mu_t)_{t\in [0,T]} \in C([0,T],\PP(\R^d))$ satisfies 
    \[\int_0^T \xi'(t) \int_{\R^d} \phi(x) \rho_n(\mu_t[x])d\mu_t(x) dt = - \int_0^T \xi(t) \int_{\R^d}  \rho_n(\mu_t[x]) \nabla \phi(x)\cdot  v_t(x) d\mu_t(x) dt,\]
    for all $\xi \in C_c^1(0,T)$. Using the fact that $\rho_n \to \mathds{1}_{\{a\}}$ pointwise and dominated convergence theorem, it is not hard to obtain
    \[\int_0^T \xi'(t) \int_{\R^d} \phi(x) \mathds{1}_{\{a\}}(\mu_t[x])d\mu_t(x) dt = - \int_0^T \xi(t) \int_{\R^d}  \mathds{1}_{\{a\}}(\mu_t[x]) \nabla \phi(x)\cdot  v_t(x) d\mu_t(x) dt.\]

    The case in which $\phi \in C_b^1(\R^d)$ can be obtained through a standard cut-off argument, considering the approximating sequence $\phi_R(x) := \phi(x) f(x/R)$ as $R\to +\infty$, where $f \in C_c^1(\R^d)$ and $f\equiv 1$ in $B(0,1)$.

    To prove \eqref{eq: const on J_mu}, consider $\{\rho^{(n)}\}_{n\in \N}\subset C_c^1((0,1])$ a dense and countable subset of $C_0([0,1])$, that are continuous functions $\rho:[0,1]\to\R$ such that $\rho(0) = 0$. We know that, for all $n\in \N$,  
    \( t\mapsto \int_{\R^d} \rho^{(n)}(\mu_t[x]) d\mu_t(x)\)
    is weakly differentiable with null derivative, and so we can find a set $J_{\boldsymbol{\mu}} \subset I_{\boldsymbol{\mu}}$ where those maps are constant, say $C^{(n)}$. 
    Now, for all $\rho\in C_0([0,1])$, the map 
    \[J_{\boldsymbol{\mu}} \ni t \mapsto \int_{\R^d} \rho(\mu_t[x]) \ \text{ is constant.}\]
    Indeed, let $\rho^{(n_k)}$ be a sequence that converges uniformly to $\rho$, which implies that for all $t\in J_{\boldsymbol{\mu}}$
    \[\lim_{k\to +\infty} C^{(n_k)} = \lim_{k \to +\infty} \int_{\R^d} \rho^{(n_k)}(\mu_t[x]) d\mu_t(x) = \int_{\R^d} \rho(\mu_t[x]) d\mu_t(x),\]
    and in particular the right hand side does not depend on $t\in J_{\boldsymbol{\mu}}$, and we denote it by $C_{\rho}$. Similarly, for all $a\in (0,1]$, consider $\rho_{n} \in C_0([0,1])$ converging pointwise to $\mathds{1}_{\{a\}}$ as above, and then, by dominated convergence theorem, we have that for all $t\in J_{\boldsymbol{\mu}}$
    \[C_{\rho_n} = \int_{\R^d} \rho_n(\mu_t[x]) d\mu_t(x) \to \int_{\R^d} \mathds{1}_{\{a\}}(\mu_t[x]) d\mu_t(x),\]
    and in particular it does not depend on $t\in J_{\boldsymbol{\mu}}$.
\end{proof}

Now, fix $\Bar{t}\in J_{\boldsymbol{\mu}}$, and define
\[S_{a}:=\{i\in \N \ : \ a_i(\Bar{t}) = a\}, \quad a \in \mathcal{A}:= \{ a \in (0,1] \ : \ a = a_i(\Bar{t}) \text{ for some } i\in \N \}.\]

\noindent Consider $N_a = \sharp S_a$ for all $a\in \mathcal{A}$. Thanks to Lemma \ref{lemma: dual with indicators} and $\Bar{t}\in J_{\boldsymbol{\mu}}$, for all $a\in \mathcal{A}$
\[\int_{\R^d} \mathds{1}_{\{a\}}(\mu_t[x]) d\mu_t(x) = a N_a \quad \text{for all }t\in J_{\boldsymbol{\mu}}.\]

This argument leads to the equality $a_i(t) = a_i(\Bar{t})$ for all $t\in J_{\boldsymbol{\mu}}$, since for all $t\in J_{\boldsymbol{\mu}}$, the sequence $\{a_i(t)\}_{i\in \N}$ is non-increasing in $i$ and there exist $\sharp S_a$ points in $\R^d$ to which $\mu_t$ is assigning mass equal to $a$. In particular, thanks to Lemma \ref{lemma: atomic topology}, the map 
\begin{equation}\label{eq: j_mu}
J_{\boldsymbol{\mu}} \ni t\mapsto \mu_t \text{ is continuous in the atomic topology}.
\end{equation}

Thanks to \cite{delellis2003regular} and Lemma \ref{lemma: atomic topology}, we could also rearrange the positions $\{x_i(t)\}_{i\in \N}$ with permutations of $S_a$, to have that $J_{\boldsymbol{\mu}}\ni t \mapsto x_i(t) $ is continuous for all $i\in \N$. However, we do not need that to pursue our strategy and it will be a consequence of Lemma \ref{lemma: nu_t^a}.

Now, let us exploit again Lemma \ref{lemma: dual with indicators} using $\hat{L}_{\phi,a}$ as test function, with general $\phi\in C_b^1(\R^d)$ and $a\in \mathcal{A}$. This gives the following: for all $\xi \in C_c^1(0,T)$, $\phi \in C_b^1(\R^d)$ and $a\in \mathcal{A}$
\begin{equation}
   a \int_0^T \xi'(t) \sum_{i\in S_a} \phi(x_i(t)) dt = - a \int_0^T \xi(t) \sum_{i\in S_a}\nabla \phi(x_i(t))\cdot  v_t(x_i(t)) dt.
\end{equation}
This can be rewritten more compactly as 
\begin{equation}
    \partial_t \nu_t^{(a)} + \operatorname{div}(v_t \nu_t^{(a)}) = 0, \qquad \nu_t^{(a)}:= \frac{1}{N_a} \sum_{i\in S_a} \delta_{x_i(t)} \in \PP^{N_a}(\R^d). 
\end{equation}
The measure $\nu_t^{(a)}$ is well defined for all $t\in J_{\boldsymbol{\mu}}$ and there it is narrowly continuous because of \eqref{eq: j_mu} and Lemma \ref{lemma: atomic topology}. Thanks to \cite[Lemma 8.1.2]{ambrosio2005gradient}, it exists a unique narrowly continuous representative $\tilde{\nu}_t^{(a)}$ on $[0,T]$, that coincides with $\nu_{t}^{(a)}$ in $J_{\boldsymbol{\mu}}$.

\begin{lemma}\label{lemma: nu_t^a}
    For all $t\in [0,T]$, $\tilde{\nu}_t^{(a)} \in \PP^{N_a}(\R^d)$. 
\end{lemma}
\begin{proof}
    Take $t\in[0,T]\setminus J_{\boldsymbol{\mu}}$ and consider $t_n \in J_{\boldsymbol{\mu}}$ such that $t_n \to t$ as $n\to +\infty$. By tightness of $\{\nu_{t_n}^{(a)}\}_{n\in \N}$, there exists a compact set $K\subset \R^d$ such that $x_i(t_n)\in K$ for all $i,n\in \N$. Then, there exists a subsequence $n_k$ such that for all $i\in \N$ $x_i(t_{n_k}) \to x_i$ as $k\to +\infty$, for some $x_i\in \R^d$. Then, it holds \[ \tilde{\nu}_t^{(a)} = \frac{1}{N_a} \sum_{i\in S_a} \delta_{x_i}. \qedhere\]
\end{proof}

At this point, notice that 
\[[0,T] \ni t \mapsto \sum_{a\in \mathcal{A}} aN_a \tilde{\nu}_t^{(a)} \in \PPpa(\R^d) \text{ is narrowly continuous}.\]
Then, it must coincide with $\mu_t$ for all $t\in [0,T]$ by uniqueness of the continuous representative, since it is narrowly continuous and it coincides with $\mu_t$ for all $t\in J_{\boldsymbol{\mu}}$. In particular, $\mu_t \in \PPpa(\R^d)$ for all $t\in [0,T]$.
\\
Moreover, for all $a\in \mathcal{A}$, we can apply Theorem \ref{thm: atomic superposition} to the curve $(\tilde{\nu}_t^{(a)})_{t\in[0,T]}$, to obtain a measure $$\lambda^{(a)} = \frac{1}{N_a} \sum_{i\in S_a} \delta_{\gamma_i} \in \PP^{N_a}(C([0,T],\R^d))$$ for some $\gamma_i \in AC([0,T],\R^d)$ satisfying $\dot\gamma_i(t) =  v_t(\gamma_i(t))$, and such that $(e_t)_\sharp \lambda^{(a)} = \tilde{\nu}_t^{(a)}$ for all $t\in[0,T]$. Then, the following lemma concludes the proof.
\begin{lemma}
    The measure
    \begin{equation}
        \lambda := \sum_{a\in \mathcal{A}} aN_a \lambda^{(a)} = \sum_{i\in \N} a_i(\Bar{t}) \delta_{\gamma_i} \in \PPpa(C([0,T],\R^d))
    \end{equation}
    satisfies $(\e_t)_\#\lambda = \mu_t$ for all $t\in [0,T]$ and is concentrated on $\gamma \in AC([0,T],\R^d)$ that solves $\dot{\gamma}(t) = v_t(\gamma(t))$. 
\end{lemma}

\begin{proof}
    For all $t\in[0,T]$ and $g:\R^d\to [0,1] $ Borel measurable, it holds
    \begin{align*}
        \int f(\gamma(t)) d\lambda(\gamma) = \sum_{a\in \mathcal{A}} aN_a \int f(\gamma(t)) d\lambda^{(a)}(\gamma) = \sum_{a\in \mathcal{A}} a N_a \int f(x) d\tilde{\nu}_t^{(a)}(x) = \int f(x)d\mu_t(x),
    \end{align*}
    and in particular
    \[\int \int_0^T |v_t(\gamma(t))| dt d\lambda(\gamma) = \int_0^T \int | v_t(x)| d\mu_t(x) dt <+\infty.\]
    Moreover, $\lambda$ is concentrated on the collection of curves $\{\gamma_i : i\in \N\}$, that are absolutely continuous and are solutions of the ODE associated to $v$.    
\end{proof}

Afterwards, we can assess that $I_{\boldsymbol{\mu}} = [0,T]$, but not that $J_{\boldsymbol{\mu}} =[0,T]$. Indeed, $J_{\boldsymbol{\mu}}$ is characterized by \eqref{eq: j_mu} and it may happen that for some $a\in \mathcal{A}$ and $i,j\in S_a$, the curves $\gamma_i$ and $\gamma_j$ cross at a time $t\in[0,T]\setminus J_{\boldsymbol{\mu}}$, raising a singularity for $(\mu_t)_{t\in[0,T]}$ with respect to the atomic topology.

\section{On the capacity of the diagonal}\label{section: capacity}
We will analyze closer the hypothesis (3) in Theorem \ref{superposition in duality with gen cyl functions}. In particular, we will see sufficient conditions on $\nu$ that imply it, and how it is connected to geometric properties of the measure. Moreover, we show what we can deduce about the evolution of random measures under stronger assumptions.

\subsection{Measure of an \texorpdfstring{$\varepsilon$}{}-strip around the diagonal}
In the following, for all $\varepsilon$ we denote by $\Delta_\varepsilon\subset \R^d \times \R^d$ the $\varepsilon$-strip around the diagonal $\Delta$, that is
\begin{equation}
    \Delta_\varepsilon := \{(x,y) \in \R^d \times \R^d \ : \ |x-y|<\varepsilon\}.
\end{equation}

\begin{prop}\label{prop: strip}
    Let $\nu\in \PP(\R^d)$ be an atomless measure. Assume that, for some $\Tilde{r}\in [1,+\infty)$, it holds
    \begin{equation}\label{eq: strip rate}
        \nu\otimes \nu (\Delta_\varepsilon)\leq C \varepsilon^{\Tilde{r}},
    \end{equation}
    for a constant $C>0$ independent of $\varepsilon>0$. Then, 
    \begin{equation}
        \operatorname{cap}_{r,\nu}(\Delta) = 0 \quad \text{for all }r<\Tilde{r}.
    \end{equation}
\end{prop}

\begin{proof}
    Let $\rho\in C_c^1(\R^{2d})$ be a standard mollifier, and $\rho_\varepsilon(\cdot):= \frac{1}{\varepsilon^{2d}}\rho(\cdot/\varepsilon)$ for all $\varepsilon>0$. Then, define 
    \[h_\varepsilon(x,y):= \mathds{1}_{\Delta_{2\varepsilon}}* \rho_{\varepsilon/2}(x).\]
    It is immediate to verify that $0\leq h_\varepsilon \leq 1$, $h_\varepsilon(x,y) = 1$ for all $(x,y) \in \Delta_{\varepsilon}$ and $h_\varepsilon(x,y) = 0$ for all $(x,y)\notin \Delta_{3\varepsilon}$, which implies that $\nabla h_\varepsilon(x,y) = 0$ if $(x,y)\in \Delta_\varepsilon\cup \Delta_{3\varepsilon}^c$. On the other hand, if $(x,y) \in \Delta_{3\varepsilon}\setminus \Delta_\varepsilon$, it holds
    \begin{align*}
        |\nabla h_\varepsilon(x,y)| = & \left| \int_{\R^{2d}} \mathds{1}_{\Delta_{2\varepsilon}}(x-z,y-w) \nabla \rho_{\varepsilon/2}(z,w) dz dw \right|
        \\
        \leq & \frac{2^{2d+1}}{\varepsilon^{2d+1}}\int_{\{|(z,w)|\leq \varepsilon/2\}}\mathds{1}_{\Delta_{2\varepsilon}}(x-z,y-w) \left|\nabla\rho \left( \frac{2z}{\varepsilon},\frac{2w}{\varepsilon} \right) \right| dzdw
        \\
        \leq &
        \frac{2^{2d+1}\|\nabla\rho\|_\infty}{\varepsilon^{2d+1}} \mathcal{L}^{2d}\big( \{(z,w) \ : \ |(z,w)|\leq \varepsilon/2\} \big) \leq C(d)\varepsilon^{-1}.
    \end{align*}
    Then, for all $r<\Tilde{r}$, it holds
    \begin{align*}
        \int_{\R^{2d}} & |h_\varepsilon(x,y)|^r + |\nabla h_\varepsilon(x,y)|^r d\nu\otimes \nu(x,y) \leq \nu\otimes \nu(\Delta_{3\varepsilon}) + \int_{\Delta_{3\varepsilon}} |\nabla h_\varepsilon(x,y)|^r d\nu\otimes\nu(x,y)
        \\
        \leq & 
        C\varepsilon^{\Tilde{r}} + \frac{C(d)^r}{\varepsilon^r} \nu\otimes \nu(\Delta_{3\varepsilon}) 
        \leq \big(C + 3^{\Tilde{r}}C^{\Tilde{r}}C(d)^r\big) \left(\varepsilon^r + \varepsilon^{\Tilde{r}-r}\right) \to 0.
        \qedhere
    \end{align*}
\end{proof}

\noindent Some corollaries immediately follow from the previous Proposition.

\begin{co}
    Let $\nu\in \PP(\R^d)$ be an atomless measure and $\Tilde{r}\in [1,+\infty)$, be such that 
    \begin{equation}\label{eq: small ball rate}
        \sup_{x\in \R^d} \nu\big( B(x,\varepsilon) \big) \leq C \varepsilon^{\Tilde{r}},
    \end{equation}
    for some $C>0$. Then $\nu\otimes \nu(\Delta_\varepsilon) \leq C\varepsilon^{\tilde{r}}$, and in particular $\operatorname{cap}_{r,\nu}(\Delta) = 0$ for all $r<\Tilde{r}$.
\end{co}

\begin{co}
    The conclusions of Theorem \ref{main theorem} and Theorem \ref{main theorem metric} hold if we replace assumptions (3) and (4), respectively, by 
    \begin{itemize}
        \item[(3')] $\nu\otimes \nu(\Delta_\varepsilon)\leq C\varepsilon^{\Tilde{r}}$ for all $\varepsilon>0$;
        \item[(4')] $\frac{p'}{\Tilde{r}} + \frac{1}{q} <1$.
    \end{itemize}
\end{co}

We can exploit Proposition \ref{prop: strip} to find examples of $\nu$ and $r\geq 1$ that satisfy $\operatorname{cap}_{r,\nu} = 0$. In fact, we see how $\Tilde{r}$ for which \eqref{eq: strip rate} is valid, is related to the dimension of the measure $\nu$.

\begin{co}
    Let $\nu\in \PP(\R^d)$ and assume that $\nu\ll\mathcal{L}^d$, with density $f \in L^p(\R^d,\mathcal{L}^d)$. Then:
    \begin{itemize}
        \item[(i)] if $p\geq 2$, then \eqref{eq: strip rate} holds with $\Tilde{r} = d$;
        \item[(ii)] if $p\in (1,2)$, then \eqref{eq: strip rate} holds with $\Tilde{r} = \frac{2d}{p'}$.
    \end{itemize}
\end{co}

\begin{proof}
    Let $\Tilde{f}(x):= f(-x)$ for all $x\in \R^d$. For all $\varepsilon>0$ it holds
    \begin{align*}
        \nu\otimes \nu(\Delta_\varepsilon) = &\int_{\R^d} \nu(B(x,\varepsilon)) d\nu(x) 
         = \int_{\R^d} f(x) \left(\int_{B(x,\varepsilon)} f(y) dy\right) dx
         \\
         = & 
         \int_{\R^d} f(x) \left(\int_{B(0,\varepsilon)} f(x-y) dy \right) dx  = \int_{B(0,\varepsilon)} \left( \int_{\R^d} f(x-y) f(x) dx \right) dy 
         \\
         = & \int_{B(0,\varepsilon)} \Tilde{f}*f(y) dy.
    \end{align*}
    Then, by Young inequality and the fact that $f,\tilde{f}\in L^1(\mathcal{L}^d) \cap L^p(\mathcal{L}^d)$, it holds:
    \begin{itemize}
        \item[(i)] if $p\geq 2$, then $p'\in [1,p]$ and $\tilde{f} \in L^{p'}$, which gives $\tilde{f}*f \in L^{\infty}$ and $\nu\otimes \nu(\Delta_\varepsilon) \leq \omega_d\|\tilde{f}*f\|_{L^\infty} \varepsilon^d$;
        \item[(ii)] if $p\in (1,2)$, then $\tilde{f}*f\in L^{q}(\mathcal{L}^d)$ with $q>1$ satisfying $\frac{1}{q}+1 = \frac{2}{p}$, that is $q = \frac{p}{2-p}$. Then, by H\"older inequality, we have 
        \[\nu\otimes \nu (\Delta_\varepsilon) \leq \omega_d^{1/q'}\|\tilde{f}*f\|_{L^q} \varepsilon^{\frac{d}{q'}},\]
        and we can verify that $\frac{d}{q'} = \frac{2d}{p'}$. \qedhere
    \end{itemize}
\end{proof}

A different technique shows that, in case $\nu\ll\mathcal{L}^d$ has a bounded density, then $\operatorname{cap}_{d,\nu}(\Delta) = 0$. Before stating the result, let us introduce the `logarithmic-type' test functions that will make the job.

Consider $\eta \in C^{\infty}(\R)$ non-increasing and such that $\eta(x) = 1$ for all $x\leq 0$, $\eta(x) = 0$ for all $x\geq 1$ and $|\eta'(x)|\leq 2$ for all $x\in(0,1)$. Then, for all $0<\varepsilon<R<+\infty$ define
\begin{equation}\label{eq: log cut-off}
    f_{\varepsilon,R}(r) := \eta\left( \frac{\log r - \log \varepsilon}{\log R - \log \varepsilon} \right) \in C^{\infty}((0,+\infty)), \quad h_{\varepsilon,R}(x,y) := f_{\varepsilon,R}(|x-y|).
\end{equation}

The functions $h_{\varepsilon,R}$ are well-defined and smooth in $\R^d\times \R^d\setminus \Delta$. However, since $\eta\equiv 1$ on negative numbers, we have that $h_{\varepsilon,R} \equiv 1$ in $\Delta_{\varepsilon}\setminus \Delta$, and in particular it can be smoothly extended in the whole $\R^d\times\R^d$ simply putting $h_{\varepsilon,R}(x,x) = 1$ for all $x\in \R^d$.

Now, $\eta$ takes values in $[0,1]$, and its properties give $f_{\varepsilon,R}(r) = 1$ for all $r\leq \varepsilon$ and $f_{\varepsilon,R}(r) = 0$ for all $r\geq R$. Then, for all $(x,y)\in \R^d\times \R^d$ 
\begin{equation}\label{eq: estimates log cut-off}
    0 \leq h_{\varepsilon,R}(x,y) \leq 1, \quad |\nabla h_{\varepsilon,R}(x,y)| =\eta'\left( \frac{\log \frac{r}{\varepsilon}}{\log \frac{R}{\varepsilon}} \right) \frac{|\nabla_{x,y}|x-y||}{|x-y|\log \frac{R}{\varepsilon}} \leq \frac{2\sqrt{2}}{|x-y|\log \frac{R}{\varepsilon}} \mathds{1}_{(\varepsilon,R)}(|x-y|).
\end{equation}

\begin{prop}\label{prop: sharp capacity in R^d}
    Let $\nu\in \PP(\R^d)$ and assume that $\nu\ll\mathcal{L}^d$, with density $f \in L^\infty(\R^d,\mathcal{L}^d)$. Then, if $d\geq 2$, $\operatorname{cap}_{d,\nu}(\Delta) = 0$.
\end{prop}

\begin{proof}
    For all $0<\varepsilon<R<+\infty$, the function $h_{\varepsilon,R}$ is a competitor for the minimum problem defining $\operatorname{cap}_{d,\nu}(\Delta)$. We need to prove that 
    \[\int_{\R^d \times \R^d}|h_{\varepsilon,R}(x,y)|^d + |\nabla h_{\varepsilon,R}(x,y)|^d d\nu\otimes \nu(x,y) \to 0,\]
    carefully choosing a way for letting $\varepsilon$ and $R$ go to $0$. The first term goes to $0$, as $h_{\varepsilon,R}(x,y) = 0$ outside $\Delta_R$ and is bounded by $1$, so that
    \[\int_{\R^d \times \R^d} |h_{\varepsilon,R}(x,y)|^d d\nu\otimes \nu(x,y) \leq \nu\otimes\nu(\Delta_R) \to 0, \quad \text{as }R\to 0.\]
    On the other hand, thanks to \eqref{eq: estimates log cut-off} 
    \begin{align*}
        \int_{\R^d \times \R^d} & |\nabla h_{\varepsilon,R}(x,y)|^d d\nu\otimes \nu(x,y) \leq \frac{(2\sqrt{2})^d}{\left(\log \frac{R}{\varepsilon}\right)^d} \int_{\R^d} \int_{\{x \in \R^d : |x-y|\in (\varepsilon,R)\}} \frac{1}{|x-y|^d} f(x)dx d\nu(y)
        \\
        \leq & \frac{(2\sqrt{2})^d\|f\|_\infty}{\left(\log \frac{R}{\varepsilon}\right)^d} \int_{\{z\in \R^d : |z|\in (\varepsilon,R)\}} \frac{1}{|z|^d}dz = \frac{d\omega_d (2\sqrt{2})^d\|f\|_\infty}{\left(\log \frac{R}{\varepsilon}\right)^d}  \int_\varepsilon^R \frac{1}{r}dr = \frac{d\omega_d (2\sqrt{2})^d\|f\|_\infty}{\left(\log \frac{R}{\varepsilon}\right)^{d-1}}, 
    \end{align*}
    where $\omega_d$ is the volume of the unit ball in $\R^d$. Choosing $R=\sqrt{\varepsilon}$ and allowing $\varepsilon\to0$, the proof is concluded.
\end{proof}

\subsection{Absolutely continuous curves of random measures with respect to a new metric}
In this subsection, we will show that we can obtain further information on the evolution of random measures $(M_t)_{t\in[0,T]} \in C([0,T],\PP(\PP(\R^d)))$ that satisfy a continuity equation, whenever $M_t \ll \Q_{\pi,\nu}$ and we have control of small strips around the diagonal as in \eqref{eq: strip rate}.
This information is encoded by a natural distance over random measures, similar to \eqref{eq: atomic wass distance}.

\begin{df}\label{def: atomic wass on wass}
    Let $\Psi:[0,+\infty)\to[0,1]$ be a continuous and non-increasing function with compact support, satisfying $\Psi(0) =1 $. Then, for all $M,N \in \PP_p(\PP_p(\R^d))$ we define
    \begin{equation}
    \begin{aligned}
    \mathcal{D}_{p,\Psi}(M,N) := \mathcal{W}_p(M,N) +
    \sup_{\varepsilon\in(0,1)} \left| \int F_{\varepsilon,\Psi}(\mu) dM(\mu) - \int F_{\varepsilon,\Psi}(\mu) dN(\mu) \right|
    \end{aligned}
    \end{equation}
    where
    \begin{equation}
        F_{\varepsilon,\Psi}(\mu):=  \int_{\R^{2d}} \Psi\left( \frac{|x-y|}{\varepsilon} \right) d\mu\otimes \mu(x,y).
    \end{equation}
\end{df}

\begin{lemma}
    The metric space $(\PP_p(\PP_p(\R^d)),\mathcal{D}_{p,\Psi})$ is complete and separable.
\end{lemma}
\begin{proof}
    It is immediate to see that $\mathcal{D}_{p,\Psi} $ is a distance. To show separability, we compare it with the distance $W_{p,D_{p,|\cdot|,\Psi}}$, i.e. the $p$-Wasserstein distance defined over the distance $D_{p,|\cdot|,\Psi}$, which induces a separable topology. In fact, let $M,N\in \PP_p(\PP_p(\R^d))$ and consider $\Pi\in \Gamma(M,N)$ be optimal for $W_{p,D_{p,|\cdot|,\Psi}}$, so that 
    \begin{align*}
        \mathcal{D}_{p,\Psi}^p(M,N) \leq & 2^p\left( \mathcal{W}_p^p(M,N) + \sup_{\varepsilon\in(0,1)} \left|\int F_{\varepsilon,\Psi}(\mu) dM(\mu) - \int F_{\varepsilon,\Psi}(\nu) dN(\nu)\right|^p  \right)
        \\
        \leq & 2^p\left( \int W_p^p(\mu,\nu) d\Pi(\mu,\nu) + \sup_{\varepsilon\in(0,1)} \int \left|F_{\varepsilon,\Psi}(\mu) - F_{\varepsilon,\Psi}(\nu) \right|^p d\Pi(\mu,\nu) \right)
        \\
        \leq & 2^p\left( \int W_p^p(\mu,\nu) d\Pi(\mu,\nu) + \int \sup_{\varepsilon\in(0,1)}  \left|F_{\varepsilon,\Psi}(\mu) - F_{\varepsilon,\Psi}(\nu) \right|^p d\Pi(\mu,\nu) \right)
        \\
        \leq & 2^p \left( \int D_{p,|\cdot|,\Psi}^p d\Pi(\mu,\nu) \right) = 2^p W_{p,D_{p,|\cdot|,\Psi}}^p(M,N).
    \end{align*}
    In particular, if $\{M_n\}_{n\in \N} \subset \PP_p(\PP_p(\R^d))$ is dense with respect to $W_{p,D_{p,|\cdot|,\Psi}}$, then it is dense for $\mathcal{D}_{p,\Psi}$ as well.

    Regarding completeness, let $\{M_n\}_{n\in \N} \subset \PP_p(\PP_p(\R^d))$ be a Cauchy sequence with respect to $\mathcal{D}_{p,\Psi}$. Then, it is also a Cauchy sequence with respect to $\mathcal{W}_p$, so that there exists $M\in \PP_p(\PP_p(\R^d))$ for which $\mathcal{W}_p(M_n,M)\to 0$. To show that $\mathcal{D}_{p,\Psi}(M_n,M)\to 0$ we need to prove that 
    \[\alpha_n:= \sup_{\varepsilon\in(0,1)} \left| \int F_{\varepsilon,\Psi}(\mu) dM_n(\mu) - \int F_{\varepsilon,\Psi}(\nu) dM(\nu) \right| \to 0.\]
    First, notice that the function $F_\Psi$ is narrowly continuous, so that for all $\varepsilon\in(0,1)$, $\int F_{\varepsilon,\Psi}dM_n \to \int F_{\varepsilon,\Psi} dM$. Then, for all $n\in \N$ consider $\varepsilon_n\in(0,1)$ such that 
    \[\alpha_n \leq \left| \int F_{\varepsilon_n,\Psi}(\mu) dM_n(\mu) - \int F_{\varepsilon_n,\Psi}(\nu) dM(\nu) \right| + \frac{1}{n}.\]
    Then we conclude thanks to the following inequalities:
    \begin{align*}
        \limsup_{n\to +\infty} \alpha_n \leq & \limsup_{n\to+\infty} \left| \int F_{\varepsilon_n,\Psi}(\mu) dM_n(\mu) - \int F_{\varepsilon_n,\Psi}(\nu) dM(\nu) \right|
        \\
        = & \limsup_{n\to +\infty} \lim_{m\to+\infty} \left| \int F_{\varepsilon_n,\Psi}(\mu) dM_n(\mu) - \int F_{\varepsilon_n,\Psi}(\nu) dM_m(\nu) \right|
        \\
        \leq & \limsup_{n\to +\infty} \limsup_{m\to+\infty} \mathcal{D}_{p,\Psi}(M_n,M_m) = 0. \qedhere
    \end{align*}
\end{proof}

Notice that the topology $\tau_{p,\Psi}$ induced by $\mathcal{D}_{p,\Psi}$ is between the narrow over narrow topology $\tau_{nn}$ and the narrow over atomic topology $\tau_{na}$. We know $\tau_{nn} \subsetneq \tau_{na}$, but it is not clear whether $\tau_{p,\Psi}$ coincides with one of them.
In the following, we will again fix a reference measure $\Q = \Q_{\pi,\nu} \in \mathcal{Q}$.

\begin{teorema}
	Let $b:[0,T]\times \R^d \times \PP(\R^d)\to \R^d$ be a Borel vector field and $\boldsymbol{M}=(M_t)_{t\in[0,T]}\in C_T(\PP(\PP(\R^d)))$ be satisfying $\frac{d}{dt}M_t + \operatorname{div}_\PP (b_t M_t)=0$ and $M_0 \in \PP_p(\PP_p(\R^d))$.
	Assume that 
	\begin{enumerate}
		\item $M_t = u_t \Q$, with $u \in L^\infty_t(L^q(\Q))$, with $q\in[1,+\infty]$. In other words there exists $K\in (0,+\infty)$ such that
		\begin{equation}
			 \|u_t\|_{L^{q}(\Q)} \leq K \quad \text{ for a.e. }t\in(0,T);
		\end{equation}
		\item $\|b\|_{L^p(\widetilde{M}_t\otimes dt)}^p=\int_0^T \int_{\PP(\R^d)} \int_{\R^d} |b_t(x,\mu)|^p d\mu(x) dM_t(\mu) dt<+\infty$, with $p\in[1,+\infty]$;
		\item $\nu\otimes \nu(\Delta_\varepsilon) \leq C \varepsilon^{\tilde{r}}$ for some $\tilde{r}\geq 1$;
		\item $\frac{p'}{\tilde{r}}+\frac{1}{q}\leq1$.
	\end{enumerate}
	Further assuming that $\Psi\in C_c^1([0,\infty),[0,1])$ and $\Psi'(0) = 0$, we have that $\boldsymbol{M}$ is $p$-absolutely continuous with respect to $\mathcal{D}_{p,\Psi}$. 
\end{teorema}

\begin{proof}
	Thanks to \cite[Proposition 4.10]{pinzisavare2025}, the curve $\boldsymbol{M}$ is $p$-absolutely continuous with respect to $\mathcal{W}_p$. We will be done if we show that there exists $g\in L^p(0,T)$ such that for all $0\leq t_1 < t_2 \leq T$ it holds
	\begin{equation}\label{eq: AC D}
		\sup_{\varepsilon\in(0,1)} \left| \int F_{\varepsilon,\Psi}(\mu) dM_{t_1}(\mu) - \int F_{\varepsilon,\Psi}(\mu) dM_{t_2}(\mu) \right| \leq \int_{t_1}^{t_2} g(s) ds.
	\end{equation}
	To this aim, notice that $F_{\varepsilon,\Psi}$ is an admissible test function for all $\varepsilon\in(0,1)$ because of Lemma \ref{interaction cyl function}, for which it holds
	\[\frac{d}{dt} \int F_{\varepsilon,\Psi}(\mu) dM_t(\mu) = \int \int_{\R^d} \nabla_WF_{\varepsilon,\Psi}(x,\mu) \cdot b_t(x,\mu) d\mu(x)dM_t(\mu)\]
	in the sense of distribution, where 
	\[\nabla_WF_{\varepsilon,\Psi}(x,\mu) = \frac{2}{\varepsilon}\int_{\R^d} \Psi'\left(\frac{|x-y|}{\varepsilon}\right) \frac{x-y}{|x-y|} d\mu(y),\]
	that is well-defined since $\Psi'(0) = 0$. Then, consider
	\[ \tilde{p}:= \frac{\tilde{r}}{\tilde{r}-q'}, \quad \text{so that} \quad \frac{\tilde{p}'}{\tilde{r}} + \frac{1}{q} = 1.\]
	Notice that $1\leq \tilde{p} \leq p$ and that $\tilde{p} = p$ if and only if (4) holds with equality. Moreover, let $R>0$ be such that $\operatorname{supp}\Psi \subset [0,R]$. Then, for all $\varepsilon\in(0,1)$, similar computations as in the proof of Theorem \ref{superposition in duality with gen cyl functions} give    \begin{align*}
		\bigg| \int & F_{\varepsilon,\Psi}(\mu) dM_{t_1}(\mu) - \int F_{\varepsilon,\Psi}(\mu) dM_{t_2}(\mu) \bigg| \leq \int_{t_1}^{t_2} \int \int_{\R^d} |b_t(x,\mu)| |\nabla_W F_{\varepsilon,\Psi}(x,\mu)|d\mu(x)dM_{t}(\mu) dt
		\\
		\leq & \int_{t_1}^{t_2} \|b_t(\cdot,\cdot)\|_{L^{\tilde{p}}(\widetilde{M}_t)} \left( \int \int |\nabla_W F(x,\mu)|^{\tilde{p}'} d\mu(x) u_{t}(\mu)d\Q(\mu) \right)^{\frac{1}{\tilde{p}'}} dt
		\\
		\leq & \int_{t_1}^{t_2} \|b_t(\cdot,\cdot)\|_{L^{\tilde{p}}(\widetilde{M}_t)} \|u_t\|_{L^q(Q)}^{\frac{1}{\tilde{p}'}} \left( \int \int |\nabla_W F(x,\mu)|^{\tilde{p}'q'} d\mu(x) d\Q(\mu) \right)^{\frac{1}{\tilde{p}'q'}} dt
		\\
		\leq & 2 K^{\frac{1}{\tilde{p}'}} \left( \int_{t_1}^{t_2} \|b_t(\cdot,\cdot)\|_{L^{\tilde{p}}(\widetilde{M}_t)} dt \right)  \left( \int \int_{\R^{2d}} \frac{1}{\varepsilon^{\tilde{p}'q'}} \left| \Psi'\left(\frac{|x-y|}{\varepsilon}\right) \right|^{\tilde{p}'q'} d\mu\otimes \mu(x,y) d\Q(\mu) \right)^{\frac{1}{\tilde{p}'q'}}.
	\end{align*}
	If we prove that the last factor is uniformly bounded with respect to $\varepsilon>0$, then we are done taking the supremum on the left hand side in $\varepsilon>0$ and noticing that in \eqref{eq: AC D}, we can use $g(t):= \|b_t(\cdot,\cdot) \|_{L^{\tilde{p}}(\widetilde{M}_t)}$, which is $p$-integrable since, from hypothesis \textit{(2)} we have $\|b_t(\cdot,\cdot) \|_{L^{\tilde{p}}(\widetilde{M}_t)}\leq \|b_t(\cdot,\cdot)\|_{L^p(\widetilde{M}_t)} \in L^p(0,T)$.
	
	From Lemma \ref{lemma: barycenter of mu times mu}, the fact that $\Psi'(0) = 0$ and that $\Psi'(r/\varepsilon) = 0$ if $r \geq \varepsilon R$, we have that
	\begin{align*}
		\int & \int_{\R^{2d}}  \frac{1}{\varepsilon^{\tilde{p}'q'}} \left| \Psi'\left(\frac{|x-y|}{\varepsilon}\right) \right|^{\tilde{p}'q'} d\mu\otimes \mu(x,y) d\Q(\mu) = c_1 \int_{\R^{2d}} \frac{1}{\varepsilon^{\tilde{p}'q'}} \left| \Psi'\left( \frac{|x-y|}{\varepsilon} \right) \right|^{\tilde{p}'q'} d\nu\otimes \nu(x,y)
		\\
		= & c_1 \int_{\Delta_{\varepsilon R}} \frac{1}{\varepsilon^{\tilde{p}'q'}} \left| \Psi'\left(\frac{|x-y|}{\varepsilon}\right) \right|^{\tilde{p}'q'} d\nu\otimes \nu(x,y) \leq c_1 \|\Psi'\|_{C^0}^{\tilde{p}'q'} \frac{\nu\otimes\nu (\Delta_{\varepsilon R}) }{\varepsilon^{\tilde{p}'q'}}  \leq  c_1 C \|\Psi'\|_{C^0}^{\tilde{p}'q'}  R^{\tilde{r}}\varepsilon^{\tilde{r}- \tilde{p}'q'},
	\end{align*}
	thanks to which we conclude since $\tilde{r} = \tilde{p}'q'$ by definition of $\tilde{p}$.
\end{proof}

\section{Nested superposition principles on Riemannian manifolds}\label{sec: manifold}
In this subsection, we establish a version of Theorems \ref{thm: nested superposition}, \ref{main theorem} and \ref{main theorem metric} in a Riemannian setting. For the rest of the section, if not differently specified, $(\mathbb{S},\ttg)$ will be a compact boundaryless and oriented Riemannian manifold of dimension $k \geq 2$, with induced distance and volume measure denoted as $d_{\mathbb{S}}$ and $\operatorname{vol}_{\mathbb{S}}$. Moreover, we will use the same construction of reference random measures showed in §\ref{subsec: reference}: fix $\pi \in \PP(\bold{T})$ and $\nu\in \PP(\mathbb{S})$ an atomless probability measure, and denote $Q = Q_{\pi,\nu} \in \PP(\PP(\Ss))$.

The general idea is to use the Nash embedding theorem for transporting all the notions over $\mathbb{S}$ in $\R^d$ and then use the superposition principles in the Euclidean setting. To this end, here we collect the notation and the main properties of the Nash embedding theorem \cite{nash1954c}: there exists a smooth map (the Nash isometry) 
\begin{equation}\label{eq: nash isometry}
    \jmath: \mathbb{S} \to \R^d
\end{equation}
satisfying the following properties.
\begin{enumerate}
    \item It is a local isometry, and in particular for all $y\in \Ss$ and $v,w \in T_x\Ss$,
    \begin{equation}
        \ttg_{\operatorname{Eucl}}\big( d_y \jmath (v) , d_y \jmath(w) \big) = \ttg_{y}(v,w),
    \end{equation}
    where $\ttg_{\operatorname{Eucl}}$ is the standard Euclidean Riemannian metric on $\R^d$ and $d_y \jmath: T_y \Ss \to T_{\jmath(x)}S\subset \R^d$ is the differential operator of the map $\jmath$ in the point $y\in \Ss$.
    \item Its image $S := \jmath(\mathbb{S}) \subset \R^d$ is compact. Moreover, it inherits the differential and Riemannian structure from $\mathbb{S}$ and $\jmath^{-1}: S \to \mathbb{S}$ is well-defined and smooth;
    \item The intrinsic distance of $S$ as a subset of $\R^d$ is defined as 
    \begin{equation}
        d_{S}(x_0,x_1) := \inf\{\operatorname{length}(\gamma) \ : \ \gamma \in C([0,1],S), \ \gamma(0) = x_0, \ \gamma(1) = x_1\},
    \end{equation}
    where the length is evaluated looking at $\gamma$ as a curve in $\R^d$. The space $(S,d_S)$ is Polish and isometric to $(\mathbb{S},d_{\mathbb{S}})$ through $\jmath$. Moreover, there exists a constant $C\geq 1$ such that 
    \begin{equation}\label{eq: equivalence}
        |x-y| \leq d_S(x,y)\leq C|x-y| \quad \forall x,y \in S;
    \end{equation}
    \item The volume measure $\operatorname{vol}_{\mathbb{S}}$ is transported to $\operatorname{vol}_{S} = \mathcal{H}^k|_{S}$, that is 
    \begin{equation}\label{eq: push for volume}
        \jmath_{\sharp} \operatorname{vol}_{\mathbb{S}} = \mathcal{H}^k|_{S}.
    \end{equation}
    In particular, if $\nu\ll \operatorname{vol}_\Ss$, then $\tilde{\nu}:= \jmath_\sharp \nu$ is absolutely continuous w.r.t. $\mathcal{H}^k|_{S}$, with density given by $\frac{d\nu}{d\operatorname{vol}_{\mathbb{S}}}\circ \jmath^{-1}$;
    \item For all $p\in[1,+\infty]$, the spaces $(\PP_p(\mathbb{S}),W_{p,d_{\mathbb{S}}})$ and $(\PP_p(S),W_{p,d_S})$ are isometric through the push forward function $\boldsymbol{\jmath}:= \jmath_{\sharp}$, and $\boldsymbol{\jmath}^{-1} = \jmath^{-1}_\sharp$. Similarly, for the spaces of random measures $((\PP_p(\PP_p(\mathbb{S})),\mathcal{W}_{p,d_{\mathbb{S}}}))$ and $(\PP_p(\PP_p(S)),\mathcal{W}_{p,d_S})$ through $\boldsymbol{\jmath}_\sharp = \jmath_{\sharp\sharp}$. Moreover $\PP_p(S)$ and $\PP_p(\PP_p(S))$ are, respectively, embedded in $\PP_p(\R^d)$ and $\PP_p(\PP_p(\R^d))$, and \eqref{eq: equivalence} implies that 
    \begin{equation}
        \begin{gathered}
        W_p(\mu,\nu)\leq W_{p,d_{S}}(\mu,\nu) \leq C W_{p}(\mu,\nu) \quad \forall \mu,\nu\in \PP_p(S),
        \\
        \mathcal{W}_{p}(M,N) \leq \mathcal{W}_{p,d_S}(M,N)\leq C\mathcal{W}_{p}(M,N) \quad \forall M,N \in \PP_p(\PP_p(S));
        \end{gathered}
    \end{equation}
    \item The reference measure $Q_{\pi,\nu}$ is transported as
    \begin{equation}\label{eq: nash ref measure}
        \jmath_{\sharp\sharp}Q_{\pi,\nu} = Q_{\pi,\jmath_\sharp \nu}. 
    \end{equation}
\end{enumerate}

\subsection{Continuity equation and nested superposition principle for random measures on compact Riemannian manifolds}

Let $b:[0,T]\times \Ss \times \PP(\Ss) \to T \Ss$ be a Borel vector field (with the natural topology over the tangent bundle $T\Ss$), i.e. such that $b(t,y,\nu) \in T_y \Ss$ for any $(t,y,\nu) \in [0,T]\times \Ss \times \PP(\Ss)$.

\begin{df}\label{ce manifold}
    Let $\boldsymbol{\mu} = (\mu_t)_{t\in [0,T]} \in C([0,T],\PP(\Ss))$. We say that the continuity equation $\partial_t\mu_t + \operatorname{div}_{\Ss}(b_t(\cdot,\mu_t) \mu_t) = 0$ is satisfied if 
    \begin{equation}
        \int_0^T \int_\Ss \lVert b(t,y,\mu_t) \rVert_{\ttg_{y}} d\mu_t(y) dt <+\infty
    \end{equation}
    and for all $\xi\in C_c^1(0,T)$ and $\phi\in C^1(\Ss)$
    \begin{equation}\label{weak ce manifold}
        \int_0^T \xi'(t) \int_{\Ss} \phi(y) d\mu_t(y) dt = -\int_0^T \xi(t) \int_\Ss d_y\phi(b(t,y,\mu_t)) d\mu_t(y) dt.
    \end{equation}
\end{df}

As for the Euclidean case, we need to introduce the cylinder functions.

\begin{df}[Cylinder functions and Wasserstein differential]\label{def cyl func manifold and wass grad}
    A functional $F:\PP(\Ss)\to\R$ is called a cylinder function if there exists $k\in \N$, $\Phi = (\phi_1,\dots,\phi_k)\in C^1(\Ss;\R^k)$ and $\Psi\in C^1(\R^k)$ such that 
    \begin{equation}\label{cyl functions manifold}
        F(\mu) = \Psi\left( L_\Phi(\mu) \right), \quad L_\Phi(\mu) =  \big(L_{\phi_1}(\mu) ,\dots , L_{\phi_k}(\mu)\big), \quad  L_{\phi_i}(\mu) := \int_{\Ss} \phi_i(y) d\mu(y).
    \end{equation}
    Its Wasserstein differential is then defined as 
    \begin{equation}\label{wass gradient cyl manifold}
        D_W F(y,\mu):= \sum_{i=1}^k \partial_{i}\Psi\left( L_{\Phi}(\mu)\right) d_y \phi_i \quad \forall y\in \Ss, \ \forall\mu \in \PP(\Ss).
    \end{equation}
    The collection of all the cylinder functions its called $\operatorname{Cyl}^1(\PP(\Ss))$. If $\Psi \in C^\infty(\R^k)$ and $\phi_1,\dots,\phi_k \in C^\infty(\Ss)$, we write that $F\in \operatorname{Cyl}^\infty(\PP(\Ss))$. 
\end{df}

Similarly, we may define the sets of generalized cylinder functions $\operatorname{GC}_c^1(\PP(\Ss))$ and $\operatorname{GC}_c^\infty(\PP(\Ss))$: we say that $\hat{F} \in \operatorname{GC}_c^1(\PP(\Ss))$ (resp. $\operatorname{GC}_c^\infty(\PP(\Ss))$) is there exist $k\geq 1$, $\Psi \in C^1(\R^k)$ (resp. $C^\infty(\R^k)$) and $\hat{\Phi} = (\hat{\phi}_1,\dots,\hat{\phi}_k)\in C^1_c(\Ss\times(0,1];\R^k)$ (resp. $C_c^\infty(\Ss\times(0,1];\R^k)$) such that 
    \begin{equation}\label{gen cyl functions manifold}
        \hat{F}(\mu) = \Psi\left( \hat{L}_{\hat{\Phi}}(\mu) \right), \quad\hat{L}_{\hat{\Phi}}(\mu) =  \big(\hat{L}_{\hat{\phi}_1}(\mu) ,\dots , \hat{L}_{\hat{\phi}_k}(\mu)\big), \quad  \hat{L}_{\hat{\phi}_i}(\mu) := \int_{\Ss} \hat{\phi}_i(y,\mu[y]) d\mu(y).
    \end{equation}
Their Wasserstein differential is then defined as 
\begin{equation}\label{wass gradient gen cyl manifold}
        D_W \hat{F}(y,\mu):= \sum_{i=1}^k \partial_{i}\Psi\left( \hat{L}_{\hat{\Phi}}(\mu)\right) d_y \hat{\phi}_i(\cdot,\mu[y]) \quad \forall y\in \Ss, \ \forall\mu \in \PP(\Ss).
    \end{equation}

\begin{df}\label{ce manifold rm}
    Let $\boldsymbol{M}=(M_t)_{t\in [0,T]}\in C([0,T],\PP(\PP(\Ss)))$. We say that the continuity equation $\partial_tM_t + \operatorname{div}_{\PP(\Ss)}(b_t M_t) = 0$ holds, if 
    \begin{equation}
        \int_0^T \int_{\PP(\Ss)}\int_\Ss \lVert b(t,y,\mu) \rVert_{\ttg_{y}} d\mu(y) dM_t(\mu)dt <+\infty
    \end{equation}
    and for all $\xi\in C_c^1(0,T)$ and $F \in \operatorname{Cyl}^1(\PP(\Ss))$
    \begin{equation}\label{weak ce manifold rm}
    \begin{aligned}
        \int_0^T \xi'(t) & \int_{\PP(\Ss)} F(\mu) dM_t(\mu) dt =  -\int_0^T \xi(t) \int_{\PP(\Ss)} \int_\Ss D_W F(y,\mu)\big(b(t,y,\mu)\big) d\mu(y) dM_t(\mu) dt 
        \\
        = & -\int_0^T \xi(t) \int_{\PP(\Ss)}  \sum_{i=1}^k \partial_i \Psi(L_\Phi(\mu)) \int_\Ss d_y\phi_i\big(b(t,y,\mu)\big) d\mu(y) dM_t(\mu) dt.
    \end{aligned}
    \end{equation}
\end{df}

Given a Borel vector field $b:[0,T]\times \Ss \times \PP(\Ss) \to T\Ss$, define the Borel vector field $\Tilde{b}:[0,T]\times \R^d \times \PP(\R^d) \to \R^d$ as 
\begin{equation}\label{def of tilde b from b}
\Tilde{b}(t,x,\tilde{\mu}) = 
\begin{cases}
    d_{\jmath^{-1}(x)}\jmath\big( b(t,\jmath^{-1}(x), \boldsymbol{\jmath}^{-1}(\tilde{\mu})) \big) \quad \quad & \text{if } (x,\tilde{\mu})\in S\times \PP(S)
    \\
    0 & \text{otherwise}.
\end{cases}
\end{equation}

Moreover, given a curve of random measures on $\Ss$, $\boldsymbol{M} = (M_t)_{t\in[0,T]} \in C([0,T],\PP(\PP(\Ss)))$, define 
\begin{equation}\label{eq: def of M_t from N_t}
    \tilde{M}_t := \boldsymbol{\jmath}_\sharp M_t = \jmath_{\sharp\sharp} M_t \in \PP(\PP(S)).
\end{equation}

\begin{prop}\label{propr: tilde M}
Let $b:[0,T]\times \Ss \times \PP(\Ss) \to T\Ss$ be a Borel vector field and $\boldsymbol{M}=(M_t)_{t\in [0,T]}\in C([0,T],\PP(\PP(\Ss)))$ be a solution of the continuity equation $\partial_tM_t + \operatorname{div}_{\PP(\Ss)}(b_tM_t)=0$ in the sense of Def. \ref{ce manifold rm}. 
Then, the curve $(\tilde{M}_t)_{t\in[0,T]}\subset C([0,T],\PP(\PP(\R^d)))$ solves the continuity equation $\partial_t\tilde{M}_t + \operatorname{div}_\PP(\tilde{b}_t \tilde{M}_t) = 0$ in the sense of \eqref{eq: CERM}.
\end{prop}

\begin{proof}
The local isometry property of $\jmath$ implies that, for all $\mu\in \PP(\Ss)$, for $\mu$-a.e. $y$ and for a.e. $t\in[0,T]$, it holds $\|b_t(y,\mu)\|_{\ttg_{y}} =|\tilde{b}_t(\jmath(y), \jmath_\sharp \mu)|$, which gives
\[\int_0^T \int_{\PP(\R^d)} \int_{\R^d} |\tilde{b}(t,x,\tilde{\mu}) | d\tilde{\mu}(x) d\tilde{M}_t(\tilde{\mu}) dt = \int_0^T \int_{\PP(\Ss)} \int_{\Ss} \|b(t,y,\mu)\|d\mu(y) dM_t(\mu) dt<+\infty.\]
Let $\xi\in C_c^1(0,T)$ and $\tilde{F} = \Psi \circ L_{\tilde{\Phi}} \in \operatorname{Cyl}_c^1(\PP(\R^d))$, with $\tilde{\Phi} = (\tilde{\phi}_1,\dots, \tilde{\phi}_k) \in C^1_c(\R^d, \R^k)$ and $\Psi \in C_b^1(\R^k)$, then, exploiting the properties of $\jmath$ and $\boldsymbol{\jmath}$, we have 
\begin{align*}
    \int_0^T &  \xi'(t) \int_{\PP(\R^d)}\tilde{F}(\tilde{\mu}) d\tilde{M}_t(\tilde{\mu}) dt = 
    \int_0^T \xi'(t) \int_{\PP(S)}\tilde{F}(\tilde{\mu}) d\tilde{M}_t(\tilde{\mu}) dt 
    = 
    \int_0^T \xi'(t) \int_{\PP(\Ss)} \tilde{F}(\jmath_\sharp\mu) dM_t(\mu) dt
    \\
    = & 
    \int_0^T \xi'(t) \int_{\PP(\Ss)} \Psi\left( \int_\Ss \tilde{\phi}_1\circ \jmath d\mu, \dots , \int_\Ss \tilde{\phi}_k\circ \jmath d\mu \right) dM_t(\mu) dt
    \\
    = & 
    -\int_0^T \xi(t) \int_{\PP(\Ss)} \sum_{i=1}^k \left( \partial_i \Psi (L_{\tilde{\Phi} \circ \jmath}(\mu))\int_{\Ss} d_y(\Tilde{\phi_i}\circ \jmath)\big(b(t,y,\mu)\big) d\mu(y)\right) dM_t(\mu) dt
    \\
    = & 
    -\int_0^T \xi(t) \int_{\PP(\Ss)} \sum_{i=1}^k \left(\partial_i \Psi (L_{\tilde{\Phi} \circ \jmath}(\mu))\int_{\Ss} d_{\jmath(y)}\Tilde{\phi_i}\circ d_y\jmath\big(b(t,y,\mu)\big) d\mu(y) \right) dM_t(\mu) dt
    \\
    = & 
    -\int_0^T \xi(t) \int_{\PP(S)} \sum_{i=1}^k \left(\partial_i \Psi (L_{\tilde{\Phi}}(\tilde{\mu}))\int_{\Ss} d_{\jmath(y)}\Tilde{\phi_i}\circ d_y\jmath\big(b(t,y,\jmath^{-1}_\sharp\tilde{\mu})\big) d\jmath^{-1}_\sharp\tilde{\mu}(y) \right) d\tilde{M}_t(\tilde\mu) dt
    \\
    = & 
    -\int_0^T \xi(t) \int_{\PP(S)} \sum_{i=1}^k \left(\partial_i \Psi (L_{\tilde{\Phi}}(\tilde{\mu}))\int_{S} d_{x}\Tilde{\phi_i}\circ d_{\jmath^{-1}(x)}\jmath\big(b(t,\jmath^{-1}(x),\jmath^{-1}_\sharp\tilde{\mu})\big) d\tilde\mu(x) \right) d\tilde{M}_t(\tilde{\mu}) dt
    \\
    = & 
    -\int_0^T \xi(t) \int_{\PP(S)} \sum_{i=1}^k \left(\partial_i \Psi (L_{\tilde{\Phi}}(\tilde{\mu}))\int_{S} d_{x}\Tilde{\phi_i}\big(\Tilde{b}(t,x,\tilde{\mu})\big) d\tilde{\mu}(x) \right) d\tilde{M}_t(\tilde{\mu}) dt
    \\
    = & 
    -\int_0^T \xi(t) \int_{\PP(S)} \sum_{i=1}^k \left(\partial_i \Psi (L_{\tilde{\Phi}}(\tilde{\mu}))\int_{S} \nabla\Tilde{\phi_i}\cdot \Tilde{b}(t,x,\tilde{\mu})\ d\tilde{\mu}(x) \right) d\tilde{M}_t(\tilde{\mu}) dt
    \\
    = & 
    -\int_0^T \xi(t) \int_{\PP(S)} \int_{S} \nabla_W \tilde{F}(x,\tilde{\mu}) \cdot \Tilde{b}(t,x,\tilde{\mu}) d\tilde{\mu}(x) d\tilde{M}_t(\tilde{\mu}) dt.
\end{align*}
\end{proof}

\normalcolor

Before proceeding, let us give some preliminary results: define the maps,
\begin{equation*}
\begin{aligned}
     J:C([0,T],\PP(\Ss))  \to C([0,T],\PP(S)),
     \quad 
     (\mu_t)_{t\in [0,T]}  \mapsto (\jmath_\sharp\mu_t)_{t\in [0,T]},
\end{aligned}
\end{equation*}
\begin{equation}\label{eq: ell direct}
    \ell:C([0,T],\Ss) \to C([0,T],S), \ \  \ell(\boldsymbol{y}) = (\jmath(y_t))_{t\in [0,T]},
\end{equation}
\begin{equation*}
    \mathfrak{J}:\PP(C([0,T],\Ss)) \to \PP(C([0,T],S)), \ \ 
    \mathfrak{J}\lambda = \ell_\sharp \lambda.
\end{equation*}
and their inverse
\begin{equation*}
\begin{aligned}
     J^{-1}:C([0,T],\PP(S))  \to C([0,T],\PP(\Ss)),
     \quad 
     (\tilde\mu_t)_{t\in [0,T]}  \mapsto (\jmath^{-1}_\sharp\tilde\mu_t)_{t\in [0,T]},
\end{aligned}
\end{equation*}
\begin{equation}\label{eq: ell}
    \ell^{-1}:C([0,T],S) \to C([0,T],\Ss), \ \  \ell^{-1}(\boldsymbol{x}) = (\jmath^{-1}(x_t))_{t\in [0,T]},
\end{equation}
\begin{equation*}
    \mathfrak{J}^{-1}:\PP(C([0,T],S)) \to \PP(C([0,T],\Ss)), \ \ 
    \mathfrak{J}^{-1}\tilde\lambda = \ell^{-1}_\sharp \tilde\lambda.
\end{equation*}

\begin{prop}\label{prop: correspondence}
    The following hold:
    \begin{enumerate}
        \item $\boldsymbol{x} = (x_t)_{t\in[0,T]} \in AC^p([0,T],S)$ if and only if $\boldsymbol{y}:= \ell^{-1}(\boldsymbol{x}) \in AC^p([0,T],\Ss)$. Moreover
        \begin{equation}\label{eq: ac curves corr}
            |\dot{x}|_{d_S}(t) = |\dot{y}|_{d_\Ss}(t) \quad \text{for a.e. }t\in[0,T];
        \end{equation}
        \item $\boldsymbol{\tilde\mu} = (\tilde\mu_t)_{t\in[0,T]} \in AC^p([0,T],\PP_p(S))$ if and only if $\boldsymbol{\mu}:= J^{-1}(\boldsymbol{\tilde\mu}) \in AC^p([0,T],\PP_p(\Ss))$. Moreover
        \begin{equation}\label{eq: ac curves meas corr}
            |\dot{\tilde{\mu}}|_{W_{p,d_S}}(t) = |\dot{\mu}|_{W_{p,d_\Ss}}(t) \quad \text{for a.e. }t\in[0,T];
        \end{equation}
        \item $\tilde{\lambda} \in \PP(C([0,T],S))$ is concentrated over $\boldsymbol{x} = (x_t)_{t\in[0,T]} \in AC([0,T],S)$ that solve
        \begin{equation}\label{ode x}
            \dot{x}(t) = \tilde{b}(t,x_t,(e_t)_\sharp\tilde\lambda),
        \end{equation} 
        if and only if $\lambda := \mathfrak{J}^{-1} \tilde\lambda \in \PP(C([0,T],\Ss))$ is concentrated over $\boldsymbol{y}\in AC([0,T],\Ss)$ that solve
        \begin{equation}\label{ode y}
            \dot{y}(t) = b(t,y_t,(e_t)_\sharp\lambda);
        \end{equation}
        \item $\boldsymbol{\tilde\mu} = (\tilde\mu_t)_{t\in[0,T]} \in C([0,T],\PP(S)) \subset C([0,T],\PP(\R^d))$ is a solution of 
        \begin{equation}\label{eq: ce tilde}
            \partial_t\tilde\mu_t + \operatorname{div}(\tilde{b}_t(\cdot,\tilde{\mu}_t)\mu_t) = 0
        \end{equation} 
        if and only if $\boldsymbol{\mu}:= J^{-1}(\boldsymbol{\tilde\mu})$ is a solution of $\partial_t\mu_t + \operatorname{div}_{\Ss}(b_t(\cdot,\mu_t)\mu_t) = 0$.
    \end{enumerate}
\end{prop}

\begin{proof}
    \textit{(1)} and \textit{(2)} follow directly from the metric properties of the Nash isometry, i.e. $\jmath$ is an isometry between the metric spaces $(\Ss,d_\Ss)$ and $(S,d_s)$, and $\jmath_\sharp$ is an isometry between $(\PP(\Ss),W_{p,d_\Ss})$ and $(\PP(S),W_{p,d_S})$. Regarding \textit{(3)}, since $\mathfrak{J}^{-1} = \ell^{-1}_\sharp$, for the direct implication it suffices to show that whenever $\boldsymbol{x} \in AC([0,T],S)$ solves \eqref{ode x}, then $\boldsymbol{y} := \ell^{-1}(\boldsymbol{x})$ solves \eqref{ode y}. This follows from noticing that for all $f\in C^1(\Ss)$
\begin{align*}
    \frac{d}{dt}  f(y_t) = &
    \frac{d}{dt}f\circ \jmath^{-1} (x_t) 
    = 
    d_{\jmath^{-1} (x_t)}f \circ d_{x_t}\jmath^{-1} \big( \tilde{b}(t,x_t,(e_t)_\sharp\tilde\lambda) \big) 
    \\
    = &
    d_{\jmath^{-1} (x_t)}f \circ d_{x_t}\jmath^{-1} \circ d_{\jmath^{-1}(x_t)}\jmath \big( b(t, \jmath^{-1}(x_t), (\jmath^{-1} \circ e_t)_\sharp \tilde\lambda) \big)
    \\
    = &
    d_{y_t}f \big( b(t,y_t, (e_t\circ \ell^{-1})_\sharp\tilde\lambda) \big) 
    =  d_{y_t}f \big( b(t,y_t, (e_t)_\sharp\lambda) \big),
\end{align*}  
where we used the trivial  equality $e_t\circ \ell^{-1} = \jmath^{-1}\circ e_t$. We conclude thanks to the fact that for a curve $\boldsymbol{y}:[0,T]\to \Ss$, being absolutely continuous is equivalent to ask that $f\circ \boldsymbol{y} \in AC(0,T)$ for all $f\in C^1(\Ss)$. Moreover, its derivative is characterized as follows: for a.e. $t\in (0,T)$
    \[\dot{\boldsymbol{y}} = v_t \in T_{y_t}\Ss \quad\text{for a.e. }t  \iff \frac{d}{dt}f\circ y(t) = d_{y_t}f\big( v_t \big)\quad\text{for a.e. }t.\]
The converse implication in \textit{(3)} follows from a similar computation. 
Regarding \textit{(4)}, the local isometry property of $\jmath$ implies 
\[\int_0^T \int_{\Ss} \| b(t,y,\mu_t) \|_{T_y\Ss} d\mu_t(y) dt = \int_0^T \int_S |\tilde{b}(t,x,\tilde\mu_t)|d\tilde\mu_t(x) dt,\]
so that the integrability assumptions for both the continuity equations are satisfied. 
Now, let $\xi\in C_c^1(0,T)$ and $\phi\in C^1(\Ss)$, then 
\begin{align*}
    \int_0^T & \xi'(t) \int_{\Ss} \phi(y) d\mu_t(y) dt + \int_0^T \xi(t) \int_{\Ss} d_y\phi\big( 
    b(t,y,\mu_t) \big) d\mu_t(y) dt 
    \\
    = & 
    \int_0^T \xi'(t) \int_{\Ss} \phi(y) d\jmath^{-1}_\sharp\tilde\mu_t(y) dt + \int_0^T \xi(t) \int_{\Ss} d_y\phi\big( 
    b(t,y,\jmath^{-1}_\sharp\tilde\mu_t) \big) d\jmath^{-1}_\sharp\tilde\mu_t(y) dt
    \\
    = & 
     \int_0^T \xi'(t) \int_{S} \phi(\jmath^{-1}(x)) d\tilde\mu_t(x) dt + \int_0^T \xi(t) \int_{S} d_{\jmath^{-1}(x)}\phi\big( 
    b(t,\jmath^{-1}(x),\jmath^{-1}_\sharp\tilde\mu_t) \big) d\tilde\mu_t(x) dt
    \\
    = & 
    \int_0^T \xi'(t) \int_{S} \phi(\jmath^{-1}(x)) d\tilde\mu_t(x) dt + \int_0^T \xi(t) \int_{S} d_{\jmath^{-1}(x)}\phi\Big( d_x\jmath^{-1} \circ d_{\jmath^{-1}(x)}\jmath\big(
    b(t,\jmath^{-1}(x),\jmath^{-1}_\sharp\tilde\mu_t) \big)\Big) d\tilde\mu_t(x) dt
    \\
    = & 
    \int_0^T \xi'(t) \int_{S} \phi(\jmath^{-1}(x)) d\tilde\mu_t(x) dt + \int_0^T \xi(t) \int_{S} d_{x}(\phi\circ \jmath^{-1})\big(
    \Tilde{b}(t,x,\tilde\mu_t) \big) d\tilde\mu_t(x) dt.
\end{align*}
Now, for the direct implication, since $\phi \circ \jmath^{-1} \in C^1(S)$, the last row is equal to $0$ whatever its extension to the whole $\R^d$ is. For the converse implication, choosing $\phi = \tilde{\phi}\circ \jmath$ with $\tilde{\phi} \in C_c^1(\R^d)$, again the previous computation allows to conclude because $\phi \in C^1(\Ss)$. 
\end{proof}

\begin{teorema}\label{superposition for manifolds}
    Let $\boldsymbol{M}=(M_t)_{t\in [0,T]}\in C([0,T],\PP(\PP(\Ss)))$ be a solution of the continuity equation $\partial_tM_t + \operatorname{div}_{\PP(\Ss)}(b_tM_t)=0$ in the sense of Definition \ref{ce manifold rm}. Then, there exists $\Lambda\in \PP(C([0,T],\PP(\Ss)))$ and $\mathfrak{L}\in \PP(\PP(C([0,T],\Ss)))$ such that:
    \begin{enumerate}
        \item\label{prop of Lambda} $(\mathfrak{e}_t)_{\sharp}\Lambda = M_t$ for any $t\in [0,T]$, with $\mathfrak{e}_t(\boldsymbol{\mu}) = \mu_t$ and $\Lambda$-a.e. $\boldsymbol{\nu} = (\mu_t)_{t\in[0,T]}\in C([0,T],\PP(\Ss))$ solves $\partial_t\mu_t + \operatorname{div}_\Ss(b_t(\cdot,\mu_t)\mu_t)=0$, in the sense of Def. \ref{ce manifold};
        \item\label{prop of mathfrak L} $(E_t)_\sharp\mathfrak{L}=M_t$, where $E_t=(e_t)_\sharp$ and $e_t(\boldsymbol{y}) = y_t$ for any $\boldsymbol{y}=(y_t)_{t\in [0,T]}\in C([0,T],\Ss)$. Moreover, for $\mathfrak{L}$-a.e. $\lambda\in \PP(C([0,T],\Ss))$, it holds that $\lambda$ is concentrated over $(y_t)_{t\in[0,T]}\in AC([0,T],\Ss)$ that are solutions of $\dot{\boldsymbol{y}}(t) = b(t,y_t,(e_t)_\sharp\lambda)$.
    \end{enumerate}
\end{teorema}


\begin{proof}
    Define $(\tilde{M}_t)_{t\in[0,T]} \in C([0,T],\PP(\PP(\R^d)))$ and $\tilde{b}:[0,T]\times \R^d \times \PP(\R^d) \to \R^d$ as in, respectively, \eqref{eq: def of M_t from N_t} and \eqref{def of tilde b from b}. Thanks to Proposition \ref{propr: tilde M}, the continuity equation $\partial_t\tilde{M}_t + \operatorname{div}(\tilde{b}_t\tilde{M}_t) = 0$ is satisfied. 

Thanks to Theorem \ref{thm: nested superposition}, there exists $\Tilde{\Lambda}\in \PP(C_T(\PP(\R^d)))$ such that $(\mathfrak{e}_t)_\sharp\Tilde{\Lambda}= \tilde{M}_t$ and $\Tilde{\Lambda}$-a.e. $\boldsymbol{\tilde\mu}=(\tilde\mu_t)_{t\in [0,T]}\in C([0,T],\PP(\R^d))$ solves 
\begin{equation*}
    \partial_t\tilde\mu_t +\operatorname{div}(\Tilde{b}_t(\cdot,\tilde\mu_t)\tilde\mu_t)= 0.
\end{equation*}
In particular, note that $\Tilde{\Lambda} \in \PP(C([0,T],\PP(S)))$, since $S\subset \R^d$ is closed.
Then, we set $\Lambda:= J^{-1}_\sharp \Tilde{\Lambda}$. We now show that $\Lambda$ satisfies the properties in \ref{prop of Lambda}. For all $F:\PP(\Ss)\to [0,+\infty]$ Borel, it holds
\begin{align*}
    \int_{C([0,T],\PP(\Ss))} &F(\mathfrak{e}_t\boldsymbol{\mu}) d\Lambda(\boldsymbol{\mu}) 
    = 
    \int_{C([0,T],\PP(\Ss))} F(\mathfrak{e}_t\boldsymbol{\mu})  dJ^{-1}_\sharp\Tilde{\Lambda}(\boldsymbol{\mu})
    = 
    \int_{C([0,T],\PP(S))} F(\mathfrak{e}_t(J^{-1}\boldsymbol{\tilde\mu})) d\Tilde{\Lambda}(\boldsymbol{\tilde\mu})
    \\
    = &
    \int_{C([0,T],\PP(S))} F(\jmath^{-1}_\sharp\tilde\mu_t) d\Tilde{\Lambda}(\boldsymbol{\tilde\mu}) 
    = 
    \int_{\PP(S)} F(\jmath^{-1}_\sharp \tilde\mu) d\tilde{M}_t(\tilde\mu) 
    =
    \int_{\PP(\Ss)} F(\mu) dM_t(\mu).
\end{align*}
The fact that it is concentrated over solutions of $\partial_t \mu_t + \operatorname{div}_\Ss(b_t(\cdot,\mu_t)\mu_t) = 0$, is a consequence of Proposition \ref{prop: correspondence}, \textit{(4)}.

Again by Theorem \ref{thm: nested superposition}, there exists a random measure $\Tilde{\mathfrak{L}}\in \PP(\PP(C([0,T],\R^d)))$ satisfying $(E_t)_\sharp\Tilde{\mathfrak{L}} = \tilde{M}_t$ and $\Tilde{\mathfrak{L}}$-a.e. $\tilde\lambda\in \PP(C([0,T],\R^d))$ is concentrated over curves $\boldsymbol{x}= (x_t)_{t\in [0,T]}\in AC([0,T],\R^d)$ that are solutions of $\dot{x}(t) = \Tilde{b}(t,x_t,(e_t)_\sharp\tilde\lambda)$. Notice that $\Tilde{\mathfrak{L}}\in \PP(\PP(C([0,T],S)))$, since $S\subset\R^d$ is closed. Then, set $\mathfrak{L}:= \mathfrak{J}^{-1}_\sharp \Tilde{\mathfrak{L}} = \ell^{-1}_{\sharp\sharp} \Tilde{\mathfrak{L}}$, for which we check that satisfies the properties in \ref{prop of mathfrak L}. For all $F:\PP(\Ss)\to[0,1]$ Borel, it holds
\begin{align*}
    & \int_{\PP(C([0,T],\Ss))}  F((e_t)_\sharp\eta) d\mathfrak{L}(\eta) = \int_{\PP(C([0,T],\Ss))} F((e_t)_\sharp\lambda) d\mathfrak{J}^{-1}_\sharp\Tilde{\mathfrak{L}}(\lambda)
    \\
    &=
    \int_{\PP(C([0,T],S))} F((e_t\circ\ell^{-1})_\sharp \tilde\lambda) d\Tilde{\mathfrak{L}}(\tilde\lambda)
    = 
    \int_{\PP(C([0,T],S))} F((\jmath^{-1}\circ e_t)_\sharp \tilde\lambda) d\Tilde{\mathfrak{L}}(\tilde\lambda)
    \\
    &=
    \int_{\PP(C([0,T],S))} F(\jmath^{-1}_\sharp(E_t\tilde\lambda)) d\Tilde{\mathfrak{L}}(\tilde\lambda)
    = 
    \int_{\PP(S)} F(\jmath^{-1}_\sharp\tilde\mu) d\tilde{M}_t(\tilde\mu)
    =
    \int_{\PP(\Ss)} F(\mu) dM_t(\mu),
\end{align*}
where we used $e_t \circ \ell^{-1} = \jmath^{-1}\circ e_t$. The other property follows from Proposition \ref{prop: correspondence}, \textit{(3)}. \qedhere
\end{proof}

\begin{oss}
    By construction, the following property holds as well: $\Lambda = E_\sharp\mathfrak{L}$, where $E: \PP(C([0,T],\Ss))\to C([0,T],\PP(\Ss))$ is defined as $E(\lambda) = ((e_t)_\sharp\lambda)_{t\in [0,T]}$, for any $\lambda\in \PP(C([0,T],\Ss))$. Moreover, there exists a (Souslin-Borel) measurable map $G:C([0,T],\PP(\Ss)) \to \PP(C([0,T],\Ss))$ such that $\mathfrak{L} = G_\sharp\Lambda$, and for $\Lambda$-a.e. $\boldsymbol{\mu}$ it holds $E(G(\boldsymbol{\mu})) = \boldsymbol{\mu}$. Indeed, such a property holds in $\R^d$, as shown in \cite{pinzisavare2025}, and it is inherited thanks to the properties of the Nash isometry.
\end{oss}

\subsection{Atomic nested superposition principle for compact Riemannian manifolds}\label{subsec: atomic manifolds}
In this subsection, we exploit again the Nash isometry theorem to prove a version of Theorem \ref{main theorem} and \ref{main theorem metric} in the ambient space $\Ss$.

The next proposition shows how the capacity assumption is related to the dimension of the manifold.
\begin{prop}\label{prop: manifold capacity}
    Let $(\Ss,g_{\Ss})$ be a compact Riemannian manifold of dimension $k\geq 2$, with $\operatorname{vol}_{\Ss}$ its natural volume measure. Then, for all $f\in L^{\infty}(\operatorname{vol}_{\Ss})$, given $\nu = f\operatorname{vol}_{\Ss}$, it holds
    $\operatorname{cap}_{k,\jmath_\sharp \nu}(\Delta) = 0$. In particular, $\operatorname{cap}_{r,\jmath_\sharp \nu}(\Delta) = 0$ for all $r\leq k$.
\end{prop}
\begin{proof}
    The proof relies again on the logarithmic cutoff functions introduced in \eqref{eq: log cut-off}. As in Proposition \ref{prop: sharp capacity in R^d}, writing $\tilde{\nu}:= \jmath_\sharp\nu$ we show that 
    \[\int_{\R^d \times \R^d}|h_{\varepsilon,R}(x_1,x_2)|^k + |\nabla h_{\varepsilon,R}(x_1,x_2)|^k d\tilde{\nu}\otimes \tilde{\nu}(x_1,x_2) \to 0,\]
    when choosing $R=\sqrt{\varepsilon}$ and letting $\varepsilon\to 0$. The first term is again trivial, since 
    \[\int_{\R^d \times \R^d}|h_{\varepsilon,R}(x_1,x_2)|^k d\tilde{\nu}\otimes \tilde{\nu}(x_1,x_2) \leq \tilde{\nu}\otimes \tilde{\nu}(\Delta_R) \to \tilde{\nu}\otimes \tilde{\nu}(\Delta) =  0,\]
    as $R\to 0$, since $\tilde{\nu}$ is atomless. The second term is a bit more delicate, and we will use many properties of the Nash embedding $\jmath$, in particular \eqref{eq: equivalence}, \eqref{eq: push for volume} and the fact that $\{x\in S : |x_1-x_2|\in(\varepsilon,R)\}\subset \{x_1\in S : d_S(x_1,x_2) \in (\varepsilon,CR)\}$, as well as the estimate \eqref{eq: estimates log cut-off}:
    \begin{align*}
        \int_{\R^d \times \R^d} &|\nabla h_{\varepsilon,R}|^k d\tilde{\nu}\otimes \tilde{\nu} \leq \frac{(2\sqrt{2})^k}{\left( \log \frac{R}{\varepsilon} \right)^k} \int_S \int_{\{x_1\in S : |x_1-x_2|\in (\varepsilon,R)\}} \frac{1}{|x_1-x_2|^k} f\circ \jmath^{-1}(x_1) d\mathcal{H}^k(x_1) d\tilde\nu(x_2)
        \\
        \leq & 
        \frac{(2\sqrt{2})^k\|f\|_\infty}{\left( \log \frac{R}{\varepsilon} \right)^k} \sup_{x_2\in S} \int_{\{x_1\in S : |x_1-x_2|\in (\varepsilon,R)\}} \frac{1}{|x_1-x_2|^k} d\mathcal{H}^k(x_1)
        \\ 
        \leq & 
        \frac{(2\sqrt{2} C)^k\|f\|_\infty}{ \left( \log \frac{R}{\varepsilon} \right)^k} \sup_{x_2\in S} \int_{\{x_1\in S : d_S(x_1,x_2)\in (\varepsilon,CR)\}} \frac{1}{d_S^k(x_1,x_2)} d\mathcal{H}^k(x_1)
        \\
        = & 
        \frac{(2\sqrt{2}C)^k\|f\|_\infty}{ \left( \log \frac{R}{\varepsilon} \right)^k} \sup_{y_2\in \Ss} \int_{\{y_1\in \Ss : d_\Ss(y_1,y_2)\in (\varepsilon,CR)\}} \frac{1}{d_\Ss^k(y_1,y_2)} d\operatorname{vol}_\Ss(y_1)
        \\
        = &
        \frac{(2\sqrt{2}C)^k\|f\|_\infty}{ \left( \log \frac{R}{\varepsilon} \right)^k} \sup_{y_2\in \Ss} \int_\varepsilon^{CR}  \frac{1}{r^k} \mathcal{H}^{k-1}(\{y_1\in \Ss : d_\Ss(y_1,y_2) = r\}) dr.
    \end{align*}
    Notice that, in the last row, the measure $\mathcal{H}^{k-1}$ is the $(k-1)$-dimensional Hausdorff measure built on the manifold $\Ss$. Now, thanks to \cite[Equation (11)]{volume1974gray}, for all $y_2\in \Ss$ there exists a constant $\tilde{C}>0$ such that, for $r>0$ small enough, it holds $\mathcal{H}^{k-1}(\{y_1\in \Ss : d_\Ss(y_1,y_2) = r\}) \leq \alpha_{k} r^{k-1} + \tilde{C}r^{k+1}$, where $\alpha_k$ is the $(k-1)$-dimensional measure of a unit sphere in $\R^k$. Then 
    \begin{align*}
         \int_{\R^d \times \R^d} &|\nabla h_{\varepsilon,R}|^k d\tilde{\nu}\otimes \tilde{\nu} \leq \frac{(2\sqrt{2}C)^k\|f\|_\infty}{ \left( \log \frac{R}{\varepsilon} \right)^k} \int_\varepsilon^{CR} \frac{\alpha_k}{r} + \tilde{C}r dr 
         \\
         = & (2\sqrt{2}C)^k\|f\|_\infty \left(\frac{\alpha_k}{ \left( \log \frac{R}{\varepsilon} \right)^{k-1}} + \frac{\log C + C^2R^2 - \varepsilon^2}{ \left( \log \frac{R}{\varepsilon} \right)^k}\right),
    \end{align*}
    and we conclude choosing $R = \sqrt{\varepsilon}$ and letting $\varepsilon \to 0$.
\end{proof}


We are now ready to state and prove the atomic nested superposition principle, in both the differential and metric versions, when the ambient space is a Riemannian manifold.

\begin{teorema}\label{main theorem manifold}
    Let $(\mathbb{S},g_{\mathbb{S}})$ be a compact, boundaryless and oriented Riemannian manifold. Assume $\pi \in \PP(\bold{T})$ and $\nu\in \PP(\Ss)$ with bounded density with respect to $\operatorname{vol}_\Ss$.
    
    Let $\Lambda \in \PP(C([0,T],\PP(\mathbb{S})))$ be concentrated over solutions of $\partial_t\mu_t + \operatorname{div}_\Ss(b_t(\cdot,\mu_t)\mu_t) = 0$ and define $\boldsymbol{M} = (M_t)_{t\in[0,T]} \in C([0,T],\PP(\PP(\mathbb{S})))$ as $M_t = (\mathfrak{e}_t)_\sharp \Lambda$. Assume:
    \begin{enumerate}
        \item $M_t = u_t \Q_{\pi,\nu}$, with $u \in L^1_t(L^q(\Q_{\pi,\nu}))$, with $q\in[1,+\infty]$;
        \item for some $p>1$, it holds
        \begin{equation}
            \int_0^T\int_{\PP(\Ss)} \int_\Ss \|b_t(y,\mu)\|_{T_y\Ss}^p d\mu(y)dM_t(\mu)dt <+\infty;
        \end{equation}
        \item $\frac{p'}{k}+\frac{1}{q}\leq1$. 
    \end{enumerate}
    Then, there exists $\mathfrak{L}\in \PP(\PPpa(AC^p([0,T],\Ss)))$ such that $E_\sharp \mathfrak{L} = \Lambda$ and is concentrated over $\lambda\in \PP(C([0,T],\Ss))$ that are, in turn, concentrated over absolutely continuous curves $\boldsymbol{y} = (y_t)_{t\in [0,T]}$ that are solutions of $\dot{\boldsymbol{y}}(t) = b_t(y_t,(e_t)_\sharp\lambda)$.
    \\
    In particular, $(E_t)_\sharp \mathfrak{L} = M_t$ for all $t\in[0,T]$ and $\Lambda$ is concentrated over $C([0,T],\PP^{\operatorname{pa}}(\mathbb{S}))$.
\end{teorema}

\begin{proof}
    Define $\tilde\Lambda := J_\sharp \Lambda \in \PP(C([0,T],\PP(S))) \subset \PP(C([0,T],\PP(\R^d)))$. Thanks to Proposition \ref{prop: correspondence}, \textit{(4)}, $\tilde\Lambda$ is concentrated on $\operatorname{CE}(\tilde{b})$. Next, we show that it satisfies the hypothesis of Theorem \ref{main theorem}.
    \begin{enumerate}
        \item The equality $\mathfrak{e}_t\circ J = \jmath_\sharp \circ \mathfrak{e}_t$ for all $t\in[0,T]$, gives that $(\mathfrak{e}_t)_\sharp \tilde\Lambda = \tilde{M}_t := \jmath_{\sharp\sharp} M_t$. Moreover, recalling \eqref{eq: nash ref measure}, we have that 
        \[\tilde{M}_t = \tilde{u}_t Q_{\pi,\jmath_\sharp \nu}, \quad \tilde{u}_t := u_t \circ \jmath_{\sharp}^{-1},\]
        which satisfy $\tilde{u}\in L^1_t(L^q(Q_{\pi,\jmath\sharp \nu}))$ from the following computation:
        \[ \int_0^T \|u_t\|_{L^q(Q_{\pi,\nu})} dt = \int_0^T \|u_t\circ \boldsymbol{\jmath}^{-1} \circ \boldsymbol{\jmath}\|_{L^q(Q_{\pi,\nu})} dt = \int_0^T \|u_t\circ \boldsymbol{\jmath}^{-1}\|_{L^q(Q_{\pi,\boldsymbol{\jmath}(\nu)})} dt. \]
        \item Again, the local isometry property of $\jmath$ implies $\|b_t(y,\mu)\|_{\ttg} =|\tilde{b}_t(\jmath(y), \jmath_\sharp \mu)|$, and in particular $\|b_t(\cdot,\mu)\|_{L^p(\mu)} = \|\tilde{b}_t(\cdot,\jmath_\sharp\mu)\|_{L^p(\jmath_\sharp \mu)}$. Then, since $\tilde{M}_t = \jmath_{\sharp\sharp}M_t$, we finally have
        \[\int_0^T \int_{\PP(\Ss)}\int_{\Ss} \|b_t(y,\mu)\|_{\ttg}^p d\mu(y)dM_t(\mu)dt = \int_0^T \int_{\PP(\R^d)} \int_{\R^d} |\tilde{b}_t(x,\tilde{\mu})|^p d\tilde{\mu}(x) d\tilde{M}_tdt; \]
        \item Proposition \ref{prop: manifold capacity} gives that $\operatorname{cap}_{k,\jmath_\sharp\nu}(\Delta) = 0$.
        \item The balance between the exponents works for $r = k$.
    \end{enumerate}
    Then, there exists $\tilde{\mathfrak{L}} \in \PP(\PPpa(C([0,T],\R^d)))$ that is concentrated over $\operatorname{SPS}(\tilde{b})$ and satisfies $E_\sharp \tilde{\mathfrak{L}} = \tilde\Lambda$, which implies that $\mathfrak{\tilde{L}}\in \PP(\PP(C([0,T],S)))$. Then, defining $\mathfrak{L} := \mathfrak{J}^{-1}_\sharp \tilde{\mathfrak{L}} \in \PP(\PP(C([0,T],\Ss)))$, we have: it is concentrated over $\lambda \in \PPpa(C([0,T],\Ss))$ thanks to Lemma \ref{lemma: push for of PA meas}; it is concentrated over $\lambda$ that, in turn, are concentrated over $\boldsymbol{y} = (y_t)_{t\in [0,T]} \in AC([0,T],\Ss)$ that are solutions of $\dot{y}(t) = b_t(y_t,(e_t)_\sharp \lambda)$, thanks to Proposition \ref{prop: correspondence}, \textit{(3)}; the trivial equality $E\circ \mathfrak{J}^{-1} = J^{-1}\circ E$ gives that $E_\sharp \mathfrak{L} = (E\circ \mathfrak{J}^{-1})_\sharp \tilde{\mathfrak{L}} = (J^{-1} \circ E)_\sharp \tilde{\mathfrak{L}} = J^{-1}_\sharp \tilde\Lambda = \Lambda$. 
\end{proof}

\begin{teorema}\label{main theorem metric manifold}
    Let $(\mathbb{S},g_{\mathbb{S}})$ be a compact, boundaryless and oriented Riemannian manifold. Assume $\pi \in \PP(\bold{T})$ and $\nu\in \PP(\Ss)$ with bounded density with respect to $\operatorname{vol}_\Ss$.
    
    Let $\Lambda \in \PP(C([0,T],\PP(\mathbb{S})))$ be concentrated over $AC^p([0,T],\PP_p(\mathbb{S}))$ and define $\boldsymbol{M} = (M_t)_{t\in[0,T]} \in C([0,T],\PP(\PP(\mathbb{S})))$ as $M_t = (\mathfrak{e}_t)_\sharp \Lambda$. Assume:
    \begin{enumerate}
        \item $M_t = u_t \Q_{\pi,\nu}$, with $u \in L^1_t(L^q(\Q_{\pi,\nu}))$, with $q\in[1,+\infty]$;
        \item for some $p>1$, it holds
        \begin{equation}
            \int\int_0^T |\dot{\boldsymbol{\mu}}|_{W_{p,d_\Ss}}^p(t)dt d\Lambda(\boldsymbol{\mu})<+\infty;
        \end{equation}
        \item $\frac{p'}{k}+\frac{1}{q}\leq1$. 
    \end{enumerate}
    Then, there exists $\mathfrak{L}\in \PP(\PPpa(AC^p([0,T],\Ss)))$ such that $E_\sharp \mathfrak{L} = \Lambda$ and 
    \begin{equation}\label{eq: minimal lifting riem}
        \int\int\int_0^T |\dot{\boldsymbol{y}}|^p_{d_{\mathbb{S}}}(t) dt d\lambda (\boldsymbol{y})d\mathfrak{L}(\lambda) <+\infty
    \end{equation} In particular, $(E_t)_\sharp \mathfrak{L} = M_t$ for all $t\in[0,T]$ and $\Lambda$ is concentrated over $C([0,T],\PP^{\operatorname{pa}}(\mathbb{S}))$.
\end{teorema}

\begin{proof}
    The proof follows the same line of the one of Theorem \ref{main theorem manifold}. Define $\tilde{\Lambda} := J_\sharp \Lambda \in \PP(C([0,T],\PP(S)))$. Thanks to Proposition \ref{prop: correspondence}, \textit{(2)}, $\tilde\Lambda$ is concentrated over $AC^p([0,T],S)$, which coincides with $AC^p([0,T],\R^d) \cap C([0,T],S)$ because $|x-y| \leq d_S(x,y) \leq C|x-y|$ for all $x,y \in S$. Let us verify that $\tilde\Lambda$ satisfies the assumptions of Theorem \ref{main theorem metric}:
    \begin{enumerate}
        \item As in the proof of Theorem \ref{main theorem manifold}, $(\mathfrak{e}_t)_\sharp \tilde\Lambda = \tilde{M}_t := \jmath_{\sharp\sharp} M_t$, and $\tilde{M}_t = \tilde{u}_t Q_{\pi,\jmath_\sharp \nu}$, with $\tilde{u}\in L^1_t(L^q(Q_{\pi,\jmath_\sharp \nu}))$.
        \item Thanks to \eqref{eq: ac curves meas corr} and $W_p(\tilde{\mu}_1,\tilde{\mu}_2) \leq  W_{p,d_S}(\tilde{\mu}_1,\tilde{\mu}_2)$ for all $\tilde{\mu}_1,\tilde{\mu}_2 \in \PP(S)$, we have 
        \[\int\int_0^T |\boldsymbol{\dot{\tilde\mu}}|_{W_p}^p(t)dt d\tilde{\Lambda}(\boldsymbol{\tilde\mu}) \leq  \int\int_0^T |\boldsymbol{\dot{\tilde\mu}}|_{W_{p,d_S}}^p(t)dt d\tilde{\Lambda}(\boldsymbol{\tilde\mu}) =  \int\int_0^T |\dot{\boldsymbol{\mu}}|_{W_{p,d_\Ss}}^p(t)dt d\Lambda(\boldsymbol{\mu}) <+\infty. \]
        \item[(3)]and (4) They both easily follow as in the proof of Theorem \ref{main theorem manifold}.
    \end{enumerate}
    Then, there exists $\tilde{\mathfrak{L}} \in \PP(\PPpa(AC^p([0,T],\R^d)))$ satisfying $E_\sharp \tilde{\mathfrak{L}} = \tilde\Lambda$, and in particular $\tilde{\mathfrak{L}} \in \PP(\PP(C([0,T],S)))$, so that 
    \[\int \int \int_0^T |\dot{x}|_{d_S}^p(t) dt d\tilde{\lambda}(\boldsymbol{x}) d\tilde{\mathfrak{L}}(\tilde\lambda) \leq C^p \int \int \int_0^T |\dot{x}|^p(t) dt d\tilde{\lambda}(\boldsymbol{x}) d\tilde{\mathfrak{L}}(\tilde\lambda)<+\infty.\]
    Define $\mathfrak{L}:= \mathfrak{J}^{-1}_\sharp \tilde{\mathfrak{L}}$. Again, we have: $E_\sharp \mathfrak{L} = (E\circ \mathfrak{J}^{-1})_\sharp \tilde{\mathfrak{L}} = (J^{-1} \circ E)_\sharp \tilde{\mathfrak{L}} = J^{-1}_\sharp \tilde\Lambda = \Lambda$; $\mathfrak{L}\in \PP(\PPpa(C([0,T],\Ss)))$ thanks to Lemma \ref{lemma: push for of PA meas}; thanks to Proposition \ref{eq: ac curves corr}, \textit{(1)}, $\mathfrak{L}\in \PP(\PP(AC^p([0,T],\Ss)))$ and \eqref{eq: minimal lifting riem} holds as well.
\end{proof}

\section{A metric measure viewpoint on the Wasserstein space}\label{sec: sobolev and curvature}
Let us recall some general notions about Sobolev spaces on metric measure spaces, for which we mainly follow \cite{gigli2020lectures, AmIkLuPa24}. In particular, over all the equivalent definitions of Sobolev spaces, we will give the one that exploits \textit{test plans}, that will be particularly useful for our setting.
\\
Let $(X,d,\ttm)$ be a metric probability measure space, that is $(X,d)$ is a complete and separable metric space and $\ttm\in \mathcal{P}(X)$.
\begin{df}[$p'$-test plan]
    Let $p\in [1,+\infty)$. A $p'$-test plan is a probability measure over curves $\Pi\in \PP(C([0,1],X))$ satisfying:
    \begin{enumerate}
        \item there exists $C\geq 0$ (compression constant) such that $(\e_t)_\sharp \Pi \leq C \ttm$ for all $t\in[0,1]$;
        \item it has finite $p'$-energy, that is 
        \begin{equation}\label{eq: finite p' energy}
            \int \int_0^1 |\dot{\boldsymbol{x}}|_d^{p'}(t) dt d\Pi(\boldsymbol{x})<+\infty.
        \end{equation}
    \end{enumerate}
\end{df}

\begin{df}[$p$-Sobolev space]
    Let $p\in[1,+\infty)$ and $f\in L^p(\ttm)$. We say that $f\in W^{1,p}(X,d,\ttm)$ if there exists a non-negative function $G \in L^{p}(\ttm)$ such that for all $p'$-test plan $\Pi$, it holds $f\circ \boldsymbol{x} \in W^{1,1}(0,1)$ for $\Pi$-a.e. $\boldsymbol{x}$ and 
    \begin{equation}\label{eq: sob space}
        |(f\circ \boldsymbol{x})'(t)| \leq G(\boldsymbol{x}_t)|\dot{\boldsymbol{x}}|_d(t), \quad \text{for } \Pi \otimes \mathcal{L}^1\text{-a.e. }(\boldsymbol{x},t) \in C([0,1],X)\times [0,1].
    \end{equation}
    Any $G$ as above is called a $p$-weak upper gradient, and in the class of weak upper gradients we can consider the $\ttm$-a.e. pointwise minimum, denoted by $|Df|$. Then, $W^{1,p}(X,d,\ttm)$ is a Banach space when endowed with the norm 
    \begin{equation}
        \|f\|_{W^{1,p}(X,d,\ttm)} := \left(\|f\|_{L^p(\ttm)}^p+ \||Df|\|_{L^p(\ttm)}^p\right)^{1/p}.
    \end{equation}
\end{df}

A different, but equivalent, approach for defining the Sobolev spaces is by relaxing Lipschitz functions \cite[Section 5.1]{AmIkLuPa24}, giving in particular that globally $L$-Lipschitz functions $f$ are in $W^{1,p}(X,d,\ttm)$, with $|Df|\leq L$ $\ttm$-a.e. The vice versa is called \textit{Sobolev-to-Lipschitz} property and does not always hold:
\begin{equation}\tag{Sob-Lip}\label{eq: p-sob to lip}
    \begin{gathered}
        W^{1,p}(X,d,\ttm) \text{ satisfies the $p$-Sobolev-to-Lipschitz property}
        \\
        \iff 
        \\
        \text{for all }f\in W^{1,p}(X,d,\ttm), \ \ |Df|\leq L \implies \exists \Tilde{f}:X \to \R \text{ $L$-Lip. with }f=\Tilde{f} \ \ttm\text{-a.e.}
    \end{gathered}
\end{equation}

Another important property regards the Hilbertianity of the $2$-Sobolev space. It is not always true, as in the usual Sobolev theory, that $W^{1,2}(X,d,\ttm)$ is an Hilbert space. We thus recall the following definition due to N. Gigli \cite{gigli2013splitting}:
\begin{equation}\label{inf-H}\tag{inf-H}
    (X,d,\ttm) \text{ is \textbf{infinitesimally Hilbertian}} \iff W^{1,2}(X,d,\ttm) \text{ is a Hilbert space}.
\end{equation}

From now on, we assume that $(X,d,\ttm)$ is infinitesimally Hilbertian, which implies that the associated Cheeger energy 
\[\operatorname{Ch}(f) := \begin{cases}
\int_{X} |Df|^2 d\ttm, & f \in W^{1,2}(X,d,\ttm)
\\
+\infty, & f\in L^2(\ttm)\setminus W^{1,2}(X,d,\ttm)
\end{cases}\]
is quadratic. Moreover, it is lower semicontinuous with respect to the $L^2$-convergence. In particular, we can build its gradient flow in the Lebesgue space $L^2(\ttm)$, obtaining a contractive linear semigroup $\mathtt{H}_t:L^2(\ttm) \to L^2(\ttm)$, $t\geq 0$, usually called \textit{heat semigroup} associated to $(X,d,\ttm)$, performing the \textit{minimizing movement scheme}: for all $f\in L^2(\ttm)$ and $\tau>0$, define $f^{(\tau)}_{(0)} := f$ and
\begin{equation}\label{MMS}
    f^{(\tau)}_{(n+1)} \in \operatorname{argmin} \bigg( \operatorname{Ch}(\cdot) + \frac{\|\cdot - f^{(\tau)}_{(n)} \|^2}{2\tau} \bigg), \quad f^{(\tau)}_t := f^{(\tau)}_{(n)}, \ \text{with } t\in[n\tau, (n+1)\tau).
\end{equation}
Then, $\mathtt{H}_t f = \lim_{\tau\to 0} f^{(\tau)}_t$.
Its main properties are listed in \cite[Theorem 5.1.12]{gigli2020lectures}. Here we list the ones that are of our interest:
\begin{enumerate}
    \item for all $t\geq 0$, $\mathtt{H}_t$ is well-defined as an operator from $L^p(\ttm)$ to $L^p(\ttm)$, as shown in \cite[Proposition 5.2.15]{gigli2020lectures};
    \item for all $t>0$ and $f\in L^2(\ttm)$, $\mathtt{H}_t f \in W^{1,2}(X,d,\ttm)$;
    \item for all $p\in[1,+\infty)$ and $t>0$, $\|\mathtt{H}_t f\|_{L^p(\ttm)} \leq \|f\|_{L^p(\ttm)}$ for all $L^2(\ttm)\cap L^p(\ttm)$. In particular, it can be extended to a linear and continuous operator $\mathtt{H}_t:L^p(\ttm) \to L^p(\ttm)$;
    \item it has the Markov property, that is $\mathtt{H}_t \mathds{1} = \mathds{1}$ and $\mathtt{H}_t f \geq 0$ $\ttm$-a.e., whenever $f\geq 0$ $\ttm$-a.e. Both properties can be directly verified from the minimizing movement scheme \eqref{MMS}.
\end{enumerate}

The last property we introduce about general metric measure spaces is the \textit{Bakry--\'Emery curvature condition}, for some $K\in \R$: 
\begin{equation}\label{eq: BE}\tag{$\operatorname{BE}(K,\infty)$}
   |D(\mathtt{H}_t f)|^2 \leq e^{-2Kt} \mathtt{H}_t(|Df|^2) \quad \ttm\text{-a.e.}, \quad \  \forall f\in W^{1,2}(X,d,\ttm), \ \forall t\geq 0.
\end{equation}

The three conditions above provides one of the equivalent definitions of $\operatorname{RCD}(K,\infty)$ spaces (see for example \cite{savare2013self, ambrosio2015bakry}).

\begin{df}\label{def: rcd}
    Let $K\in \R$. A metric measure space $(X,d,\ttm)$ is said to be $\operatorname{RCD}(K,\infty)$ if it satisfies \eqref{inf-H}, \eqref{eq: p-sob to lip} with $p=2$ and \eqref{eq: BE}.
\end{df}

In the next section, we will see how our atomic nested superposition principles can be exploited to show that the $L^2$-Wasserstein space, endowed with suitable reference measures of the form $Q_{\pi,\nu}$, does not satisfy \eqref{eq: p-sob to lip}, and therefore it cannot be a $\operatorname{RCD}$ space.

Before applying this general theory to the Wasserstein space, let us state a useful lemma.

\begin{lemma}\label{lemma: BE dense in energy}
    Let $(X,d,\ttm)$ be infinitesimally Hilbertian, and $K\in \R$. Let $\mathcal{A}\subset W^{1,2}(X,d,\ttm)$ be dense in energy, that is, for all $f\in W^{1,2}(X,d,\ttm)$ there exist $f_n \in \mathcal{A}$ such that 
    \begin{equation}\label{eq: dense in energy}
        f_n \overset{L^2(\ttm)}{\longrightarrow} f, \quad |Df_n|  \overset{L^2(\ttm)}{\longrightarrow} |Df|.
    \end{equation}
    Then, it suffices to verify \eqref{eq: BE} for $f\in \mathcal{A}$.
\end{lemma}

\begin{proof}
    First, we show that for all Borel function $\varphi:X \to [0,+\infty)$ such that there exist $0< c < C <+\infty$ with $c\leq \varphi \leq C$, it holds 
    \begin{equation}
    \operatorname{Ch}_\varphi: L^2(\ttm) \to [0,+\infty], \quad 
    \operatorname{Ch}_\varphi(f):=
    \begin{cases} 
    \int_{X} \varphi |Df|^2 d\ttm, & f\in W^{1,2}(X,d,\ttm)
    \\
    +\infty, & f\in L^2(\ttm)\setminus W^{1,2}(X,d,\ttm)
    \end{cases}
    \end{equation}
    is lower semicontinuous with respect to the $L^2$-convergence. Take $f_n \to f$ in $L^2(\ttm)$. If $\liminf \operatorname{Ch}_\varphi(f_n) = +\infty$, there is nothing to prove, thus we assume it is finite and, up to consider a subsequence, we assume that the limit exists. Moreover, define the finite positive measure $\ttm_\varphi := \varphi \ttm$, which is equivalent to $\ttm$, since $c \ttm \leq  \ttm_\varphi \leq C \ttm$. Now, since the limit of $\operatorname{Ch}_\varphi(f_n)$ is finite, there exists a subsequence $f_{n_k}$ and $G \in L^2(\ttm)$ such that $|D f_{n_k}| \rightharpoonup G$ weakly in $L^2(\ttm_\varphi)$. Moreover, the same limit holds weakly in $L^2(\ttm)$, implying that $G \geq |Df|$, see for example \cite[Proposition 2.1.13]{gigli2020lectures}. In particular
    \[\lim_{n\to +\infty} \operatorname{Ch}_\varphi(f_n) = \lim_{n\to+\infty} \||Df_n|\|_{L^2(\ttm_\varphi)}^2 \geq \|G\|_{L^2(\ttm_\varphi)}^2 = \int_{X} \varphi G^2 \ttm \geq \int_X \varphi |Df|^2 d\ttm.\]
    
    Now, we are done if we prove that 
    \[\int_X \varphi |D(\mathtt{H}_t f)|^2 d\ttm  \leq e^{-2Kt} \int_X \varphi \ \mathtt{H}_t(|Df|^2) d\ttm, \]
    for all non-negative and bounded functions $\varphi$. Consider $f_n \in \mathcal{A}$ such that $f_n \to f$ as in \eqref{eq: dense in energy}. Then, $|Df_n|^2 \to |Df|^2 $ in $L^1(\ttm)$, and by the $L^1$-contraction property of $\ttH_t$, also $\ttH_t(|Df_n|^2) \to \ttH_t(|Df|^2) $ in $L^1(\ttm)$. Then, for all $\varepsilon>0$, we have 
    \[
    \begin{aligned}
    \int_X & \varphi |D(\mathtt{H}_t f)|^2 d\ttm \leq  \int_X  (\varphi+\varepsilon) |D(\mathtt{H}_t f)|^2 d\ttm \leq   \liminf_{n\to+\infty} \int_X (\varphi+\varepsilon) |D(\mathtt{H}_t f_n)|^2 d\ttm 
    \\
    & \leq \liminf_{n\to +\infty} e^{-2Kt} \int_X (\varphi +\varepsilon)\ \mathtt{H}_t(|Df_n|^2) d\ttm = e^{-2Kt}\left( \int_X \varphi \ \mathtt{H}_t(|Df|^2) d\ttm + \varepsilon \int_X \ttH_t(|Df|^2)d\ttm \right),  
    \end{aligned}
    \]
    and we conclude by arbitrariness of $\varepsilon>0$, since $\int_X \ttH_t(|Df|^2)d\ttm<+\infty$.
\end{proof}

\subsection{Lack of Sobolev-to-Lipschitz and Poincar\'e inequality}\label{subsec: sob to lip poincare}

\subsubsection{Euclidean case}
Let $\pi \in \PP(\bold{T})$ and $\nu\in \PP_p(\R^d)$ be an atomless measure. Here we show that, for suitable $p>1$, the $p'$-Sobolev-to-Lipschitz and the $p'$-Poincaré inequality properties fail in the metric measure space $(\PP_p(\R^d),W_p,\Q_{\pi,\nu})$. In the next theorem, we build a function that will be a counterexample to both properties.

\begin{teorema}\label{thm: euclidean case}
    Let $p\in(1,+\infty)$ and $r\in(1,+\infty)$ be such that $p\geq \frac{r}{r-1}$ and $\operatorname{cap}_{r,\nu}(\Delta) = 0$. Assume that $\pi \neq \delta_{\mathbf{a}}$ for all $\mathbf{a}\in \bold{T}$. 
    \\
    Then, there exists a function $F \in W^{1,p'}(\PP_p(\R^d),W_p,Q_{\pi,\nu})$ for which $|DF| \equiv 0$, but it does not admit a constant representative.
\end{teorema}

\begin{proof}
	For all $a\in(0,1]$, consider the function $F_{a}(\mu):= \int_{\R^d} \mathds{1}_{[a,1]}(\mu[x]) d\mu(x)$, for all $\mu\in \PP_p(\R^d)$. We prove that $F_a \in W^{1,p'}\big(\PP_p(\R^d),W_p,Q_{\pi,\nu}\big)$ and $|DF_a|\equiv 0$. To show that, consider a $p$-test plan $\Lambda \in \PP(C([0,1],\PP_p(\R^d)))$ and notice that it falls in the hypothesis of Theorem \ref{main theorem metric}, so that there exists $\mathfrak{L}\in \PP(\PPpa(AC^{p}([0,1],\R^d)))$ such that $E_\sharp \mathfrak{L} = \Lambda$. In particular, thanks to the proof of such a theorem, we know that $\Lambda$-a.e. $(\mu_t)_{t\in[0,T]}$ satisfies $(\textit{H1})$ and $(\textit{H2})$ (see Lemma \ref{lemma: main}) and thanks to Lemma \ref{lemma: dual with indicators} the map 
		\(t\mapsto F_a(\mu_t)\) is in $W^{1,1}(0,1)$ and constant in time, giving that $G\equiv 0$ is a $p'$-weak upper gradient for $F_a$.
	
	
	Now, we show that the hypothesis that $\pi$ is not a Dirac measure prevents the existence of a constant $C_a$ such that $F_a(\mu) = C_a$ for $Q_{\pi,\nu}$-a.e. $\mu\in \P_p(\R^d)$. Indeed, define
	\[\underline{i} := \min\left\{ j-1 \ : \ \nexists (s_1,\dots,s_j) \in [0,1]^j \ \text{s.t.} \ \pi(\{s_1\}\times \dots \times \{s_j\}\times[0,1]^{\infty}) = 1 \right\}.\]
	Then, there exists $1 \geq a_1\geq \dots \geq a_{\underline{i}} \geq \underline{a}>0$ such that 
	\[A^+ := \{\mathbf{a}\in \bold{T} \ : \ a_{\underline{i}+1} \geq \underline{a}\}, \quad 0<\pi(A^+) <1, \quad \pi(\{a_1\}\times \dots \times \{a_{\underline{i}}\}\times[0,1]^{\infty}) = 1.\]
	Define then $A^-:= \bold{T}\setminus A^+$, so that
	\[Q_{\pi,\nu}\big( \operatorname{em}(A^+ \times \bold{X}) \big)>0, \quad Q_{\pi,\nu}\big( \operatorname{em}(A^- \times \bold{X}) \big) >0. \]
	At this point, consider $\mu^+ = \sum_{j} a_j^+ \delta_{x_j^+} \in \operatorname{em}(A^+ \times \bold{X})$ and $\mu^- = \sum_{j} a_j^-\delta_{x_j^-} \in \operatorname{em} (A^- \times \bold{X})$, for some $\bold{a}^{\pm}\in \bold{T}$ and $\bold{x}^{\pm}\in \bold{X}$ with $a_1^+ = a_1^- = a_1$,$\dots$,$a_{\underline{i}}^+ = a_{\underline{i}}^- = a_{\underline{i}}$. Notice that $a_j^- < \underline{a}$ for all $j\geq \underline{i}+1$. Then
	\[F_{\underline{a}}(\mu^+) = \bigg(\sum_{j=1}^{\underline{i}} a_j\mathds{1}_{[\underline{a},1]}(a_j)\bigg) + a_{\underline{i}+1}^+\mathds{1}_{[\underline{a},1]}(a_{\underline{i}+1}^+) + \bigg(\sum_{j\geq \underline{i}+2} a_j^+\mathds{1}_{[\underline{a},1]}(a_j^{+})\bigg)\geq a_1+ \dots + a_{\underline{i}} + \underline{a},\]
	\[F_{\underline{a}}(\mu^-) = \bigg(\sum_{j=1}^{\underline{i}} a_j \mathds{1}_{[\underline{a},1]}(a_j)\bigg) + a_{\underline{i}+1}^-\mathds{1}_{[\underline{a},1]}(a_{\underline{i}+1}^-) + \bigg(\sum_{j\geq \underline{i}+2} a_j^-\mathds{1}_{[\underline{a},1]}(a_j^{-}) \bigg) = a_1+ \dots + a_{\underline{i}}.\]
	In particular $F_{\underline{a}}$ cannot admit a constant representative that coincides with it $\Q_{\pi,\nu}$-a.e.
\end{proof}

\begin{co}\label{co: sob-to-Lip R^d}
     Let $p\in(1,+\infty)$ and $r\in(1,+\infty)$ be such that $p\geq \frac{r}{r-1}$ and $\operatorname{cap}_{r,\nu}(\Delta) = 0$. Assume that $\pi \neq \delta_{\mathbf{a}}$ for all $\mathbf{a}\in \bold{T}$. 
     Then, $W^{1,p'}\big(\PP_p(\R^d),W_p,Q_{\pi,\nu}\big)$ does not satisfy \eqref{eq: p-sob to lip}.
\end{co}

\begin{proof}
    If \eqref{eq: p-sob to lip}, then the function $F$ from Theorem \ref{thm: euclidean case} would admit a $0$-Lipschitz representative, i.e. a constant representative, raising a contradiction.
\end{proof}

\begin{co}\label{co: poincare R^d}
    Let $p\in(1,+\infty)$ and $r\in(1,+\infty)$ be such that $p\geq \frac{r}{r-1}$ and $\operatorname{cap}_{r,\nu}(\Delta) = 0$. Assume that $\pi \neq \delta_{\mathbf{a}}$ for all $\mathbf{a}\in \bold{T}$. 
    Then, there exists a function $F \in W^{1,p'}(\PP_p(\R^d),W_p,Q_{\pi,\nu})$ such that 
    \begin{equation}\label{eq: non poincare}
        \int_{\PP_p(\R^d)} \bigg| F(\mu) - \int_{\PP_p(\R^d)} F dQ_{\pi,\nu} \bigg|^{p'} dQ_{\pi,\nu}(\mu) > N \int_{\PP_p(\R^d)} |DF|^{p'}(\mu) dQ_{\pi,\nu} \quad \text{for all }N\in \N.
    \end{equation}
    In particular, the metric measure space $(\PP_p(\R^d),W_p,Q_{\pi,\nu})$ does not support a $p'$-Poincaré inequality.
\end{co}

\begin{proof}
    The function $F$ from Theorem \ref{thm: euclidean case} is not constant, implying that the left hand side in \eqref{eq: non poincare} is strictly positive. On the other hand, $|DF| \equiv 0$ implies that the right hand side in \eqref{eq: non poincare} is always null. 
\end{proof}

\subsubsection{Riemannian case}

\begin{teorema}\label{thm: manifold case}
    Let $(\mathbb{S},g_{\mathbb{S}})$ be a boundaryless compact Riemannian manifold of dimension $k\geq 2$, and $\nu\in \PP(\mathbb{S})$ be a probability measure that has a bounded density with respect to $\operatorname{vol}_{\mathbb{S}}$. Assume that $\pi \neq \delta_{\mathbf{a}}$ for all $\mathbf{a}\in \bold{T}$.
    \\
    Then, there exists a function $F \in W^{1,p'}(\PP_p(\Ss),W_p,Q_{\pi,\nu})$ for all $p\in[\frac{k}{k-1},+\infty)$, for which $|DF| \equiv 0$ and that does not admit a constant representative.
\end{teorema}

\begin{proof}
    The strategy is the same as for the Euclidean case, with the difference that we will make use of the atomic nested superposition principles in Section \ref{subsec: atomic manifolds}. For all $a\in(0,1]$, consider the function $F_a(\mu):= \int_{\Ss} \mathds{1}_{[a,1]}(\mu[y]) d\mu(y)$, for all $\mu\in \PP(\Ss)$. Consider a $p$-test plan for the metric measure space $(\PP(\Ss),W_p,Q_{\pi,\nu})$, that is $\Lambda\in \PP(C([0,1],\PP(\Ss)))$, which falls in the hypothesis of Theorem \ref{main theorem metric manifold}, with $q=+\infty$. The proof is concluded arguing as in the proof of Theorem \ref{thm: euclidean case}.
\end{proof}

As for the Euclidean case, the previous theorem provides a counterexample for the $p'$-Sobolev-to-Lipschitz and the $p'$-Poincar\'e inequality.

\begin{co}\label{co: manifold sob to lip}
    Let $(\mathbb{S},g_{\mathbb{S}})$ be a boundaryless compact Riemannian manifold of dimension $k\geq 2$, and $\nu\in \PP(\mathbb{S})$ be a probability measure that has a bounded density with respect to $\operatorname{vol}_{\mathbb{S}}$. Assume that $\pi \neq \delta_{\mathbf{a}}$ for all $\mathbf{a}\in \bold{T}$.
    \\
    Then, the space $W^{1,p'}(\PP(\mathbb{S}),W_p,Q_{\pi,\nu})$ does not satisfy \eqref{eq: p-sob to lip} for all $p \in [\frac{k}{k-1},+\infty)$.
\end{co}

\begin{co}\label{co: manifold non poincare}
    Let $(\mathbb{S},g_{\mathbb{S}})$ be a boundaryless compact Riemannian manifold of dimension $k\geq 2$, and $\nu\in \PP(\mathbb{S})$ be a probability measure that has a bounded density with respect to $\operatorname{vol}_{\mathbb{S}}$. Assume that $\pi \neq \delta_{\mathbf{a}}$ for all $\mathbf{a}\in \bold{T}$.
    \\
    Then, there exists a function $F \in W^{1,p'}(\PP_p(\Ss),W_p,Q_{\pi,\nu})$ such that for all $p \in [\frac{k}{k-1},+\infty)$
    \begin{equation}\label{eq: non poincare manifold}
        \int_{\PP_p(\Ss)} \bigg| F(\mu) - \int_{\PP_p(\Ss)} F dQ_{\pi,\nu} \bigg|^{p'} dQ_{\pi,\nu}(\mu) > N \int_{\PP_p(\Ss)} |DF|^{p'}(\mu) dQ_{\pi,\nu}, \quad \text{ for all }N\in \N.
    \end{equation}
    In particular, the metric measure space $(\PP_p(\Ss),W_p,Q_{\pi,\nu})$ does not support a $p'$-Poincaré inequality.
\end{co}

\subsection{Bakry--\'Emery curvature condition}\label{subsec: BE}
In this subsection, we establish a Bakry--\'Emery curvature dimension condition for the space $(\PP(\Ss),W_2,Q_{\pi,\nu})$, under the following assumption.

\begin{ass}\label{assumption: manifold}
    $(\Ss,g)$ is a compact, oriented and boundaryless Riemannian manifold of dimension $k\geq 2$, with associated volume measure $\operatorname{vol}_{\mathtt{g}}$. Moreover, $\nu = \rho \operatorname{vol}_{\mathtt{g}} \in \PP(\Ss)$ with $\rho = e^{-V}$ with $V\in C^\infty(\Ss)$.
\end{ass}

The necessity of working with a compact Riemannian manifold comes from the characterization of the diffusion process, and consequently of its associated semigroup $(\mathtt{H}^{\pi,\nu}_t)_{t\geq 0}$, induced by the Cheeger energy in $(\PP(\Ss),W_2,Q_{\pi,\nu})$, which is one of the main results of \cite[Theorem 3.24]{delloschiavo2024massive}. Thus, that hypothesis can be relaxed as soon as such a characterization is recovered: it is still open for non-compact underlying spaces (see \cite[Example 3.40]{delloschiavo2024massive}). 

Next, we introduce the main notation that we need to describe the aforementioned characterization. 

\subsubsection{Rescaled and product Riemannian metric} We will use the following notation:
\begin{enumerate}
    \item $\nabla^{\mathtt{g}}$ the gradient operator of $(\Ss,\mathtt{g})$, which is characterized by requiring that for all $f\in C^{1}(\Ss)$, $y\in \Ss$ and $v\in T_y\Ss$, $d_yf(v) = \mathtt{g}_y(\nabla^g f(y), v)$;
    \item for all $f\in C^1(\Ss)$ and $y\in \Ss$, denote $\Gamma^{\mathtt{g}}f(y) := \mathtt{g}_y(\nabla^\mathtt{g} f(y), \nabla^\mathtt{g} f(y))$. It is actually well defined for all $f\in W^{1,2}(\Ss,\mathtt{g})$ and it is called the \textit{carr\'e du champ} associated to its Dirichlet form;
    \item for all $y\in \Ss$, we can uniquely associate to the measure $\nu$ a diffusion process $(Y_t^{y,\nu})_{t\geq 0}$ with $Y_0^{y,\nu} = y$. Formally, it is a solution of the stochastic differential equation $dY_t^{y,\nu} = -\nabla V dt + \sqrt{2}dW_t^{y}$, where $(W_t^y)_{t\geq 0}$ is the Wiener process of the manifold exiting from $y$;
    \item the associated linear and contractive heat semigroup is denoted, for all $t\geq 0$, $\mathtt{h}_t^{\mathtt{g},\nu}:L^p(\nu) \to L^p(\nu)$, for $p\in[1,+\infty)$;
    \item the Ricci tensor of $(\Ss,\ttg)$ is denoted as $\operatorname{Ric}^{\mathtt{g}}$. We say that it is bounded from below by $K \in \R$ if $\operatorname{Ric}^{\mathtt{g}} \geq K \mathtt{g}$;
    \item we say that the triple $(\Ss,d_\ttg, \nu)$ satisfies the \textit{Bakry--\'Emery curvature condition} \ref{eq: BE}, with $K\in \R$, if
    \begin{equation}
        \Gamma^{\mathtt{g}}(\mathtt{h}^{\mathtt{g},\nu}_t f) \leq e^{-2Kt} \mathtt{h}^{\mathtt{g},\nu}_t(\Gamma^\mathtt{g}f), \quad \forall f\in C^1(\Ss), \ \forall t\geq 0.
    \end{equation} 
    Recall that, in this smooth Riemannian setting, if $\nu= \operatorname{vol}_\ttg$, \ref{eq: BE} is satisfied if and only if $\operatorname{Ric}^\ttg \geq K \ttg$.
\end{enumerate}

For all $a>0$, we define the rescaled metric $\mathtt{g}^a := a\ttg$, for which the objects introduced above rescale as (we omit the term $\ttg$ for the sake of notation)
\begin{equation}\label{eq: rescaled quantites}
\begin{gathered}
    \nabla^{a} = \frac{1}{a}\nabla, \quad \Gamma^{a} = \frac{1}{a}\Gamma , \quad\operatorname{vol}_{a} = a^{\frac{k}{2}} \operatorname{vol},
    \\
    \tth_t^{a} = \tth_{t/a}, \quad Y_t^{y,\nu,a} = Y_{t/a}^{y,\nu} , \quad\operatorname{Ric}^{a} = \operatorname{Ric}.
\end{gathered}
\end{equation}

In particular, if $(\Ss,\ttg)$ has Ricci curvature bounded from below by $K$, then $(\Ss,\ttg^a)$ has Ricci curvature bounded from below by $K/a$.

We proceed looking at the curvature properties of the product between rescaled spaces. Fix $\mathbf{a} = (a_1,a_2,\dots) \in \bold{T}_0$ and, for all $n\geq 1$, on the manifold $\Ss^n$ we consider the Riemannian metric $\ttg^{n,\mathbf{a}} := \ttg^{a_1}\oplus \dots\oplus \ttg^{a_n}$, which raises the objects
\begin{equation}\label{eq: product}
\begin{gathered}
    \nabla^{n,\mathbf{a}}:= \nabla^{\ttg^{n,\mathbf{a}}} = \nabla^{a_1}\oplus \dots \oplus \nabla^{a_n}, \quad\Gamma^{n,\mathbf{a}}:= \Gamma^{\ttg^{n,\mathbf{a}}} = \Gamma^{a_1}\oplus \dots \oplus \Gamma^{a_n}, 
    \\
    \operatorname{vol}_{n,\mathbf{a}} := \operatorname{vol}_{\ttg^{n,\mathbf{a}}} = \operatorname{vol}_{a_1}\otimes \dots\otimes \operatorname{vol}_{a_n},\quad 
    Y_t^{\mathrm{y},\nu,n,\mathbf{a}}:= \big( Y_{t/a_1}^{y_1,\nu}, \dots, Y_{t/a_n}^{y_n,\nu} \big),
    \\
    \ttH^{n,\mathbf{a}}:= \tth^{\ttg^{n,\mathbf{a}}} = \big( \tth_{t/a_1}, \dots, \tth_{t/a_n} \big), 
    \quad
    \operatorname{Ric}_{n,\mathbf{a}} := \operatorname{Ric}_{\ttg^{n,\mathbf{a}}} = \operatorname{Ric}\oplus \dots\oplus \operatorname{Ric}.
\end{gathered}
\end{equation}
where $\mathrm{y} = (y_1,y_2,\dots)\in \Ss^{\infty}$. Notice that the diffusion process $Y_t^{\mathrm{y},\nu,n,\mathbf{a}}$ is well-posed (existence and uniqueness for any starting vector $\mathrm{y}$) thanks to \cite[Corollary 2.25]{delloschiavo2024massive}.

\begin{lemma}\label{lemma: curvature product}
    Assume that $\operatorname{Ric}^{\ttg} \geq K \ttg$. Then, the Riemannian manifold $(\Ss^{n,\mathbf{a}}, \ttg^{n,\mathbf{a}})$ has Ricci curvature bounded from below by
    \[K^{n,\mathbf{a}} := 
    \begin{cases}
    \frac{K}{a_1} & \text{ if }K\geq0; \\
    \frac{K}{a_n} & \text{ if }K<0.
    \end{cases}\]
    \\
    Moreover, if $\nu = e^{-V}\operatorname{vol}^\ttg \in \PP(\Ss)$ is as in Assumption \ref{assumption: manifold} and $V$ is $\lambda$-convex, then the metric measure space $(\Ss^n, \ttg^{n,\mathbf{a}}, \nu^{n})$ satisfies $\operatorname{BE}( K^{n,\mathbf{a},\lambda} ,\infty)$, with 
    \[K^{n,\mathbf{a},\lambda} := 
    \begin{cases}
    \frac{K+\lambda}{a_1} & \text{ if }K+\lambda\geq0; \\
    \frac{K+\lambda}{a_n} & \text{ if }K+\lambda<0.
    \end{cases}\]
\end{lemma}

\begin{proof}
    For all $y_1,\dots,y_n \in \Ss$ and $v_i \in T_{y_i}\Ss$, it holds
    \begin{equation}
    \operatorname{Ric}(v_1,v_1)+ \dots + \operatorname{Ric}(v_n,v_n) \geq \frac{K}{a_1} \ttg^{a_1}(v_1,v_1) + \dots + \frac{K}{a_n}\ttg^{a_n}(v_n,v_n), 
    \end{equation}
    which concludes the first statement, thanks to \eqref{eq: product}. Regarding the second statement, notice that 
    \[\nu^n = e^{-V\oplus \dots \oplus V} \operatorname{vol}_\ttg^n = e^{-\left(V\oplus \dots \oplus V + \frac{k}{2}\sum_{i=1}^n \log(a_i)\right)} \operatorname{vol}_{n,\mathbf{a}}.\]
    On the other hand, recall that $V$ is $\lambda$-convex if and only if the Hessian tensor $\operatorname{Hess}_\ttg V \geq \lambda \ttg$. The Hessian is invariant with respect to rescaling, so $V$ $\lambda$-convex implies that for all $i\leq n$, $\operatorname{Hess}_{\ttg^{a_i}} V = \operatorname{Hess}_{\ttg} V  \geq \frac{\lambda}{a_i}\ttg^{a_i}$. Putting everything together, we conclude thanks to the first part of the statement and \cite[Theorem 4.9]{sturm2006geometry}, since the $\operatorname{CD}(K,\infty)$ condition introduced there implies the $\operatorname{BE}(K,\infty)$ for infinitesimally Hilbertian metric measure spaces, see \cite{savare2013self}.
\end{proof}

\subsubsection{Characterization for the heat semigroup of $(\PP(\Ss),W_2, Q_{\pi,\nu})$}
    Under Assumption \ref{assumption: manifold}, further assume that $\pi \in \PP(\bold{T}_0)$, and define $Q_{\pi,\nu}$ as usual (see \eqref{subsec: reference}). Thanks to \cite[Theorem 3.39]{delloschiavo2024massive}, the quadratic pre-Dirichlet form $\big( \widehat{\mathcal{E}}, \operatorname{GC}_c^\infty(\PP(\Ss)) \big)$ defined as 
    \begin{equation}
        \widehat{\mathcal{E}}(\hat{F}) := \int_{\PP(\Ss)} \int_{\Ss} |\nabla_W\hat{F}(y,\mu)|^2
        _\ttg d\mu(y) dQ_{\pi,\nu}(\mu) \quad \forall \hat{F} \in \operatorname{GC}_c^{\infty},
    \end{equation}
    is closable. Moreover, by \cite[Proposition 3.31]{delloschiavo2024massive} and \cite[Theorem 6.2]{fornasier2023density}, it coincides with the Cheeger energy, in the sense that $\mathcal{D}(\widehat{\mathcal{E}}) = W^{1,2}(\PP(\Ss),W_2,Q_{\pi,\nu})$ and $\widehat{\mathcal{E}}(F) = \operatorname{Ch}(F)$ for all $F\in W^{1,2}(\PP(\Ss),W_2,Q_{\pi,\nu})$. In particular, for all $\widehat{F}\in \operatorname{GC}_c^1(\PP(\Ss))$, it holds
    \begin{equation}\label{eq: sob = dir form}
        \int_{\Ss} |\nabla_W\hat{F}(y,\mu)|^2
        _\ttg d\mu(y) = |D\widehat{F}|^2(\mu) \quad \text{ for }Q_{\pi,\nu}\text{-a.e. }\mu.
    \end{equation}
    Note that the equivalence with the Cheeger energy can also be recovered applying \cite[Theorem 6.2]{fornasier2023density}, Proposition \ref{prop: manifold capacity} and Proposition \ref{prop: approx GC with cyl}, after using the Nash isometry to embed $\Ss$ in some $\R^d$.

    The equivalence between the Cheeger energy and the Dirichlet form $\widehat{\mathcal{E}}$ allows us to exploit \cite[Theorems 3.15 \& 3.24]{delloschiavo2024massive}, which in particular gives the following characterization of the associated diffusion process: again by \cite[Corollary 2.25]{delloschiavo2024massive}, for all $\mathbf{a} = (a_1,a_2, \dots ) \in \bold{T}_0$ and $\mathrm{y} = (y_1,y_2,\dots) \in \Ss^\infty_{\neq}$, we can find a countable number of independent processes $(Y^{y_i,\nu}_t)_{t\geq 0}$, all defined on a fixed probability space $(\Omega,\mathcal{F},\mathbb{P})$. Then, for all $\mu\in \PPpa(\Ss)$, $\mu = \operatorname{em}(\mathbf{a}, \mathrm{y})$, the diffusion process associated to $(\PP(\Ss),W_2,Q_{\pi,\nu})$ exiting from $\mu$ is given by
    \begin{equation}
        \mu_t := \operatorname{em}\big(\mathbf{a}, \mathrm{Y}^{\mathbf{a},\mathrm{y},\nu}_t\big) = \sum_{i=1}^{+\infty} a_i \delta_{Y_{t/a_i}^{y_i, \nu}}, \quad \mathrm{Y}^{\mathbf{a},\mathrm{y},\nu}_t := \big(Y_{t/a_i}^{y_i, \nu}\big)_{i\geq 1}.
    \end{equation}
    In other words, the masses are kept fixed, while the positions move according to the underlying geometry, and the lighter the particles are (smaller weight $a_i$) the faster they move. We can use this to gain information on the heat semigroup $(H^{\pi,\nu}_t)_{t\geq 0}$: for our purposes, it suffices to see how it acts on generalized cylinder functions. First, we introduce the subsets of generalized cylinder function given $a\in(0,1)$
    \[\operatorname{GC}_{a}^\infty(\PP(\Ss)) := \left\{\Psi\circ \hat{L}_{\hat{\Phi}} \ : \ k\geq1, \ \Psi \in C^\infty(\R^k), \ \hat{\Phi} \in C^\infty_c(\Ss\times(a,1];\R^k) \right\}.\]
    Notice that $\operatorname{GC}_c^\infty(\PP(\Ss)) = \cup_{n\geq 2} \operatorname{GC}_{1/n}^\infty(\PP(\Ss))$.
    Let $n\geq 2$, $\hat{F} = \Psi(\hat{L}_{\hat{\phi}_1},\dots, \hat{L}_{\hat{\phi}_k}) \in \operatorname{GC}_{1/n}^\infty (\PP(\Ss))$,
    and $\mu = \operatorname{em}(\mathbf{a},\mathrm{y})$, with $\mathbf{a} \in \bold{T}_0$ and $\mathrm{y} \in \Ss^\infty_{\neq}$. Then define 
    \[\hat{F}^{n,\mathbf{a}}(y_1,\dots,y_n) = \hat{F} \big(\operatorname{em}\big(\mathbf{a}, \mathrm{y}\big)\big),\]
    which is a function that truly depends just on the first $n$-coordinates of $\mathrm{y}$ (possibly even less coordinates may contribute) whatever $\mathbf{a}\in \bold{T}_0$, and it is smooth on $\Ss^n$, so that
    \begin{equation}\label{eq: heat semigroup on GC}
    \begin{aligned}
        \ttH^{\pi,\nu}_t \hat{F} (\mu) = &\mathbb{E}\left[\hat{F}\big(\operatorname{em}\big(\mathbf{a}, \mathrm{Y}^{\mathbf{a},\mathrm{y},\nu}_t\big)\big) \right] = \mathbb{E}\left[\hat{F}^{n,\mathbf{a}} \big(Y_{t/a_1}^{y_1, \nu},\dots,Y_{t/a_n}^{y_n, \nu}\big) \right] 
        \\
        = & \ttH^{n,\mathbf{a},\nu}_t \hat{F}^{n,\mathbf{a}}(y_1,\dots,y_n).
    \end{aligned}
    \end{equation}
    Moreover, 
    \begin{equation}\label{eq: carre du champ transfer}
        \begin{aligned}
            |D\hat{F}|^2(\mu) = & \int_{\Ss} |\nabla_W \hat{F}(y,\mu)|_\ttg^2 d\mu(y) = \sum_{i=1}^n a_i  |\nabla_W \hat{F}(y_i,\mu)|_\ttg^2 
            \\
            = & 
            \sum_{i=1}^n a_i  \left|\frac{1}{a_i}\nabla_{y_i}^{\ttg}\hat{F}^{n,\mathbf{a}}(y_1,\dots,y_n)\right|_\ttg^2 = 
            \sum_{i=1}^n \frac{1}{a_i} \Gamma_{y_i}^{\ttg}\hat{F}^{n,\mathbf{a}}(y_1,\dots,y_n) 
            \\
            = & 
            \sum_{i=1}^n \Gamma_{y_i}^{a_i}\hat{F}^{n,\mathbf{a}}(y_1,\dots,y_n) = \Gamma^{n,\mathbf{a}} \hat{F}^{n,\mathbf{a}}(y_1,\dots,y_n).
        \end{aligned}
    \end{equation}
More generally, the intertwining property shown in \cite[Theorem 3.15]{delloschiavo2024massive}, gives that 
\begin{equation}\label{eq: intertwining}
    |D(\ttH^{\pi,\nu}_t \hat{F})|^2(\mu) = \Gamma^{n,\mathbf{a}}\big(\ttH^{n,\mathbf{a},\nu}_t \hat{F}^{n,\mathbf{a}}\big)(y_1,\dots,y_n)
\end{equation}

We can finally state and prove the Bakry--\'Emery condition for the metric measure space $(\PP(\Ss),W_2,Q_{\pi,\nu})$.

\begin{teorema}\label{thm: BE wass}
    Let $(\Ss,\ttg)$ be a compact, oriented and boundaryless Riemannian manifold of dimension $k\geq 2$ and with Ricci curvature bounded from below by $K\in \R$. Let $\pi\in \PP(\bold{T}_0)$ and $\nu = e^{-V}\operatorname{vol}_\ttg\in \PP(\Ss)$ as in Assumption \ref{assumption: manifold}.
    \\
    Assume that $V$ is $\lambda$-convex, with $K+\lambda\geq 0$. Then $(\PP(\Ss),W_2,Q_{\pi,\nu})$ satisfies $\operatorname{BE}(K+\lambda,\infty)$.
\end{teorema}

\begin{proof}
    Thanks to Lemma \ref{lemma: BE dense in energy}, it suffices to prove the Bakry--\'Emery condition for all $\hat{F}\in \operatorname{GC}^{\infty}_c(\PP(\Ss))$. Indeed, the algebra of classic cylinder functions $\operatorname{Cyl}^\infty(\PP(\Ss))$ is dense in energy thanks to \cite[Section 6.2]{fornasier2023density} (see also \cite[Corollary 4.20]{sodini2023general}), and for all cylinder functions $F(\cdot) = \Psi (L_{\phi_1}(\cdot),\dots, L_{\phi_n}(\cdot) )$, it is approximated in the norm of $W^{1,2}(\PP(\Ss),W_2,Q_{\pi,\nu})$ by the generalized cylinder functions $\hat{F}_n \in \operatorname{GC}_c^\infty(\PP(\Ss))$ defined as 
    \[\hat{F}_n(\mu) := \Psi\left( \hat{L}_{\hat{\phi}_1}(\mu),\dots,\hat{L}_{\hat{\phi}_n}(\mu) \right), \quad \hat{\phi}_i(y,r) := \phi_i(y)\eta_n(r),\]
    where $\eta_n\in C_c^\infty((0,1])$ is such that $\eta_n \equiv 1$ in $[1/n,1]$. Then, fix $n\geq 2$ and $\hat{F}\in \operatorname{GC}_{1/n}^{\infty}(\PP(\Ss))$. Thanks to \eqref{eq: intertwining}, applying Lemma \ref{lemma: curvature product}, we have 
    \begin{align*}
         |D(\ttH^{\pi,\nu}_t \hat{F})|^2(\mu) = \Gamma^{n,\mathbf{a}}\big(\ttH^{n,\mathbf{a},\nu}_t \hat{F}^{n,\mathbf{a}}\big)(y_1,\dots,y_n) \leq e^{-2\frac{(K+\lambda)}{a_1}t} \ttH^{n,\mathbf{a},\nu}\big( \Gamma^{n,\mathbf{a}}\hat{F}^{n,\mathbf{a}} \big)(y_1,\dots,y_n).
    \end{align*}
    Now, thanks to \eqref{eq: carre du champ transfer}, $\Gamma^{n,\mathbf{a}}\hat{F}^{n,\mathbf{a}}(y_1,\dots,y_n) = |D\hat{F}|^2(\mu)$, and the latter is again a generalized cylinder function in $\operatorname{GC}_c^\infty(\PP(\Ss))$ satisfying $\Gamma^{n,\mathbf{a}}\hat{F}^{n,\mathbf{a}}(y_1,\dots,y_n) = |D\hat{F}|^2\big(\operatorname{em}(\mathbf{a},\mathrm{y})\big)$, so that \eqref{eq: heat semigroup on GC} gives
    \[\ttH^{n,\mathbf{a},\nu}\big( \Gamma^{n,\mathbf{a}}\hat{F}^{n,\mathbf{a}} \big)(y_1,\dots,y_n) = \ttH^{\pi,\nu}|D\hat{F}|^2(\mu).\]
    We conclude observing that $\frac{(K+\lambda)}{a_1} \geq K+\lambda$ for all $\mathbf{a} = (a_1,a_2,\dots) \in \bold{T}_0$.
\end{proof}

\normalcolor

\printbibliography

\appendix

\section{On the definition of capacity}\label{app: capacity}
\begin{lemma}\label{lemma: capacity appendix}
    Assume that $\nu \ll \mathcal{L}^d$. Then for all $p\geq 1$
    \[\operatorname{cap}_{p,\nu}(\Delta) = \widetilde{\operatorname{cap}}_{p,\nu}(\Delta).\]
\end{lemma}

\begin{proof}
Clearly, $\operatorname{cap}_{p,\nu}(\Delta) \geq \widetilde{\operatorname{cap}}_{p,\nu}(\Delta)$, so we focus on the other inequality.
    Let us call $\mathcal{C}(h) := \int |h^p| + |\nabla h|^p d\nu\otimes \nu$. We are going to prove that for any $h\in C_b^1(\R^{2d})$ satisfying $h\geq 1$ on $\Delta$ and for any $\delta>0$, there exists a function $g\in C_b^1(\R^{2d})$ such that $g\leq 1$ everywhere, $g=1$ on $\Delta$ and 
    \begin{equation}\label{eq: proof of cap lemma}
       \mathcal{C}(g) \leq \mathcal{C}(h) + \delta.
    \end{equation}

    \textbf{Step 1}: define $\tilde{h}(x,y) := h(x,y)\wedge 1$. Notice that $\tilde{h}$ is a Lipschitz function, and satisfies $\tilde{h}\leq 1$, $\tilde{h} = 1$ on $\Delta$ and $\mathcal{C}(\tilde{h})\leq \mathcal{C}(h)$, thanks to the locality of the gradient for Lipschitz functions, i.e. 
    \begin{equation}\label{eq: gradient of truncated h}
    \nabla \tilde{h}(x,y) = \nabla h(x,y) \mathds{1}_{\{h\leq 1\}}(x,y) \quad \mathcal{L}^{2d}\text{-a.e.},
    \end{equation}
    and the fact that $\nu \ll \mathcal{L}^d$.

    \textbf{Step 2}: we claim that there exists a Lipschitz function $\overline{h}$ such that $\overline{h}\leq 1$, $\mathcal{C}(\overline{h})\leq \mathcal{C}(\tilde{h}) + \frac{\delta}{2}$ and $\overline{h} = 1$ on $\Delta_{\varepsilon}:= \Delta+B(0,\varepsilon)$, for some $\varepsilon>0$, i.e. there exists a `strip' of length $\varepsilon$ around the diagonal $\Delta$ where the function $\overline{h}$ is constantly $1$.
    \\
    To this aim, define $\pi_\Delta$ and $\pi_{\Delta^\perp}$, respectively, the orthogonal projection over $\Delta$ and $\Delta^\perp$, i.e. 
    \begin{align*}
        \pi_\Delta(x,y) = \left(\frac{x+y}{2}, \frac{x+y}{2}\right), \quad \pi_{\Delta^\perp}(x,y) = \left( \frac{x-y}{2}, - \frac{x-y}{2} \right), \\
        (x,y) = \pi_\Delta(x,y) + \pi_{\Delta^\perp}(x,y), \quad \|(x,y)\|^2 = \|\pi_\Delta(x,y)\|^2 + \|\pi_{\Delta^\perp}(x,y)\|^2.
    \end{align*}
    For any $\varepsilon>0$, define $p_\varepsilon:\R^{2d}\to \R^{2d}$ as follows 
    \begin{equation}
        p_\varepsilon(x,y) := \begin{cases}
            \pi_{\Delta}(x,y) \quad & \text{ if } (x,y) \in \Delta_\varepsilon
            \\
            (x,y) - \varepsilon \frac{\pi_{\Delta^\perp}(x,y)}{\|\pi_{\Delta^\perp}(x,y)\|} & \text{ otherwise}.
        \end{cases}
    \end{equation}
    In the next step, we will prove that $p_\varepsilon$ is $1$-Lipschitz for any $\varepsilon>0$. For now, assume that is true and we continue defining $h_\varepsilon(x,y) := \tilde{h}(p_\varepsilon(x,y))$. 
    We conclude the proof of this step if 
    \[\limsup_{\varepsilon\searrow 0} \mathcal{C}(h_\varepsilon)\leq \mathcal{C}(\tilde{h}).\]
    To this aim, notice that $h_\varepsilon \to \tilde{h}$ pointwise as $\varepsilon \searrow 0$, so that 
    \[\int |h_\varepsilon|^p d\nu\otimes \nu \to \int|\tilde{h}|^p d\nu\otimes \nu\]
    by dominated convergence theorem. Regarding the gradients, notice that $|\nabla \tilde{h}(x,y)|=  |\nabla h(x,y)| \mathds{1}_{\{ h\leq 1\} }(x,y)$ almost everywhere, and the right hand side is an upper semicontinuous function. Then, thanks to the straightforward fact that $(p_\varepsilon)_\#(\nu\otimes \nu) \to \nu\otimes \nu$ in the narrow convergence and 
    \[|\nabla h_\varepsilon(x,y)| \leq |\nabla \tilde{h}(p_\varepsilon(x,y))| = |\nabla h(p_\varepsilon(x,y))| \mathds{1}_{\{h\leq 1\}}(p_\varepsilon(x,y)) \quad \mathcal{L}^{2d}\text{-a.e.},\] 
    we have that 
    \begin{align*}
        \limsup_{\varepsilon\searrow 0} &  \int |\nabla h_\varepsilon|^p(x,y) d\nu\otimes \nu(x,y) \leq \limsup_{\varepsilon\searrow 0} \int |\nabla h(x,y)|^p\mathds{1}_{ \{h\leq 1\} } (x,y ) d\big((p_\varepsilon)_\#(\nu \otimes \nu)\big)(x,y) 
        \\
        \leq & 
        \int |\nabla h(x,y)|^p\mathds{1}_{\{h\leq 1\}}(x,y) d\nu\otimes \nu(x,y) = \int |\nabla \tilde{h}(x,y)|^p d\nu\otimes \nu(x,y).
    \end{align*}

    \textbf{Step 3}: proof that $p_\varepsilon$ is $1$-Lipschitz for all $\varepsilon>0$. We need to verify that $\|p_\varepsilon(x,y) - p_\varepsilon(z,w)\| \leq \|(x,y) - (z,w)\|$ for all $(x,y),(z,w)\in \R^d\times \R^d$, where $\|\cdot\|$ is the Euclidean norm of $\R^{2d}$. There are three cases:
    \begin{itemize}
        \item $(x,y),(z,w) \in \Delta_\varepsilon$, then it is straightforward because any orthogonal projection is $1$-Lipschitz;
        \item $(x,y)\in \Delta_\varepsilon$ and $(z,w)\notin \Delta_\varepsilon$, then notice that $\|\pi_{\Delta^\perp}(x,y)\|<\varepsilon$ and $\|\pi_{\Delta^\perp}(z,w)\|\geq \varepsilon$. Then
        \begin{align*}
            \|p_\varepsilon(x,y) - p_\varepsilon(z,w)\|^2 
             = &
            \| \pi_\Delta(x,y) - \pi_\Delta(z,w) -\left(1- \frac{\varepsilon}{\|\pi_{\Delta^\perp}(z,w)\|} \right)\pi_{\Delta^\perp}(z,w) \|^2
            \\
             = &
            \|\pi_\Delta(x,y) - \pi_\Delta(z,w)\|^2 + \left(\|\pi_{\Delta^\perp}(z,w)\| - \varepsilon\right)^2
            \\
            (\star) \leq &
            \|\pi_\Delta(x,y) - \pi_\Delta(z,w)\|^2 + \|\pi_{\Delta^\perp}(x,y) - \pi_{\Delta^\perp}(z,w)\|^2 
            \\
            = & \|(x,y) - (z,w)\|^2,
        \end{align*}
        where in $(\star)$ we used that $\|\pi_{\Delta^\perp}(z,w)\| - \varepsilon\leq \|\pi_{\Delta^\perp}(z,w)\| - \|\pi_{\Delta^\perp}(x,y)\|\leq \|\pi_{\Delta^\perp}(x,y) - \pi_{\Delta^\perp}(z,w)\|$;
        \item $(x,y),(z,w) \notin \Delta_\varepsilon$. Then 
        \begin{equation}\label{eq: outside strip}
        \begin{aligned}
            \|p_\varepsilon(x,y) & - p_\varepsilon(z,w)\|^2 
            = 
            \|\pi_{\Delta}(x,y) - \pi_{\Delta}(z,w)\|^2
            \\
            & + \left\|\left(1 - \frac{\varepsilon}{\|\pi_{\Delta^\perp}(x,y) \|}\right)\pi_{\Delta^\perp}(x,y) - \left(1 - \frac{\varepsilon}{\|\pi_{\Delta^\perp}(z,w) \|}\right)\pi_{\Delta^\perp}(z,w)\right\|^2.
        \end{aligned}
        \end{equation}
        For simplicity, call $v_1 =\pi_{\Delta^\perp}(x,y)$ and $v_2 =\pi_{\Delta^\perp}(z,w)$, together with $\alpha_i := 1- \varepsilon/\|v_i\|$, for $i=1,2$. Then
        \begin{align*}
            \|\alpha_1 v_1 - \alpha_2 v_2\| = & 
            \|(\alpha_1-\alpha_2)v_1 + \alpha_2 (v_1 - v_2)\|
            \leq 
            |\alpha_1 - \alpha_2|\|v_1\| + \alpha_2 \|v_1 - v_2\|
        \end{align*}
        and 
        \begin{align*}
            |\alpha_1 - \alpha_2|\|v_1\| = \varepsilon\left|\frac{1}{\|v_2\|} - \frac{1}{\|v_1\|}\right|\|v_1\|\leq\varepsilon \frac{\|v_1-v_2\|}{\|v_2\|},
        \end{align*}
        so that 
        \[\|\alpha_1 v_1 - \alpha_2 v_2\| \leq \frac{\varepsilon}{\|v_2\|}\|v_1 - v_2\| + \|v_1-v_2\| - \frac{\varepsilon}{\|v_2\|}\|v_1-v_2\| = \|v_1-v_2\|.\]
        Substituting in \eqref{eq: outside strip} we conclude that $p_\varepsilon$ is $1$-Lipschitz.
    \end{itemize}

    \textbf{Step 4}: convolution and conclusion. Consider $\overline{h}$ as in step 2. Then, for any $0 <\alpha < \varepsilon$, define the functions $g_\alpha:= \overline{h} * \rho_\alpha$, where $\rho_\alpha(x,y) := \frac{1}{\alpha^{2d}}\rho\left(\frac{(x,y)}{\alpha}\right)$ and $\rho$ is a mollifier with support in $B(0,1)$. Notice that each $g_\alpha\in C_b^1(\R^{2d})$ satisfies $h\leq 1$ and $h = 1$ on $\Delta$. We conclude observing that, up to selecting a suitable subsequence $\alpha_n \to 0$, $|g_\alpha|^p + |\nabla g_\alpha|^p \to |\overline{h}|^p + |\nabla \overline{h}|^p$ pointwise almost everywhere, and then by dominated convergence theorem we have that 
    \[\lim_{\alpha \to 0}\mathcal{C}(g_\alpha) = \mathcal{C}(\overline{h}). \qedhere \]
\end{proof}

{\small
		
		\vspace{15pt} (Alessandro Pinzi) Universit\`{a} Commerciale Luigi Bocconi, Dipartimento di Scienze delle Decisioni, \par
		\textsc{via Roentgen 1, 20136 Milano, Italy}
		\par
		\textit{e-mail address}: \textsf{alessandro.pinzi@phd.unibocconi.it}
		\par
		\textit{Orcid}: \textsf{https://orcid.org/0009-0007-9146-5434}
		\par

	}

\end{document}